\documentclass[11pt,a4paper]{article}
\usepackage[margin=26mm]{geometry}
\usepackage[T1]{fontenc}
\usepackage{amsmath,amssymb,amsthm,mathtools}
\usepackage{microtype,enumitem,booktabs,array}
\usepackage[hidelinks]{hyperref}
\hypersetup{pdftitle={Integrable distortion I: Weil-Petersson homeomorphisms and the Schoen conjecture},pdfauthor={Gaven Martin and Cong Yao}}
\newcommand{\D}{\mathbb D}
\newcommand{\ID}{\mathbb D}
\newcommand{\Sone}{{\mathbb S^1}}
\newcommand{\IS}{{\mathbb S^1}}
\newcommand{\IR}{\mathbb R}
\newcommand{\IC}{\mathbb C}
\newcommand{\K}{\mathbb K}
\newcommand{\IK}{\mathbb K}
\newcommand{\IH}{\mathbb H}
\newcommand{\I}{\mathcal I}
\newcommand{\E}{\mathcal E}

\newcommand{\area}{\,dA}

\newcommand{\bw}{\bar w}
\DeclareMathOperator{\loc}{loc}
\DeclareMathOperator{\id}{id}
\newtheorem{mainthm}{Theorem}

\newtheorem{theorem}{Theorem}[section]
\newtheorem{proposition}[theorem]{Proposition}
\newtheorem{lemma}[theorem]{Lemma}
\newtheorem{corollary}[theorem]{Corollary}
\theoremstyle{definition}

\newtheorem{example}[theorem]{Example}
\newtheorem{remark}[theorem]{Remark}
\numberwithin{equation}{section}
\setlist{itemsep=4pt,topsep=5pt}
\allowdisplaybreaks[1]
\title{Integrable distortion I:\\ Weil--Petersson circle homeomorphisms and the Schoen conjecture\thanks{{\tiny
Work of both authors partially supported by the New Zealand Marsden Fund. C. Yao is partially supported by the Young Scientist Program of the Ministry of Science and Technology of China (No. 2021YFA1002200), the National Natural Science Foundation of China (No. 12401096, No. 12101362), the Natural Science Foundation of Shandong Province (No. ZR2024QA035, No. ZR2022YQ01), the Fundamental Research Funds for the Central Universities.
\newline
{\bf Keywords.} Teichm\"uller theory, quasiconformal mappings, finite distortion, extremal mappings, harmonic mappings, calculus of variations
\newline
{\bf MSC Subject.} 30C62 31A05 49J10 58E20 } }}
\author{Gaven Martin \and Cong Yao}
\date{}
\begin{document}
\maketitle
\begin{abstract}
A homeomorphism of the circle $\beta:\Sone\to\Sone$ is Weil--Petersson if it has a quasiconformal extension to the disk whose Beltrami coefficient is square integrable for the hyperbolic area. We study the variational problem for the hyperbolic anti-holomorphic energy 
\[ \E_2(h)=\int_\D\eta(h)|h_{\bar w}|^2\,dA, \quad\quad  \mbox{ equivalently} \quad \quad \int_\D|\mu_f|^2(1-|\mu_f|^2)^{-1}\,dA_\eta,\quad f=h^{-1} \] over the full class of monotone Sobolev maps with the prescribed boundary values $\beta$. 

For Weil--Petersson data this energy has a unique minimiser, which is a real analytic hyperbolic harmonic diffeomorphism, globally quasiconformal and asymptotically conformal, with an explicit dilatation bound in terms of the energy of any admissible barrier. We prove that a circle homeomorphism is Weil--Petersson if and only if it admits a finite-distortion homeomorphic extension with hyperbolically integrable distortion excess 
\[ \int_\D(\K_F-1)\,dA_\eta<\infty \]
without a  quasiconformal hypothesis. More generally, for every $p>1$ the boundary values of a finite-distortion homeomorphism with 
\[ \int_\D|\mu_F|^p(1-|\mu_F|^2)^{-p_*/2}\,dA_\eta<\infty\]
 lie in Guo's $p$-integrable Teichm\"uller space $T^p$. Our proof is a capacity estimate for boundary moduli combined with two exact identities for the quasisymmetric ratio, and yields a derivative-free characterisation of the spaces  $T^p$ which at $p=2$ is essentially a theorem of Bishop. The minimum defines a M\"obius-invariant boundary energy, quantitatively comparable to the weighted square norm of the Hopf differential of the minimiser. 
 
 For arbitrary quasisymmetric data we give a proof of Markovi\'c's theorem (the Schoen conjecture) through approximation by our smooth quadratic minimisers using Mori's theorem and the mean-value property of a logarithm to verify the averaged-distance hypothesis in the interior comparison of Riestenberg--Smillie, and the resulting displacement, energy-density and dilatation bounds are explicit in the quasisymmetry constant. The nonquadratic ($p\neq 2$) theory is developed in Part~II.
\end{abstract}

\clearpage
\tableofcontents

\bigskip

\noindent\textbf{AI disclosure.}
This work was substantially developed before the widespread use of
generative AI. Indeed it followed the insistence of Professor M.~Brakalova that we look at this subject after a talk we gave at the third Quasiworld conference. We gratefully thank her for that suggestion. However in the final writing we used both Claude and GPT for proofreading, checking references
and revising the exposition. During this process, when prompted by ``what have we missed'' they both suggested looking at the work of Riestenberg--Smillie, and also of Ripoll--Tomi. This overcame some gaps in our argument for the Schoen conjecture in particular, and the estimates so obtained helped everywhere.

\clearpage
\section{Introduction}
The hyperbolic disk has infinite area, so for a map of the disk onto itself the natural finite variational quantity is the distortion in excess of conformality rather than the full Dirichlet energy. With the hyperbolic density
\[
 \eta(z)=\frac{1}{(1-|z|^2)^{2}},\qquad dA_\eta=\eta\,dA\qquad\mbox{(curvature $-4$)},
\]
and $\mu_f=f_{\bar z}/f_z$,  the {\em Beltrami coefficient} of $f$, the functionals we consider are
\begin{equation}\label{intro-functionals}
 \I_p(f)=\int_\D\frac{|\mu_f|^p}{(1-|\mu_f|^2)^{p_*/2}}\,dA_\eta, \qquad p_*=\max\{2,p\}.
\end{equation}
Following our earlier work \cite{MYRadonRiesz,MYLp} we work mostly with the ``inverse functional''. At $p=2$, if $h=f^{-1}$ has the required Sobolev regularity, then
\begin{equation}\label{intro-E2}
 \I_2(f)=\E_2(h)=\int_\D\eta(h)|h_{\bar w}|^2\,dA .
\end{equation}
This is half the pointwise excess of the hyperbolic Dirichlet-energy density of $h$ over its pullback area density. The minimisers and stationary points are hyperbolic harmonic maps,  so this ties our problem to the Schoen conjecture. 

This paper develops the quadratic theory ($p=2$) completely. The nonquadratic exponents $p\ne2$ have very different structure (convexity fails below a sharp threshold $p_{\mathrm c}=1+1/\sqrt2$, and the natural minimisers are no longer harmonic). These are treated in Part~II \cite{PartII}. The compactness results of Section~\ref{no-hairs} and the boundary results of Section~\ref{boundary-classes} are stated for general exponents because Part~II uses them in that form. Section~\ref{p-classes} is the exception to this rule. We prove, for every $p>1$, that the boundary values of a finite-energy finite-distortion map are $p$-integrable Teichm\"uller. This is a boundary statement, proved by capacity, and it uses none of the variational theory developed here.

\subsection{Main results}
Throughout, a circle homeomorphism $\beta$ is \emph{Weil--Petersson} if $\beta^{-1}$, equivalently $\beta$, has a quasiconformal extension $F$ to $\D$ with $\int_\D|\mu_F|^2\,dA_\eta<\infty$; this is the class $T^2$ of \cite{TT,Guo,Cui}. The minimisation takes place over the monotone Sobolev class $\mathcal M_\beta$ of Section~\ref{monotone-class}: continuous monotone surjections $h:\overline\D\to\overline\D$ with trace $\beta$, $h\in W^{1,2}(\D)$, $J_h\geq0$, and finite energy. No injectivity, finite distortion or quasiconformality is assumed of the competitors.

\begin{mainthm}[The quadratic problem for Weil--Petersson data]\label{thm-A}
Let $\beta$ be an orientation-preserving Weil--Petersson circle homeomorphism. Then $\E_2$ has a unique minimiser $h_\beta$ in $\mathcal M_\beta$. This map is a real analytic hyperbolic harmonic diffeomorphism of $\D$ onto itself, which is globally quasiconformal, and is asymptotically conformal in the sense that $|\mu_{h_\beta}(w)|\to0$ as $|w|\to1$. The dilatation of $h_\beta$ is controlled by the energy of any barrier. If $F_0$ is a finite-energy quasiconformal extension of $\beta^{-1}$ and $E_0=\I_2(F_0)$, then with
\[
B_0=\frac{64}{9\pi}(\pi+2E_0) e^{8\pi(\pi+2E_0)},
\]
we have 
\begin{equation}\label{intro-qc-bound}
 1-\|\mu_{h_\beta}\|_\infty^2\geq\tfrac12e^{-\frac8\pi B_0E_0},\qquad
 K(h_\beta)\leq8 e^{ \frac8\pi B_0E_0 }.
\end{equation}
\end{mainthm}

Existence of a quasiconformal harmonic extension for smooth data is due to Akutagawa \cite{Akutagawa} for $C^4$ data and to Li--Tam \cite{LiTamProper,LiTamProperII} for $C^1$ diffeomorphisms, and the harmonic characterisation of the Weil--Petersson class has recently been obtained by Diaf--Markovi\'c--Mesbah \cite{DMM}. Their Theorem~C minimises in a class of smooth quasiconformal extensions with additional hyperbolic derivative control. The novelty of Theorem~\ref{thm-A} is the variational formulation over the full monotone Sobolev class, the uniqueness of the minimiser in that class, the regularity proof by local harmonic replacement, and the quantitative bound \eqref{intro-qc-bound}, whose only input is the barrier energy.

\begin{mainthm}[The Schoen conjecture]\label{thm-B}
Every orientation-preserving quasisymmetric circle homeomorphism $\beta$ has a unique quasiconformal hyperbolic harmonic extension $h:\D\to\D$; it is a real analytic diffeomorphism. If $g:\D\to\D$ is any $K$-quasiconformal homeomorphism with the same boundary values, then
\[
 \sup_{z\in\D}d_{\mathbb H}(h(z),g(z))\leq C_{\mathrm{disp}}(K),\qquad K(h)\leq Q_K,
\]
where $C_{\mathrm{disp}}(K)$ and $Q_K$ are the explicit constants of \eqref{schoen-displacement-constants} and \eqref{schoen-qc-constants}, depending only on $K$. The harmonic extension is the limit, uniformly on $\overline\D$ and smoothly on compact subsets, of the minimisers of Theorem~\ref{thm-A} for smooth approximations of $\beta$ with a common quasisymmetry constant.
\end{mainthm}

Existence here is of course Markovi\'c's theorem \cite{Markovic} and uniqueness is Li--Tam's \cite{LiTamProperII}. The displacement estimate is the heart of Markovi\'c's proof and we give a planar proof of it using the interior comparison of Riestenberg--Smillie \cite{RiestenbergSmillie}, whose averaged-distance hypothesis we verify in the disk from Mori's theorem and the mean-value property of a logarithm. The passage from a displacement bound to a dilatation bound is an explicit Bochner argument. All constants are effective, though presumably far from sharp.  However,  we show that as the quasisymmetry constant $M\to 1$ we have $K(h)\to 1$. See Remark~\ref{schoen-constants-scope}.

\begin{mainthm}[Weil--Petersson homeomorphisms and integrable distortion]\label{thm-C}
A circle homeomorphism $g$ is Weil--Petersson if and only if it has a finite-distortion homeomorphic extension $F:\D\to\D$ with
\[
 \int_\D(\K_F-1)\,dA_\eta<\infty,\quad\quad \K_F=\frac{1+|\mu_F|^2}{1-|\mu_F|^2}=\frac{1}{2}\Big(K+\frac{1}{K}\Big).
\]
Thus the quasiconformal hypothesis in the definition of $T^2$ is redundant and $\mathcal T_2=T^2$ in the notation of Section~\ref{boundary-classes}. More generally, for every $p\geq1$, a finite-distortion homeomorphism $F$ with $\I_p(F)<\infty$ extends homeomorphically to $\overline\D$, its boundary values are quasisymmetric and in fact symmetric, with modulus bounds depending only on $p$ and $\I_p(F)$, and $F$ factors as $A^{-1}\circ H$ with $A$ quasiconformal and $H$ having $C^1$ diffeomorphic boundary values.
\end{mainthm}

The finite-distortion characterisation complements the analytic and geometric descriptions of Weil--Petersson curves in \cite{TT,Shen,Bish1,Bish2}. Its proof applies Theorem~\ref{thm-A} to the inverse of the finite-distortion extension: the quasisymmetry is extracted from the energy by a capacity argument, and the variational theorem then produces a quasiconformal representative. The analogous identity for $p\ne2$ is Theorem~\ref{thm-E} below, proved by a different route.

\begin{mainthm}[A canonical energy]\label{thm-D}
For a Weil--Petersson homeomorphism $\beta$ put
\[
 \mathcal S(\beta)=\E_2(h_\beta)=\min_{h\in\mathcal M_\beta}\E_2(h)=\min_{F|_{\Sone}=\beta^{-1}}\I_2(F),
\]
where the last minimum ranges over finite-distortion homeomorphic extensions with finite energy and is attained only at $h_\beta^{-1}$. Then $\mathcal S$ is invariant under pre- and postcomposition by disk automorphisms, is nonnegative, and vanishes exactly for M\"obius boundary maps. If 
\begin{equation}\label{HD} \Phi_\beta=\eta(h_\beta)(h_\beta)_w\overline{(h_\beta)_{\bar w}}\end{equation}
 is the Hopf differential of $h_\beta$ and 
 \[ B_\beta=\sup_\D\eta(h_\beta)|(h_\beta)_w|^2/\eta,\] then
\[
 \frac{\|\Phi_\beta\|_{A^2_\eta}^2}{B_\beta}\leq\mathcal S(\beta)\leq\|\Phi_\beta\|_{A^2_\eta}^2,
 \qquad \|\Phi\|_{A^2_\eta}^2=\int_\D\frac{|\Phi|^2}{\eta}\,dA,
\]
and $B_\beta$ is bounded explicitly in terms of $\mathcal S(\beta)$.
\end{mainthm}

The weighted square norm $\|\cdot\|_{A^2_\eta}$ is the Weil--Petersson norm on holomorphic quadratic differentials in our normalisation, so the canonical energy is quantitatively equivalent to the Weil--Petersson norm of the Hopf differential of the harmonic extension. Section~\ref{connections} records connections with Loewner energy, minimal graphs in $\mathbb H^2\times\IR$, the minimal Lagrangian extension and anti-de Sitter geometry.

\medskip

For $p\ge1$ let $T^p$ be Guo's $p$-{\em integrable Teichm\"uller space}. These are the circle homeomorphisms with a quasiconformal extension $F$ such that $\int_\D|\mu_F|^p\,dA_\eta<\infty$, and let $\mathcal T_p$ be the circle homeomorphisms with a finite-distortion homeomorphic extension $G$ with $\I_p(G)<\infty$ (Section~\ref{two-classes}). Trivially $T^p\subset\mathcal T_p$, and Theorem~\ref{thm-C} is $\mathcal T_2=T^2$.

\begin{mainthm}[Finite-distortion boundary values are $p$-integrable Teichm\"uller]\label{thm-E}
For every $1<p<\infty$, $\mathcal T_p=T^p$, the boundary values of a finite-distortion homeomorphism $G:\D\to\D$ with $\I_p(G)<\infty$ have a quasiconformal extension with $\mu\in L^p(dA_\eta)$, and are therefore absolutely continuous with $\log g'$ lying in the Besov space $ B^{1/p}_{p,p}(\IS)$. The extension can be taken to be the Beurling--Ahlfors extension of the boundary map, with $\int|\mu|^p\,dA_\eta$ bounded in terms of $p$ and $\I_p(G)$ after normalisation. Moreover, for a quasisymmetric $g:\IR\to\IR$ and $q(x,t)=\log\bigl[(g(x+t)-g(x))/(g(x)-g(x-t))\bigr]$, the following are equivalent.
\begin{enumerate}
\item $g\in T^p$, \item the logarithmic modulus distortion of the quadrilaterals with vertices $x-t,x-\frac t3,x+\frac t3,x+t$ lies in $L^p(dx\,dt/t^2)$, \item $q\in L^p(dx\,dt/t^2)$.\end{enumerate}
\end{mainthm}

The proof is in Section~\ref{p-classes}. Its interior part is a modulus inequality for finite-distortion homeomorphisms, applied to a two-parameter family of quadrilaterals with four finite vertices and integrated against the invariant measure $dx\,dt/t^2$; its boundary part consists of two exact identities for $q$, one for the four-point cross-ratio and one for scale doubling, with telescoping arguments, and a dyadic interpolation estimate for the Beurling--Ahlfors extension. No absolute continuity or Besov theory is needed.  The Besov description of $T^p$, due to Shen \cite{Shen} at $p=2$ and Wei--Matsuzaki \cite{WeiMatsuzakiIntegrable} in general, is used only to translate the conclusion. The characterisation of $T^p$ by $q$ is, at $p=2$, essentially Bishop's Theorem~1.6 in \cite{Bish1}, and in part we have adapted his argument; see Remark~\ref{pc-bishop}.

\subsection{Organisation and method}
Section~\ref{no-hairs} proves the compactness theorem for the monotone class and the \emph{NoHairs} theorem: a bounded-energy monotone limit sends no interior point to the circle, and it is a homeomorphism as soon as it is locally injective or has a continuous right inverse. The proof is a M\"obius rescaling argument using only the conformal invariance of the excess energy. Section~\ref{quadratic-theory} proves Theorem~\ref{thm-A}: the direct method gives a monotone minimiser; local harmonic replacement on pullback cells makes it a real analytic harmonic diffeomorphism; a quantitative Schwarz--Bochner argument converts finite excess energy into the explicit dilatation bound \eqref{intro-qc-bound}; and asymptotic conformality follows by renormalisation. Section~\ref{schoen-conjecture} proves Theorem~\ref{thm-B}. Section~\ref{boundary-classes} proves Theorem~\ref{thm-C}: a boundary regularity criterion, the factorisation of a finite-energy map through a quasiconformal map, and a capacity estimate for boundary moduli. Section~\ref{p-classes} proves Theorem~\ref{thm-E} from the last two of these alone. Section~\ref{canonical-energy} proves Theorem~\ref{thm-D}, and Section~\ref{connections} collects connections and questions.

\subsection{Normalisations and conventions}\label{conventions}
Let $\ID$ be the unit disk, let $\IS=\partial\ID$, and write $dA$ for Euclidean area. Our convention for the Beltrami equation is
\begin{equation}\label{be}
 f_{\bar z}=\mu_f f_z,\qquad |\mu_f|<1\quad\hbox{almost everywhere}.
\end{equation}
For a mapping of finite distortion we require $f\in W^{1,1}_{\loc}$, a locally integrable nonnegative Jacobian, and $Df=0$ almost everywhere on $\{J_f=0\}$. We use the normalised Hilbert--Schmidt norm
\begin{equation}\label{db}
 \|Df\|^2=|f_z|^2+|f_{\bar z}|^2,\qquad
 \IK_f=\frac{1+|\mu_f|^2}{1-|\mu_f|^2},\qquad
 \|Df\|^2=\IK_f J_f.
\end{equation}
At points where $Df=0$ we set $\mu_f=0$ and $\IK_f=1$. The operator-norm distortion is instead $K_f=(1+|\mu_f|)/(1-|\mu_f|)$. In particular
\begin{equation}\label{Kminus}
 \IK_f-1=\frac{2|\mu_f|^2}{1-|\mu_f|^2}.
\end{equation}
These distinctions matter for the normalising constants in the energy. We refer to \cite{AIM,HK} for the analytic background and the theory of quasiconformal mappings.

\medskip

Throughout, ``homeomorphism of finite distortion'' means orientation preserving. We frequently use an inverse change of variables. The validity is by now well understood in the theory of mappings of finite distortion: both inverse and forward maps must have the required Sobolev regularity, the inverse derivative formula must hold, and the area formula must apply. For a finite-energy finite-distortion homeomorphism, Lemma~\ref{finite-distortion-inverse-calculus} gives the inverse Sobolev estimate and the energy inequality needed for admission to the variational class. A continuous surjection of the closed disk is \emph{monotone} if every point fibre (preimage of a point) is connected. A \emph{hair} is a fibre containing more than one point; it is a compact connected set. Boundary hairs and their tips are defined in Section~\ref{no-hairs}.

\medskip

Write
\begin{equation}\label{eta}
 \eta(z)=(1-|z|^2)^{-2},\qquad dA_\eta=\eta(z)\,dA(z).
\end{equation}
Thus $\eta|dz|^2$ has curvature $-4$. 

We use the following definition: an orientation-preserving circle homeomorphism $g(e^{i\theta})=e^{iu(\theta)}$, where $u(\theta+2\pi)=u(\theta)+2\pi$, is quasisymmetric if
\[
 M^{-1}\leq\frac{u(\theta+t)-u(\theta)}{u(\theta)-u(\theta-t)}\leq M
 \qquad(\theta\in\IR,\ 0<t\leq\pi)
\]
for some $M$. It is symmetric if the ratio tends uniformly to $1$ as $t\downarrow0$. We work initially with unnormalised boundary maps. To obtain Teichm\"uller spaces as manifolds, one fixes three boundary points or takes the appropriate M\"obius quotient.  

The inclusion of smooth circle diffeomorphisms in the quasisymmetric group connects these questions with Virasoro geometry and mathematical physics \cite{Witten,NV,TT}. The effective M\"obius group acting on the circle is $\mathrm{PSL}(2,\IR)$, equivalently $\mathrm{PSU}(1,1)$ in disk coordinates.

\medskip

If a bounded Beltrami coefficient $\mu$ on $\ID$ is extended by zero to the plane, its principal solution is normalised by $F(z)=z+O(1/z)$ at infinity. The {\em quasicircle} $\Gamma=F(\IS)$ bounds $\Omega=F(\ID)$. A Riemann map $f:\ID\to\Omega$ may be normalised by $f(0)=F(0)$ and $f(1)=F(1)$, then the Weil--Petersson condition on the curve can be expressed by
\[
 \int_\ID\left|\frac{f''}{f'}\right|^2\,dA<\infty.
\]
Its Beltrami formulation is the existence of an appropriate quasiconformal representative with $\mu\in L^2(dA_\eta)$, not the integrability of every representative \cite{TT,Bish1,Bish2}. For example, $F(z)=z|z|^{\alpha-1}$, $\alpha>0$, $\alpha\ne1$, has identity boundary values and constant nonzero $|\mu_F|$. Its boundary curve is the circle, but $\int_\D|\mu_F|^2\eta\area=\infty$.

The symbols $K$ and $\K$ always denote operator-norm and normalised Hilbert--Schmidt distortion, respectively. Boundary traces of inverse maps $h$ are denoted by $\beta$; those of forward maps $f$ are $\beta^{-1}$.

\subsection{The energies and their invariance}\label{energies}
Set
\[
 p_* = \max\{2,p\},\qquad
 R_p(s)=\frac{s^{p/2}}{(1-s)^{p_*/2}},\qquad 0\leq s<1,
\]
and define, for a map $F$ of finite distortion,
\begin{equation}\label{funct}
 \mathcal I_p(F)=\int_\ID R_p(|\mu_F|^2)\,dA_\eta .
\end{equation}
For $p\geq2$, \eqref{Kminus} gives
\begin{equation}\label{Psiidentity}
 \mathcal I_p(F)=2^{-p/2}\int_\ID(\IK_F-1)^{p/2}\,dA_\eta ,
\end{equation}
and in particular $\I_2(F)=\tfrac12\int_\D(\K_F-1)\,dA_\eta$. Under the inverse and area formulas, $H=F^{-1}$ satisfies
\begin{equation}\label{inverseenergy}
 \mathcal I_p(F)=\mathcal E_p(H)
 :=\int_\ID\eta(H)J_H R_p(|\mu_H|^2)\,dA.
\end{equation}
Writing $a=H_w$, $b=H_{\bar w}$ and $J=|a|^2-|b|^2$, this becomes
\begin{equation}\label{inversedensity}
 \mathcal E_p(H)=
 \begin{cases}
 \displaystyle\int_\ID\eta(H)|b|^p|a|^{2-p}\,dA,&1\leq p\leq2,\\[4pt]
 \displaystyle\int_\ID\eta(H)|b|^p J^{1-p/2}\,dA,&p\geq2.
 \end{cases}
\end{equation}
In particular,
\begin{equation}\label{anti}
 \mathcal E_2(H)=\int_\ID\eta(H)|H_{\bar w}|^2\,dA.
\end{equation}
This is the anti-holomorphic energy in the chosen metric normalisation. The full hyperbolic Dirichlet energy of a diffeomorphism onto the disk is infinite; \eqref{anti} is an identity between densities, or between integrals over relatively compact subdomains, and is never justified by subtracting two infinite total integrals.

\paragraph{M\"obius invariance.}
For disk automorphisms $\phi,\psi$,
\[
 |\mu_{\psi\circ F\circ\phi}(z)|=|\mu_F(\phi(z))|,
 \qquad \eta(\phi(z))|\phi'(z)|^2=\eta(z).
\]
Consequently change of variables gives
\begin{equation}\label{mobiusenergy}
 \mathcal I_p(\psi\circ F\circ\phi)=\mathcal I_p(F).
\end{equation}
The prescribed boundary data transform along with the map.
\paragraph{The inverse Sobolev estimate.}
The ordering of the energies is useful here. Since $0\leq r<1$,
\begin{equation}\label{subquadratic-energy-order}
 \frac{r}{1-r^2}\geq\frac{r^p}{1-r^2}\geq\frac{r^2}{1-r^2},
 \qquad 1\leq p\leq2.
\end{equation}
Consequently, for a fixed admissible forward map,
$\I_1(F)\geq\I_p(F)\geq\I_2(F)$. The inverse densities satisfy the same ordering:
\[
 |a||b|\geq |a|^{2-p}|b|^p\geq |b|^2,
 \qquad |b|\leq|a|.
\]
Thus bounded $p$-energy controls the quadratic anti-holomorphic energy for every $1\leq p\leq2$, including $p=1$. The $p=1$ condition is the stronger energy condition. No convexity or stationarity is needed for these comparisons.

Assume $H$ is a homeomorphism of $\D$ in $W^{1,2}(\D)$; the area inequality for Sobolev homeomorphisms then gives $\int_\ID J_H\,dA\leq\pi$, which is all that is used (no Lusin property is needed). If $1\leq p\leq2$, \eqref{inversedensity} and $|a|\geq|b|$, $\eta\geq1$, give $\int|b|^2\leq\mathcal E_p(H)$. If $p\geq2$, put $x=|b|^2/J$. The inequality $x\leq1+x^{p/2}$ gives $\int|b|^2\leq\pi+\mathcal E_p(H)$. Therefore
\begin{equation}\label{sobolev}
 \int_\ID(|H_w|^2+|H_{\bar w}|^2)\,dA\leq
 \begin{cases}\pi+2\mathcal E_p(H),&1\leq p\leq2,\\
 3\pi+2\mathcal E_p(H),&p\geq2.
 \end{cases}
\end{equation}
This does not give a forward $W^{1,2}$ bound on $F$.

For every $1\leq p\leq2$, the excess-distortion calculation can also be written
\begin{align}\label{inverse-excess-bound}
 2\E_p(H)&\geq\int_\D(\K_H-1)J_H\eta(H)\area
 =2\int_\D\eta(H)|H_{\bar w}|^2\area\\
 &\geq\int_\D(|H_w|^2+|H_{\bar w}|^2)\area-\pi.\notag
\end{align}
The pointwise identity is 
\[ (\K_H-1)J_H=2|H_{\bar w}|^2\]
 and this controls the anti-holomorphic derivative. Adding $\int J_H=\pi$ then controls the full Dirichlet integral. Lemma~\ref{monotone-compactness} obtains the same total-Jacobian identity from the fixed boundary trace for monotone Sobolev maps, without assuming an area formula for the limit.

In particular, $\sup_j\E_p(h_j)\leq M$ implies $\|h_j\|_{W^{1,2}(\D)}\leq\pi+2M$, with the same constant for all $1\leq p\leq2$. For quasisymmetric boundary data this gives the equicontinuity and monotone compactness proved in Section~\ref{no-hairs}.

\paragraph{Finite forward energy gives no forward Sobolev bound.}
Take a radial homeomorphism, near zero, with
\[
 F(re^{i\theta})=R(r)e^{i\theta},\qquad
 R(r)=(\log(e/r))^{-1/2},
\]
and join it monotonically and smoothly on an intermediate annulus to the identity near $\partial\D$. Such a joining is possible by choosing the inner radius sufficiently small. The radial and tangential stretches are $R'$ and $R/r$. Their ratio gives distortion of order $\log(1/r)$, so $\I_p(F)<\infty$ for every finite $p$; near the boundary the coefficient is zero. However,
\[
 \int_0^\varepsilon\left(\frac{R(r)}r\right)^2r\,dr
 =\int_0^\varepsilon\frac{dr}{r\log(e/r)}=\infty.
\]
Thus $F\notin W^{1,2}$ even though every finite-$p$ forward energy is finite. Estimate~\eqref{sobolev} concerns $H=F^{-1}$.
\section{Boundary compactness and the NoHairs theorem}\label{no-hairs}
This section contains the topological tools of the paper: a compactness theorem for bounded-energy monotone maps, and the \emph{NoHairs} theorem, which says that a bounded-energy monotone limit never sends an interior point to the circle. Combined with interior estimates, which are more generally available, NoHairs promotes a monotone minimiser to a homeomorphism. The argument relies on M\"obius invariance of the excess energy.

\medskip

A \emph{hair} of
a continuous monotone surjection $h:\overline\D\to\overline\D$ is a
nonsingleton fibre $h^{-1}(z)$. It is a compact connected set. If the trace
$\beta=h|_{\Sone}$ is a homeomorphism, a hair meeting $\Sone$ meets it at
exactly one point, called its \emph{tip}; such a fibre is a
\emph{boundary hair}. Equivalently, it is a nonsingleton fibre over a point
of the target circle. A hair lying wholly in the open disk is an
\emph{interior hair}. No local-connectivity assumptions are made.

Theorem~\ref{no-boundary-hairs} excludes boundary hairs from a bound for
the Dirichlet integrals over the normalised M\"obius orbit. Hyperbolic
energy invariance supplies this bound for finite-energy approximation
families. Proposition~\ref{hair-homeomorphism} then gives the full
homeomorphism conclusion once the interior hairs are excluded, for instance by local injectivity or by a continuous right inverse.
\subsection{Boundary equicontinuity of the renormalised inverse maps}
Write
\[
 \mathcal D(u)=\int_\D\bigl(|u_w|^2+|u_{\bw}|^2\bigr)\area.
\]
We use the following standard planar compactness fact, with its boundary step made explicit.

\begin{lemma}\label{boundary-modulus}
Let $u_j:\overline\D\to\overline\D$ be continuous monotone surjections in $W^{1,2}(\D)$, with orientation-preserving boundary homeomorphisms. Suppose $\sup_j\mathcal D(u_j)<\infty$. Suppose also that their boundary maps admit $K_0$-quasiconformal disk extensions $B_j$, with a common $K_0$, fixing the three points $1,i,-i$. Then the $u_j$ have a common modulus of continuity on $\overline\D$.
\end{lemma}
\begin{proof}
Let $\rho(z)=1/\bar z$ be reflection in $\Sone$, and extend $B_j$ to the sphere by
\[
 \widehat B_j(z)=
 \begin{cases}
 B_j(z),&|z|\leq1,\\
 \rho\circ B_j\circ\rho(z),&|z|>1.
 \end{cases}
\]
These are $K_0$-quasiconformal sphere homeomorphisms fixing $1,i,-i$.
Normalised quasiconformal compactness gives constants $R>1$ and $L<\infty$,
depending only on $K_0$, such that
$|\widehat B_j(z)|\leq L$ for $1\leq|z|\leq R$ and every $j$.
Indeed, the common spherical modulus of continuity, applied near $\Sone$,
keeps the images of a common collar away from infinity. The area formula
and the distortion inequality consequently give
\begin{equation}\label{no-hairs-collar-energy}
 \int_{1<|z|<R}\|D\widehat B_j\|^2\area
 \leq K_0\int_{1<|z|<R}J_{\widehat B_j}\area
 \leq K_0\pi L^2=:C_0.
\end{equation}

Glue $u_j$ on $\overline\D$ to $\widehat B_j$ outside $\D$, and call the
result $U_j$. The matching traces give $U_j\in W^{1,2}(\D_R)$ and
\[
 \int_{\D_R}\|DU_j\|^2\area\leq\mathcal D(u_j)+C_0.
\]
The glued map is a continuous monotone surjection of the sphere: fibres
over $\D$ are fibres of $u_j$, exterior fibres are singletons, and a fibre
over $\Sone$ contains its unique boundary preimage. Thus, for every closed
Euclidean disk $\overline{B(x,t)}\subset\D_R$,
\[
 U_j\bigl(\overline{B(x,t)}\bigr)
 \subset\operatorname{conv}U_j\bigl(\partial B(x,t)\bigr).
\]
For completeness, Youngs' uniform approximation theorem for monotone
surface maps, recalled in \cite[Theorem 2.1]{IOMonotoneApprox}, reduces
this inclusion to the corresponding fact for homeomorphisms. The
approximating images of this disk stay in the finite plane; a Jordan
domain lies in the convex hull of its boundary, and the inclusion passes
to the uniform limit.

Set $R_0=(R-1)/2$ and let $x,y\in\overline\D$, with
$0<d=|x-y|<R_0$. For almost every $t\in(d,R_0)$ the restriction of $U_j$
to $\partial B(x,t)$ is absolutely continuous. The convex-hull inclusion
and Cauchy--Schwarz give
\[
 \begin{split}
 |u_j(x)-u_j(y)|^2
 &\leq\left(\int_{\partial B(x,t)}|\partial_\tau U_j|\,ds\right)^2\\
 &\leq 2\pi t\int_{\partial B(x,t)}|\partial_\tau U_j|^2\,ds
 \leq4\pi t\int_{\partial B(x,t)}\|DU_j\|^2\,ds.
 \end{split}
\]
Here $|\partial_\tau U_j|^2\leq2\|DU_j\|^2$ uses our normalised
Hilbert--Schmidt convention. Integrating with respect to $dt/t$ yields
the explicit logarithmic modulus
\begin{equation}\label{no-hairs-log-modulus}
 |u_j(x)-u_j(y)|^2
 \leq\frac{4\pi\bigl(\mathcal D(u_j)+C_0\bigr)}
 {\log(R_0/|x-y|)}.
\end{equation}
It is uniform up to the circle. In particular, finite distortion of
$u_j$ is not an additional hypothesis of the lemma.
\end{proof}

A single fixed quasisymmetric boundary map is covered by the same argument with one fixed quasiconformal extension; the three-point normalisation is needed when the boundary maps themselves vary under renormalisation. Arbitrary M\"obius pre- and postcomposition need not preserve the numerical constant in a chosen chordal definition of quasisymmetry. What is preserved here is the quasiconformal extension constant.

\subsection{The NoHairs theorem}
In our application the $ {W^{1,2}(\D)}$ bound in the following result will follow easily from the conformal invariance of the excess hyperbolic energy.
\begin{theorem}[NoHairs: the M\"obius rescaling theorem]\label{no-boundary-hairs}
Let $h:\overline\D\to\overline\D$ be a continuous monotone surjection
in $W^{1,2}(\D)$, with orientation-preserving quasisymmetric boundary
values $\beta$. Suppose that there is a constant $C_*$ such that
\begin{equation}\label{orbit-bound}
\|\psi\circ h\circ\phi^{-1}\|_{W^{1,2}(\D)} \leq C_*
\end{equation}
for all disk automorphisms $\phi,\psi$ for which the boundary map
$\psi\circ\beta\circ\phi^{-1}$ fixes $1,i,-i$. Then
\begin{equation}\label{no-hairs-boundary-preimage}
 h^{-1}(\Sone)=\Sone.
\end{equation}
Equivalently, $h(\D)\subset\D$, and every boundary fibre is a singleton, that is
there are no boundary hairs. If, in addition, every nonsingleton fibre
meets $\Sone$, then $h$ is a homeomorphism of $\overline\D$.
\end{theorem}
\begin{proof}
Suppose that $z_0\in\D$ and $h(z_0)\in\Sone$. Because $\beta$ is a
homeomorphism, this fibre has a unique boundary point. Initial source
and target automorphisms place this point and its image at $i$, and a
second boundary point and its image at $-i$. These changes preserve the
orbit hypothesis: every subsequently normalised composition is also a
normalised composition of the original map. We may therefore assume
\[
 h(z_0)=h(i)=i,\qquad h(-i)=-i.
\]
For $0<t<1$ put
\[
 \phi_t(z)=\frac{z-ti}{1+itz}.
\]
The automorphism $\phi_t$ fixes $i,-i$ and tends to $-i$ at every point
of $\overline\D\setminus\{i\}$. The point
$\beta(\phi_t^{-1}(1))$ lies on the open boundary arc from $-i$ to $i$
containing $1$. Hence there is a unique automorphism $\psi_t$ fixing
$i,-i$ which sends this point to $1$. It has the form
$\psi_t=\phi_{s(t)}$ with $-1<s(t)<1$. Consequently
\[
 h_t=\psi_t\circ h\circ\phi_t^{-1}
\]
fixes all three boundary points $1,i,-i$.

Choose a quasiconformal disk extension $B$ of $\beta$. The maps
$B_t=\psi_t\circ B\circ\phi_t^{-1}$ extend the boundary values of $h_t$,
fix the same three points, and have the same quasiconformal bound as
$B$. The orbit bound and Lemma~\ref{boundary-modulus} give a common
modulus of continuity for $h_t$ on the \emph{closed} disk. On the other
hand, if $w_t=\phi_t(z_0)$, then
\[
 w_t\longrightarrow-i,\qquad
 h_t(w_t)=\psi_t(h(z_0))=i,\qquad h_t(-i)=-i.
\]
Thus $|h_t(w_t)-h_t(-i)|=2$, contrary to
\eqref{no-hairs-log-modulus}. This proves
\eqref{no-hairs-boundary-preimage}; injectivity of $\beta$ then makes
each boundary fibre a singleton.

If every nonsingleton fibre meets $\Sone$, any such fibre would be a
boundary fibre, which has just been excluded. All fibres are therefore
singletons. The continuous bijection $h$ between compact disks is a
homeomorphism.
\end{proof}

The argument excludes boundary hairs without a variational equation or
a finite-distortion hypothesis on the limit. Interior fibres require
the separate classification in the last sentence of the theorem;
Proposition~\ref{hair-homeomorphism} supplies that classification in the cases that arise. The contradiction uses the close
pair $w_t,-i$, so no locally uniform limiting map of the rescaled family
needs to be extracted.
\subsection{Application to bounded-energy approximation families}\label{no-hairs-application}
For $h\in W^{1,1}_{\loc}(\D)$ and disk automorphisms $\phi,\psi$, the inverse energy is preserved:
\begin{equation}\label{inverse-invariance}
 \E_p(\psi\circ h\circ\phi^{-1})=\E_p(h).
\end{equation}
Indeed $W_p$ is homogeneous of degree two and invariant under rotations of both source and target. By the chain rule for Sobolev maps composed with diffeomorphisms, the domain conformal factors cancel in the change of variables, and the target factors cancel by
\[
 \eta(\psi(h))|\psi'(h)|^2=\eta(h).
\]
Combining this with the inverse estimate \eqref{sobolev} gives the uniform bound
\[
 \| \psi\circ h\circ\phi^{-1} \|_{W^{1,2}(\D)}
 \leq3\pi+2\E_p(h),\qquad 1\leq p<\infty.
\]
For $p\leq2$ one may replace $3\pi$ by $\pi$.

For the distortion-power energies $\int_\D(\K_f^m-1)\eta\area$ of \cite{MYLp,MYRadonRiesz}, the calculation is
equally direct. If $f=h^{-1}$ and the inverse/area formulas hold, then for
$m\geq1$,
\begin{equation}\label{original-no-hairs-orbit-energy}
\|h\|_{W^{1,2}(\D)}=\int_\D\K_f \area
 \leq\pi+\int_\D(\K_f^m-1)\eta\area.
\end{equation}
The last integral is unchanged by the corresponding M\"obius
pre- and postcompositions. Thus either energy family supplies the orbit
bound required by the theorem.

\begin{corollary}[Compact containment without boundary hairs]\label{compact-containment}
Fix a quasisymmetric boundary map $\beta$. Let $h_j:\overline\D\to\overline\D$ be orientation-preserving homeomorphisms in $W^{1,2}(\D)$ with $h_j|_{\Sone}=\beta$, and suppose
\[
 \sup_j\E_p(h_j)\leq M<\infty.
\]
Then a subsequence converges uniformly on $\overline\D$ and weakly in $W^{1,2}(\D)$ to a continuous monotone surjection $h$ with the same boundary values. Every such limit satisfies $h^{-1}(\Sone)=\Sone$. Moreover, for each compact $C\Subset\D$ there is a compact $C'\Subset\D$ such that $h_j(C)\subset C'$ for all $j$.
\end{corollary}
\begin{proof}
The inverse estimate \eqref{sobolev}, which for a Sobolev homeomorphism uses only the area inequality $\int_\D J_{h_j}\leq\pi$, bounds the Dirichlet integral of $h_j$. Gluing to a fixed exterior extension of $\beta$ as in Lemma~\ref{boundary-modulus} gives equicontinuity on the closed disk. Arzel\`a--Ascoli and weak Sobolev compactness give the stated subsequence. A uniform limit of these homeomorphisms is a monotone surjection with boundary values $\beta$.

For each fixed pair $\phi,\psi$ normalising the boundary map, \eqref{inverse-invariance} bounds the Dirichlet integrals of $\psi\circ h_j\circ\phi^{-1}$ by $3\pi+2M$. These maps converge uniformly to $\psi\circ h\circ\phi^{-1}$, so weak lower semicontinuity of the \emph{Dirichlet integral} gives \eqref{orbit-bound} for the limit. Theorem~\ref{no-boundary-hairs} applies. This step does not require lower semicontinuity of the original nonconvex energy.

If the last assertion failed, a sequence $z_j\in C$ would have $h_j(z_j)$ tending to the circle. Apply the preceding compactness argument to that subsequence and let $z_j\to z\in C$. Uniform convergence would give $h(z)\in\Sone$, a contradiction.
\end{proof}
\subsection{From NoHairs to a homeomorphism}\label{hair-chain}
Theorem~\ref{no-boundary-hairs} removes the boundary hairs. The following elementary proposition collects the three ways in which the interior hairs are then removed in practice: by a classification of hairs, by local injectivity, or by a continuous right inverse.

\begin{proposition}[Interior hairs and the homeomorphism conclusion]\label{hair-homeomorphism}
Let $h:\overline\D\to\overline\D$ be a continuous monotone surjection whose boundary trace is a homeomorphism, and suppose $h^{-1}(\Sone)=\Sone$. Then $h$ is a homeomorphism of $\overline\D$ in each of the following cases.
\begin{enumerate}[label=\textup{(\alph*)}]
\item Every hair meets $\Sone$.
\item $h$ is a local homeomorphism on $\D$. If moreover $h$ is a $C^k$ local diffeomorphism on $\D$, $k\geq1$, then $h|_\D$ is a $C^k$ diffeomorphism of $\D$ onto itself.
\item There is a continuous map $g:\D\to\D$ with $h\circ g=\id_\D$. In this case $g(\D)=\D$ and $g=h^{-1}$ on $\D$; if $g$ is a $C^k$ diffeomorphism onto its image, then $g$ and $h$ are $C^k$ diffeomorphisms of $\D$ onto itself.
\end{enumerate}
Without the hypothesis $h^{-1}(\Sone)=\Sone$, case~\textup{(c)} still gives that $h^{-1}(\D)=g(\D)$, that $h|_{g(\D)}:g(\D)\to\D$ is a homeomorphism with inverse $g$, and that every hair meets $\Sone$.
\end{proposition}
\begin{proof}
(a) If a hair $E=h^{-1}(z)$ existed, a point $\xi\in E\cap\Sone$ would give $z=h(\xi)\in\Sone$, and $h^{-1}(\Sone)=\Sone$ would force $E\subset\Sone$. Injectivity of the boundary trace then gives $E=\{\xi\}$, a contradiction. Thus every fibre is a singleton, and a continuous bijection between compact disks is a homeomorphism.

(b) An interior fibre is compact, connected by monotonicity, and discrete by local injectivity, hence a singleton; boundary fibres are singletons because the trace is injective and $h^{-1}(\Sone)=\Sone$. The differentiability assertion follows by inverting the local diffeomorphism.

(c) The identity $h\circ g=\id$ makes $g$ injective, so invariance of domain makes $\Omega=g(\D)$ open and $h|_\Omega=g^{-1}$ a homeomorphism. For $z\in\D$ the fibre $h^{-1}(z)$ contains $g(z)$, and its intersection with $\Omega$ is exactly $\{g(z)\}$. Thus $g(z)$ is isolated in this connected fibre, which is therefore the singleton $\{g(z)\}$. Consequently $h^{-1}(\D)=\Omega$, every hair lies over $\Sone$ and meets $\Sone$ in its unique boundary preimage. This proves the final assertion, and (a) then gives the homeomorphism conclusion when $h^{-1}(\Sone)=\Sone$. The differentiability assertion follows by inverting the given diffeomorphism.
\end{proof}

If the right-inverse identity $h\circ g=\id$ initially holds only on a set of full measure, continuity of both maps extends it to all of $\D$. The hypothesis $h^{-1}(\Sone)=\Sone$ is supplied by Theorem~\ref{no-boundary-hairs} for every bounded-energy limit, through Lemma~\ref{monotone-compactness} below; neither energy convergence nor finite distortion of the limit is needed for that topological conclusion.

\subsection{The class and compactness}\label{monotone-class}
Fix an orientation-preserving quasisymmetric boundary homeomorphism $\beta$. For $0<p\leq2$, let $\mathcal M_\beta$ consist of maps $h$ such that
\begin{enumerate}
\item $h:\overline\D\to\overline\D$ is a continuous monotone surjection, $h|_{\Sone}=\beta$, and $h(\D)\subset\D$;
\item $h\in W^{1,2}(\D)$ and $J_h\geq0$ almost everywhere;
\item $\E_p(h)=\int_\D\eta(h)W_p(h_w,h_{\bw})\area<\infty$, where $W_p(a,b)=|a|^{2-p}|b|^p$ on $|b|\leq|a|$, and $W_p(0,0)=0$.
\end{enumerate}
We assume that this class contains a finite-energy competitor, for example the inverse of an admissible finite-energy homeomorphism.  
This definition intentionally permits nonzero rank-one gradients: when $|a|=|b|$, $W_p(a,b)=|a|^2$. Such a map need not have finite distortion in the sense of Section~\ref{conventions}, which requires $Dh=0$ on $J_h=0$. The energy here is our extension of the inverse density to monotone maps.

\begin{lemma}[Inverse coercivity and compactness]\label{monotone-compactness}
For every $h\in\mathcal M_\beta$,
\begin{equation}\label{trace-jacobian}
 \int_\D J_h\area=\pi,\qquad
 \|h\|_{W^{1,2}(\D)}\leq\pi+2\E_p(h).
\end{equation}
Every sequence $h_j\in\mathcal M_\beta$ with $\sup_j\E_p(h_j)\leq M$ has a subsequence converging uniformly on $\overline\D$ and weakly in $W^{1,2}(\D)$ to a continuous monotone surjection $h$ with trace $\beta$, $J_h\geq0$, and $h(\D)\subset\D$. For each compact $C\Subset\D$, all $h_j(C)$ lie in a common compact subset of $\D$. This compactness conclusion does not yet assert $\E_p(h)<\infty$.
\end{lemma}
\begin{proof}
Choose a quasiconformal extension $B$ of $\beta$. It belongs to $W^{1,2}(\D)$ and has $\int J_B=\pi$. Since $h-B\in W^{1,2}_0(\D)$, the null-Lagrangian identity gives $\int J_h=\int J_B$. We can verify it by approximating $h-B$ strongly in $W^{1,2}$ by compactly supported smooth maps: the integrated linear cofactor term vanishes by the distributional Piola identity, and the integrated determinant of each compactly supported perturbation is zero. Determinants converge in $L^1$ under strong $W^{1,2}$ convergence. Thus this argument does not use an area formula for $h$.

On $|b|\leq|a|$, $W_p(a,b)\geq|b|^2$ for $p\leq2$. As $\eta\geq1$, the identity $|a|^2+|b|^2=J+2|b|^2$ proves \eqref{trace-jacobian}. Lemma~\ref{boundary-modulus}, with one fixed exterior extension of $\beta$, now gives equicontinuity. Uniform and weak Sobolev compactness give the subsequence. A uniform limit of monotone surjections between these compact disks is a monotone surjection with the same trace. Distributional continuity of the planar Jacobian under uniform convergence and weak $W^{1,2}$ convergence gives $J_h\geq0$.

The degree-two homogeneity and rotation invariance of $W_p$ prove \eqref{inverse-invariance} directly for maps in $\mathcal M_\beta$, without an inverse change of variables. The total-Jacobian identity applies also after M\"obius pre- and postcomposition. Every normalised map $\psi\circ h_j\circ\phi^{-1}$ therefore has Dirichlet integral at most $\pi+2M$. Passing this bound to the uniform limit by weak lower semicontinuity of the Dirichlet integral gives the orbit bound \eqref{orbit-bound}. Theorem~\ref{no-boundary-hairs} implies $h(\D)\subset\D$. Compact containment of the $h_j(C)$ follows exactly as in Corollary~\ref{compact-containment}.
\end{proof}
\section{The quadratic variational problem, $p=2$.}\label{quadratic-theory}
This section proves Theorem~\ref{thm-A} using compactness, local harmonic replacement and scalar estimates. None of the convexity or approximation machinery needed for $p\ne2$ in Part~II is required here.

\subsection{Direct existence at the quadratic exponent}
\begin{proposition}\label{quadratic-direct-existence}
If $\beta$ is quasisymmetric and $\mathcal M_\beta$ contains a finite-energy competitor, $\E_2$ attains its minimum in $\mathcal M_\beta$.
\end{proposition}
\begin{proof}
For a minimising sequence, Lemma~\ref{monotone-compactness} gives uniform convergence on the closed disk, weak $W^{1,2}$ convergence and compact containment of images of compact subsets. On such a subset $U$, the weights $\eta(h_j)$ converge uniformly to $\eta(h)$. Since $W_2(a,b)=|b|^2$ is convex, the usual weak lower semicontinuity gives
\[
 \int_U\eta(h)|h_{\bar w}|^2\,dA
 \leq\liminf_j\int_U\eta(h_j)|(h_j)_{\bar w}|^2\,dA.
\]
Exhaust the disk by compactly contained domains. Nonnegativity proves global lower semicontinuity and hence attainment. No polyconvexity theorem is used.
\end{proof}

\subsection{Quadratic minimisers: local harmonic replacement}\label{quadratic-replacement}
Define the local hyperbolic Dirichlet integral by
\[
 \mathcal D_\eta(u;U)
 =\int_U\eta(u)(|u_w|^2+|u_{\bw}|^2)\area.
\]
Whenever $u(U)$ lies in a compact subset of $\D$,
\begin{equation}\label{quadratic-local-area-identity}
 \mathcal D_\eta(u;U)
 =2\E_2(u;U)+\int_U\eta(u)J_u\area.
\end{equation}
We use this identity only on compactly contained sets.

\begin{proposition}[Interior regularity of the quadratic minimiser]\label{quadratic-minimiser-regularity}
Every finite-energy minimiser of $\E_2$ in $\mathcal M_\beta$ is a hyperbolic harmonic diffeomorphism of $\D$ onto itself and is real analytic in the interior.
\end{proposition}
\begin{proof}
By compactness and NoHairs, $h^{-1}(\Sone)=\Sone$.
Fix a small hyperbolic geodesic disk $B\Subset\D$.
The appropriate cell is the regularised preimage
\[
 V=h^{-1}(B),\qquad U=\operatorname{int}\overline V\Subset\D.
\]
For a monotone disk map with homeomorphic boundary trace, this cell
is simply connected. The prime-end construction in
\cite[Theorem 10 and Lemma 11]{MYDiffeomorphic} gives a Riemann map
$\varphi:\D\to U$ for which $q=h\circ\varphi$ extends continuously
to the closed disk and has monotone degree-one trace onto $\partial B$.
Here regularisation allows the cell boundary to absorb fibres over
$\partial B$. To explain the boundary assertion, use null-chains of
crosscuts in $U$. Continuity of $h$ makes their images shrink; their
images separate off caps of $B$ away from a fixed interior basepoint.
Those caps shrink to a single point of $\partial B$. Thus $h$ is
constant on each prime-end impression. The cyclic order of the caps
gives the monotone trace. This is the topological cell step of the
replacement, and does not assume that $h$ is already injective.

Minimise $\mathcal D_\eta$ (defined above at \ref{quadratic-local-area-identity}) among maps in
$q+W^{1,2}_0(\D)$ with values in $\overline B$.
On this compact target the metric is smooth and bounded above and
below. Weak compactness and lower semicontinuity give a minimiser.
Nearest-point projection onto the geodesically convex disk is
nonexpanding, so it is also a local unconstrained energy minimiser;
the usual harmonic-map regularity makes it smooth in the interior.
It is unique. Indeed, if $u_t$ is pointwise geodesic interpolation
between two admissible maps, convexity of energy in a CAT(0) target
gives, with our normalisation,
\begin{equation}\label{quadratic-geodesic-convexity}
 \mathcal D_\eta(u_t)
 \leq(1-t)\mathcal D_\eta(u_0)+t\mathcal D_\eta(u_1)
 -\frac{t(1-t)}2\int_\D
       |\nabla d_\eta(u_0,u_1)|^2\,dA.
\end{equation}
For a smooth negatively curved target this follows by the second
variation along target geodesics; the metric-space formulation is
given in \cite{KorevaarSchoen}. Equal traces make the distance have
zero trace. If both maps minimise, \eqref{quadratic-geodesic-convexity}
forces that distance to vanish.

The minimiser, denoted by $\widetilde g$, is a diffeomorphism onto
$B$. Here is the approximation step behind that assertion.
Approximate the monotone boundary trace of $q$ by smooth
orientation-preserving boundary homeomorphisms, uniformly and in
the $H^{1/2}$ trace norm. One may smooth its nondecreasing lift and
add a positive linear term before renormalising its period.
The harmonic solutions for these traces are diffeomorphisms by
Jost's convex-target univalence theorem \cite{JostUnivalent};
see also \cite[Lemma 3.2]{KalajReplacement}.
Their energy bounds and interior estimates give a harmonic limit and
the convex-target boundary barriers give its prescribed continuous
trace. Uniqueness identifies the limit with $\widetilde g$.
Smooth local convergence gives $J_{\widetilde g}\geq0$.
The degree-one boundary trace and the maximum principle give a
proper degree-one map onto $B$, so the curved Lewy theorem
\cite{Martin} makes $\widetilde g$ a diffeomorphism.

Transplant $\widetilde g$ to $U$ and put $g=h$ off $U$.
The prime-end trace just described ensures continuity on
$\partial U$. The cell construction preserves monotonicity:
fibres over $B$ are singletons, fibres outside $\overline B$ are
unchanged, and the noncrossing boundary trace retains connected
fibres over $\partial B$.
Matching Sobolev traces gives $g\in W^{1,2}$ and $J_g\geq0$.
This is precisely the local gluing step of
\cite[Lemma 3.3]{KalajReplacement}; monotone Sobolev approximation
is available from \cite[Theorem 1.3]{IOMonotoneApprox}.
Conformal invariance of the Dirichlet integral and uniqueness give
\begin{equation}\label{quadratic-replacement-decrease}
 \mathcal D_\eta(g;U)\leq\mathcal D_\eta(h;U),
 \qquad\text{with equality only if }g=h.
\end{equation}

The weighted Jacobian term in \eqref{quadratic-local-area-identity}
is invariant under this replacement. Choose a smooth
$\Omega_0\Subset\D$ containing $\overline U$.
The maps have compact target image there and agree near
$\partial\Omega_0$. The form $\eta(z)\,dx\wedge dy$ has a smooth
primitive on the disk. Stokes' theorem makes its pullback integral
the same for $g$ and $h$. Uniform and strong $W^{1,2}$ smooth
approximation on $\Omega_0$ proves the same statement for our maps, while
locality cancels the integrals off $U$.
Thus
\[
 2[\E_2(g;U)-\E_2(h;U)]
 =\mathcal D_\eta(g;U)-\mathcal D_\eta(h;U)\leq0.
\]
The glued map is admissible and the change has finite energy.
Minimality forces equality, and therefore $h=g$ on the cell.
Such cells cover $\D$. Local injectivity and connected fibres give
global injectivity; surjectivity was already known. The hyperbolic
harmonic-map equation is analytic elliptic, so $h$ is real analytic.
\end{proof}

This establishes the local regularity required by the quadratic scalar estimates without invoking the Schoen existence theorem or a global quasiconformal bound. For $p<2$, identity~\eqref{quadratic-local-area-identity} does not identify the functional with harmonic-map energy plus a fixed area term, and hence the regularity theory of Part~II is quite different.

\subsection{A quadratic bound from the barrier energy alone}
\label{quadratic-boundary-dependence}
A quantitative interior modulus and the harmonic Bochner identities
give the following bound from finite excess energy alone. Its hypotheses concern
an already harmonic diffeomorphism; they do not hold automatically for
the approximants, whose equations have an additional term.

\begin{proposition}[Excess energy controls dilatation]
\label{quadratic-energy-qc-bound}
Let $h:\D\to\D$ be an orientation-preserving harmonic diffeomorphism
for the hyperbolic metrics, and suppose
\[
 \E_2(h)=\int_\D\eta(h)|h_{\bw}|^2\area\leq E_0<\infty.
\]
Set
\begin{equation}\label{quadratic-energy-bound-constants}
 D_0=\pi+2E_0,\qquad
 B_0=\frac{64D_0}{9\pi}\exp(8\pi D_0).
\end{equation}
If $E_0>0$, then
\begin{equation}\label{quadratic-energy-qc-explicit}
 \boxed{\quad
 1-\|\mu_h\|_\infty^2
       \geq\frac12\exp\!\left(-\frac8\pi B_0E_0\right),
 \qquad
 K_h=K_{h^{-1}}
       \leq8\exp\!\left(\frac8\pi B_0E_0\right).
 \quad}
\end{equation}
Recall $K_h$ is the maximal operator-norm dilatation.
For $E_0=0$, $h$ is conformal. More generally, for every
$t>E_0/(4\pi)$,
\begin{equation}\label{quadratic-heat-gap}
 1-\|\mu_h\|_\infty^2
       \geq e^{-16B_0t}
                    \left(1-\frac{E_0}{4\pi t}\right).
\end{equation}
Optimising this last estimate gives $K_h\to1$ as $E_0\downarrow0$,
unlike the deliberately coarse displayed bound in
\eqref{quadratic-energy-qc-explicit}.
\end{proposition}
\begin{proof}
Use the hyperbolic energy densities
\[
 \mathcal H(w)=\frac{\eta(h(w))}{\eta(w)}|h_w|^2,\qquad
 \mathcal L(w)=\frac{\eta(h(w))}{\eta(w)}|h_{\bw}|^2,
 \qquad r=|\mu_h|.
\]
We first prove
\begin{equation}\label{quadratic-energy-density-bounds}
 1\leq\mathcal H\leq\mathcal H+\mathcal L\leq B_0,
 \qquad \int_\D r^2\,dA_\eta\leq E_0.
\end{equation}
The lower bound is the usual Schwarz comparison, whose completeness
hypothesis we make explicit: The conformal metric
$g_{\mathcal H}=\eta(h)|h_w|^2|dw|^2$ is complete: the length of
$h\circ\gamma$ in the target is at most twice the
$g_{\mathcal H}$-length of $\gamma$, and $h$ is a proper
diffeomorphism onto the complete hyperbolic disk. The harmonic
Bochner identity gives
\[
 K_{g_{\mathcal H}}=-4(1-r^2)\geq-4.
\]
Apply the complete-surface Schwarz lemma
\cite[Theorem 2']{YauSchwarz} to the holomorphic identity from
$(\D,g_{\mathcal H})$ to $(\D,\eta|dw|^2)$.
This gives $g_{\mathcal H}\geq\eta|dw|^2$, hence $\mathcal H\geq1$.
As $\mathcal L=\mathcal H r^2$ and
$\int\mathcal L\,dA_\eta=\E_2(h)$, the integral assertion follows.
This comparison uses no prior quasiconformal constant.

For the upper bound, normalise any chosen source point and its image
to $0$ by disk automorphisms, and continue to denote the normalised
map by $h$. These operations preserve $\E_2$ and the hyperbolic
energy densities. The Euclidean Dirichlet estimate gives
$\|h\|_{W^{1,2}(\D)}\leq D_0$.
Let $\ell(s)$ be the Euclidean length of $h(\partial\D_s)$.
The circle Cauchy--Schwarz estimate gives, for $0<a<b<1$,
\[
 \int_a^b\ell(s)^2\,\frac{ds}{s}
 \leq 2\pi\int_{\D_b\setminus\D_a}|Dh|_{\mathrm{HS}}^2\area
 \leq4\pi D_0.
\]
Here $|Dh|_{\mathrm{HS}}^2=2(|h_w|^2+|h_{\bw}|^2)$ is the
unnormalised real Hilbert--Schmidt norm. Choose
\[
 a=\tfrac12 e^{-4\pi D_0},\qquad b=\tfrac12.
\]
There is an $s\in(a,b)$ with $\ell(s)\leq1$.
Since $h(0)=0$ and $h(\D_s)$ is a Jordan domain, its diameter is
at most $\ell(s)/2$. Consequently $h(\D_a)\subset\overline\D_{1/2}$.

Harmonicity into a nonpositively curved target makes the Euclidean
source energy density
$e(w)=\eta(h)(|h_w|^2+|h_{\bw}|^2)$ subharmonic, by Bochner's
formula. Its mean-value inequality therefore gives
\[
 (\mathcal H+\mathcal L)(0)=e(0)
 \leq\frac1{\pi a^2}\int_{\D_a}e\area
 \leq\frac{16D_0}{9\pi a^2}
 =B_0.
\]
The chosen source point was arbitrary, proving
\eqref{quadratic-energy-density-bounds}.

Now put $v=1-r^2$, so $0<v\leq1$. The harmonic Bochner identity for the ratio of the two energy
densities gives, where $r>0$,
\[
 \Delta_\eta\log r=-8\mathcal H(1-r^2),\qquad
 \Delta_\eta v
 =16\mathcal H r^2v-4|\nabla_\eta r|^2
 \leq16B_0v.
\]
The inequality extends across the conformal set:
$r^2=|\mu_h|^2$ and $v$ are smooth there.
Let $P_t$ be the heat semigroup for $\Delta_\eta$. Positivity and
heat comparison, applied to this bounded supersolution, give
\[
 v\geq e^{-16B_0t}P_tv.
\]
The heat kernel of the curvature-$-4$ disk satisfies
$0\leq p_t(x,y)\leq(4\pi t)^{-1}$, and $P_t1=1$.
For completeness, the bound follows by comparison with
$G(t,\rho)=(4\pi t)^{-1}e^{-\rho^2/(4t)}$, where
$\rho=d_\eta(x,y)$: its radial Laplacian gives
\[
 (\partial_t-\Delta_\eta)G
 =\frac{\rho\coth(2\rho)-1/2}{t}\,G\geq0,
\]
and its initial mass is the unit point mass. Thus
\[
 P_tv=1-P_t(r^2)
 \geq1-\frac1{4\pi t}\int_\D r^2\,dA_\eta
 \geq1-\frac{E_0}{4\pi t}.
\]
This proves \eqref{quadratic-heat-gap}. Take $t=E_0/(2\pi)$
and use $K_h=(1+r)^2/(1-r^2)\leq4/(1-r^2)$ to obtain
\eqref{quadratic-energy-qc-explicit}.
For small $E_0$, choosing $t=\sqrt{E_0}$ in
\eqref{quadratic-heat-gap} shows that its right side tends to $1$.
\end{proof}
\paragraph{Boundary data that enter the estimate.}
Let the prescribed forward boundary map be $g$, with inverse trace
$\beta=g^{-1}$ for $h$. Choose any admissible finite-energy extension
$F_0$ of $g$ whose inverse is an allowable competitor. Minimality
gives
\begin{equation}\label{quadratic-boundary-energy-input}
 \E_2(h)\leq E_0:=\I_2(F_0)
 =\int_\D\frac{|\mu_{F_0}|^2}{1-|\mu_{F_0}|^2}\,dA_\eta.
\end{equation}
In particular, if
$k_0=\|\mu_{F_0}\|_\infty<1$ and
$N_0=\int_\D|\mu_{F_0}|^2\,dA_\eta$, one may use
$E_0=N_0/(1-k_0^2)$ in the proposition.
Thus the quantitative input is the barrier energy; the two numbers
$k_0,N_0$ are a convenient sufficient way to bound it.
There is no additional dependence on boundary derivatives, on a
quasisymmetry constant, or on the rate at which the energy approaches
the boundary in this estimate.
It can also be expressed intrinsically using the infimum of
\eqref{quadratic-boundary-energy-input} over admissible extensions,
which is invariant under pre- and postcomposition by disk
automorphisms.
Turning it into a bound stated directly in terms of a chosen norm of
$\log g'$ requires a quantitative extension estimate.

\paragraph{Uniformity under changes of boundary data.}
For harmonic minimisers $h_j$ with admissible inverse competitors
$H_{0,j}$, the quantitative conclusion is
\[
 \sup_j\E_2(H_{0,j})\leq E_*<\infty
 \quad\Longrightarrow\quad
 \sup_jK(h_j)\leq F(E_*)<\infty,
\]
where $F$ is supplied by \eqref{quadratic-energy-qc-explicit}, with
$F(0)=1$. This implication is not circular: the bound for
$\mathcal H$ follows from the Dirichlet estimate, the circle-length
argument and Bochner's mean-value inequality, before any
quasiconformal bound for $h_j$ is used.

Two sequences denoted by $h_j$ must be distinguished.
The M\"obius renormalisations of one fixed map share its excess energy.
The harmonic extensions of varying smooth boundary maps in
Section~\ref{schoen-conjecture} need not have uniformly bounded excess energies,
even when their traces have a common quasisymmetry bound.
Thus dependence on an admissible barrier energy is a dependence
on boundary data, but the present proof does not replace it by
dependence on the quasisymmetry constant alone. We address that next.
\subsection{A bound from energy density alone}
\begin{proposition}[A quadratic bound from the holomorphic energy density]
\label{quadratic-density-qc-bound}
Let $h:\D\to\D$ be an orientation-preserving hyperbolic harmonic
diffeomorphism, and suppose
\[
 \mathcal H(w)=\frac{\eta(h(w))}{\eta(w)}|h_w|^2\leq B<\infty.
\]
Then $B\geq1$. If $B=1$, the map is conformal. If $B>1$, then
\begin{equation}\label{quadratic-density-gap}
 \boxed{\quad
 1-\|\mu_h\|_\infty^2\geq\frac{e^{-2B}}{B},
 \qquad K_h\leq4Be^{2B}.
 \quad}
\end{equation}
Also $\sup_\D|\Phi_2|/\eta\leq B$.
No total-energy hypothesis or prior quasiconformal bound is required.
This makes quantitative the bounded-energy-density criterion associated
with \cite{WanHarmonic}.
\end{proposition}
\begin{proof}
Use $g_{\mathbb H}=4\eta|dw|^2$, of curvature $-1$, for gradients and
Laplacians in this proof. Completeness of
$\mathcal H g_{\mathbb H}$ and the Schwarz comparison, justified in
the preceding energy estimate, give $\mathcal H\geq1$.
Put $v=1-|\mu_h|^2$. The Bochner identities imply
\begin{equation}\label{quadratic-density-heat-system}
 \Delta\log\mathcal H=2(\mathcal H v-1),\qquad
 \Delta v\leq4\mathcal H v(1-v),\qquad
 \Delta\mathcal H=2\mathcal H(\mathcal H v-1)
                    +\frac{|\nabla\mathcal H|^2}{\mathcal H}.
\end{equation}
The inequality for $v$ extends across Hopf zeros by smoothness.

Fix $a>2$ and $\epsilon>0$. The function
$\Psi_\epsilon=\log(v+\epsilon)+a\mathcal H$ is bounded below,
even if $\inf v=0$. For $\tau>0$ minimise
$\Psi_\epsilon+\tau\chi$, where
$\chi(w)=\sqrt{1+d_{\mathbb H}(0,w)^2}$.
At a minimum, $|\nabla\chi|\leq1$, $\Delta\chi\leq3$ and
$\nabla\log(v+\epsilon)=-a\nabla\mathcal H-\tau\nabla\chi$.
Since $\mathcal H\geq1$, completing a square gives
\[
 -|\nabla\log(v+\epsilon)|^2
       +\frac{a}{\mathcal H}|\nabla\mathcal H|^2
 \leq C_a\tau^2.
\]
Consequently, at this minimum,
\[
 -3\tau\leq\Delta\Psi_\epsilon
 \leq4\mathcal H\frac{v(1-v)}{v+\epsilon}
              +2a\mathcal H(\mathcal H v-1)+C_a\tau^2
 \leq2\mathcal H\{a\mathcal H v-(a-2)\}+C_a\tau^2.
\]
As $\tau\downarrow0$, the values of $\Psi_\epsilon$ at these
points tend to its infimum. A subsequence of the bounded pairs
$(\mathcal H,v)$ tends to $(\mathcal H_*,v_*)$, with
$\mathcal H_*v_*\geq1-2/a$.
The function $e^{aH}/H$ is increasing for $H\geq1$; hence
\[
 \inf\Psi_\epsilon
 \geq\log(1-2/a)-\log\mathcal H_*+a\mathcal H_*
 \geq\log(1-2/a)+a.
\]
It follows, first letting $\epsilon\downarrow0$, that
\begin{equation}\label{quadratic-density-time-gap}
 v\geq(1-2/a)e^{-a(B-1)}\qquad(a>2).
\end{equation}
Choose $a=2B/(B-1)$ to obtain \eqref{quadratic-density-gap},
using $K_h=(1+|\mu_h|)^2/v\leq4/v$.
If $B=1$, the first identity in \eqref{quadratic-density-heat-system}
forces $v=1$. Finally $|\Phi_2|/\eta=\mathcal H|\mu_h|\leq B$.
\end{proof}

\begin{corollary}[A density bound that is sharp at the conformal limit]
\label{quadratic-density-near-conformal}
Under Proposition~\ref{quadratic-density-qc-bound}, for every
$t>\tfrac12\log B$ one also has
\begin{equation}\label{quadratic-density-heat-improvement}
 1-\|\mu_h\|_\infty^2
 \geq\frac{4t-2\log B}{e^{4Bt}-1}.
\end{equation}
Consequently $K_h\to1$ whenever $B\downarrow1$.
For example, if $0<\ell=\log B<1$, the explicit gap
\[
 \delta_{\mathrm{near}}(B)
 =\frac{2(\sqrt\ell-\ell)}{e^{2B\sqrt\ell}-1}
\]
is valid and tends to one. This gives
$K_h\leq(1+\sqrt{1-\delta_{\mathrm{near}}})/
              (1-\sqrt{1-\delta_{\mathrm{near}}})$
and in particular $K_h-1=O((\log B)^{1/4})$ near $B=1$.
One may take the better of this estimate and
\eqref{quadratic-density-gap}.
\end{corollary}
\begin{proof}
Work again in curvature $-1$, and let $P_t$ be its heat semigroup.
Set $u=\log\mathcal H$ and $v=1-|\mu_h|^2$.
The bounded identities in \eqref{quadratic-density-heat-system}
give $0\leq u\leq\log B$ and $P_sv\leq e^{4Bs}v$.
Integrating the heat equation for $u$ yields
\[
 t-\tfrac12\log B
 \leq\int_0^tP_s(\mathcal H v)\,ds
 \leq B\int_0^tP_sv\,ds
 \leq\tfrac14(e^{4Bt}-1)v.
\]
This proves \eqref{quadratic-density-heat-improvement}.
Take $t=\sqrt\ell/2$ and expand at $\ell=0$.
All functions in this argument are bounded; no total-energy
integrability or initial quasiconformal constant is required.
\end{proof}

\begin{corollary}[Finite quadratic energy gives asymptotic conformality]
\label{quadratic-asymptotic-conformality}
Every hyperbolic harmonic disk diffeomorphism with finite $\E_2$
satisfies $|\mu_h(w)|\to0$ uniformly as $|w|\to1$.
\end{corollary}
\begin{proof}
Proposition~\ref{quadratic-energy-qc-bound} gives a global bound for
$\mathcal H+\mathcal L$ and
$\int_\D|\mu_h|^2\,dA_\eta<\infty$.
If $w_j\to\Sone$, precompose by automorphisms taking $0$ to $w_j$
and postcompose to make the resulting maps $h_j$ fix zero.
Their common hyperbolic energy-density bound gives bounded images
on compact source disks and smooth local harmonic compactness.
The lower bound $\mathcal H_j\geq1$ persists, so the limit has
nonzero holomorphic derivative. For every fixed hyperbolic ball
about zero, conformal invariance transforms
$\int|\mu_{h_j}|^2\,dA_\eta$ into an integral over a ball centred
at $w_j$. These balls escape every compact set, so the integrals
tend to zero. Smooth convergence now gives
$|\mu_{h_j}(0)|=|\mu_h(w_j)|\to0$. The arbitrary choice of the
sequence proves the assertion.
\end{proof}

\subsection{Quadratic uniqueness}
For $p=2$, a $C^2$ diffeomorphism with holomorphic hyperbolic Hopf differential is harmonic. Indeed, if
\[
 \tau=H_{w\bar w}+\frac{\eta_z}{\eta}(H)H_wH_{\bar w},
\]
then $\partial_{\bar w}Q_H=\eta(H)(\tau\overline{H_{\bar w}}+H_w\overline\tau)$. Since $|H_w|>|H_{\bar w}|$, vanishing of the left side forces $\tau=0$.

\begin{corollary}[Quadratic uniqueness]\label{quadratic-uniqueness}
Let $h_1,h_2:\overline\D\to\overline\D$ be orientation-preserving homeomorphisms with the same quasisymmetric boundary values. Suppose each is a $C^2$ diffeomorphism in $\D$, has finite $\E_2$, and has holomorphic weighted Hopf differential. Then $h_1=h_2$.
\end{corollary}
\begin{proof}
The preceding identity shows that both maps are hyperbolic harmonic. Proposition~\ref{quadratic-energy-qc-bound} makes both globally quasiconformal. The uniqueness theorem for quasiconformal harmonic extensions \cite[Section 1]{LiTamProperII} gives equality.
\end{proof}

This application uses the established quasiconformal harmonic uniqueness class. It does not require upgrading weighted $A^2_\eta$ information to global $L^1$ integrability of the Hopf coefficient. Proposition~\ref{quadratic-minimiser-regularity} verifies its regularity hypotheses for every finite-energy quadratic minimiser constructed here. The uniqueness theorem is used after the independent quasiconformality proof, not to obtain that bound.
\subsection{Proof of Theorem~\ref{thm-A}}
\begin{proof}[Proof of Theorem~\ref{thm-A}]
The Weil--Petersson class is closed under taking inverses \cite{TT,Matsuzaki}. Choose a quasiconformal extension $F_0$ of $\beta^{-1}$ such that, for $\mu_0=\mu_{F_0}$,
\[
 k_0:=\|\mu_0\|_\infty<1,
 \qquad \int_\D|\mu_0|^2\eta\area<\infty.
\]
Then $h_0=F_0^{-1}$ is an admissible barrier in $\mathcal M_\beta$, since the quasiconformal change of variables gives
\begin{equation}\label{wp-quadratic-barrier}
 \E_2(h_0)=\I_2(F_0)
 \leq\frac1{1-k_0^2}\int_\D|\mu_0|^2\eta\area<\infty.
\end{equation}
Propositions~\ref{quadratic-direct-existence} and \ref{quadratic-minimiser-regularity} give a minimiser which is a real analytic hyperbolic harmonic diffeomorphism. Since $\E_2(h_\beta)\leq\E_2(h_0)=E_0$, Proposition~\ref{quadratic-energy-qc-bound} gives \eqref{intro-qc-bound}, Corollary~\ref{quadratic-asymptotic-conformality} gives asymptotic conformality, and Corollary~\ref{quadratic-uniqueness} gives uniqueness.
\end{proof}

Membership of the Weil--Petersson class means that at least one quasiconformal representative has a hyperbolically square-integrable Beltrami coefficient; it is not an assumption that every representative has this property (see the example in Section~\ref{conventions}), nor an a priori quasiconformal bound on the minimiser.

\section{The Schoen conjecture}\label{schoen-conjecture}
Schoen asked whether every quasisymmetric homeomorphism of the ideal
boundary of the hyperbolic plane has a unique quasiconformal harmonic
extension \cite{SchoenHarmonic}. Akutagawa established existence for $C^4$ boundary data \cite{Akutagawa}, and Li--Tam existence for $C^1$ diffeomorphic boundary data and uniqueness in the quasiconformal
harmonic class \cite{LiTamProper,LiTamProperII}. Wan's description by
bounded holomorphic Hopf differentials linked the problem to universal
Teichm\"uller space \cite{WanHarmonic}. Markovi\'c resolved the general
existence problem \cite{Markovic}. His principal estimate controls the
hyperbolic distance between a harmonic quasiconformal map and any
quasiconformal comparison map with the same boundary values, using only
the dilatation of the comparison map.

Higher-dimensional extensions were obtained by Lemm--Markovi\'c
\cite{LemmMarkovic}, and Benoist--Hulin treated quasi-isometries between
rank-one symmetric spaces \cite{BenoistHulin}. Sidler--Wenger extended
the target setting to Gromov hyperbolic CAT(0) spaces
\cite{SidlerWenger}. Riestenberg--Smillie's Theorem~7 recovers the
earlier harmonic quasi-isometry results within a broader comparison
theory \cite{RiestenbergSmillie}.

We give a planar proof of that estimate and then pass to arbitrary
quasisymmetric data through the smooth quadratic minimisers constructed
above. The key comparison is the interior argument of
Riestenberg--Smillie \cite[Proposition 32 and Lemmas 33--34]{RiestenbergSmillie}.
In the disk its strict averaged-distance hypothesis follows from Mori's
theorem and the mean-value property of a logarithm. The proof below
establishes the displacement estimate without invoking Markovi\'c's
theorem; the square-root gradient estimate and the four-point comparison
really belong to Riestenberg--Smillie. Our contribution here is the planar verification of that hypothesis,
the variational construction of the smooth approximants, and explicit
constants throughout this comparison. The approximation--displacement--
compactness proof architecture is Markovi\'c's.

Throughout this section,
\[
 g_{\mathbb H}=4\eta(z)|dz|^2,\qquad
 d=d_{\mathbb H},\qquad B_R(x)=B_R^{g_{\mathbb H}}(x).
\]
Thus the curvature is $-1$, and $B_R(0)$ is the Euclidean disk of radius
$\tanh(R/2)$. Dilatations $K$ are operator-norm dilatations. Write
$\omega_{x,R}$ for harmonic measure at $x$ on $\partial B_R(x)$; after
centring $x$ it is $dt/(2\pi)$ on the corresponding Euclidean circle.

\subsection{Boundary distortion gives a strict averaged distance increase}
Mori's theorem \cite{Mori1956} gives, for every $K$-quasiconformal
$G:\D\to\D$ with $G(0)=0$,
\begin{equation}\label{schoen-mori}
 \left(\frac{|z-w|}{16}\right)^K
 \leq |G(z)-G(w)|\leq16|z-w|^{1/K},
 \qquad z,w\in\overline\D.
\end{equation}
The lower bound is the upper bound for $G^{-1}$. Both extend to the
boundary by continuity. Set
\begin{equation}\label{schoen-boundary-constants}
 \begin{split}
 b_K&=K\log16+\log2,\qquad c_K=\log32+2K\log16,\\
 R_K&=K(1+c_K)+\log2.
 \end{split}
\end{equation}

\begin{lemma}[Coarse control and positive drift]\label{schoen-drift}
Let $g:\D\to\D$ be $K$-quasiconformal. Then
\begin{equation}\label{schoen-coarse-upper}
 d(g(x),g(z))\leq Kd(x,z)+b_K.
\end{equation}
Moreover, for every $x,y\in\D$,
\begin{equation}\label{schoen-strict-drift}
 \int_{\partial B_{R_K}(x)}
 \bigl[d(y,g(z))-d(y,g(x))\bigr]\,d\omega_{x,R_K}(z)\geq1.
\end{equation}
The radius and the margin depend only on $K$.
\end{lemma}
\begin{proof}
Use source and target M\"obius isometries to arrange $x=g(x)=0$, and
denote the resulting map by $G$. If $|z|=r$, apply the upper bound in
\eqref{schoen-mori} for $G^{-1}$ to $G(z)$ and a nearest boundary point.
This gives
\[
 1-r\leq16(1-|G(z)|)^{1/K},\qquad
 d(0,G(z))\leq K\log\frac1{1-r}+b_K
             \leq Kd(0,z)+b_K.
\]
This proves \eqref{schoen-coarse-upper}.

For $\omega\in\Sone$ put $\xi=G^{-1}(\omega)$ and use the Busemann
function, normalised at $0$,
\[
 B_\omega(v)=\log\frac{|v-\omega|^2}{1-|v|^2}.
\]
The two Mori estimates imply
\[
 \log|G(z)-\omega|\geq K\log|z-\xi|-K\log16,
 \qquad 1-|G(re^{it})|^2\leq32(1-r)^{1/K}.
\]
Since $\log|z-\xi|$ is harmonic in $\D$ and is zero at $0$,
\begin{equation}\label{schoen-mori-mean-drift}
 \frac1{2\pi}\int_{-\pi}^{\pi}B_\omega(G(re^{it}))\,dt
 \geq\frac1K\log\frac1{1-r}-c_K.
\end{equation}
For $r=\tanh(R_K/2)$, the right side is at least $1$, because
$\log(1/(1-r))=R_K-\log(1+r)\geq R_K-\log2$.

Undo the normalisations. Choose $\omega$ beyond $y$ on the geodesic
ray starting at $g(x)$; if $y=g(x)$ any ray suffices. The triangle
inequality, followed by the limit defining its Busemann function, gives
\[
 d(y,g(z))-d(y,g(x))
 \geq B_\omega(g(z))-B_\omega(g(x)).
\]
Averaging proves \eqref{schoen-strict-drift}. The measure is the same
for every $y$, which will allow us to combine this inequality with
subharmonicity of a different distance function.
\end{proof}

\subsection{The Riestenberg--Smillie comparison in the disk}
\paragraph{Overview of the comparison.}
Near a maximum displacement $D$ between a harmonic map and a coarse
comparison map, the harmonic image lies in a ball of radius $D+E$.
Subtracting the central squared distance in Green's formula leaves
$2DE+E^2$, so the local gradient is $O(\sqrt D)$. The four-point
Ptolemy inequality then bounds the sum of two crossed-distance
increments by $O(D^{-1/2})$. One increment has average at least $1$
by \eqref{schoen-strict-drift}; the other has nonnegative average by
harmonicity and convexity of distance. This bounds $D$.
This is the mechanism of \cite[Section 4.1]{RiestenbergSmillie}.
Our smooth maps allow the following classical planar calculation.

\begin{lemma}[The square-root gradient estimate]\label{schoen-rs-gradient}
Let $h$ be hyperbolic harmonic on a neighbourhood of
$\overline{B_{R+2}(x)}$. Suppose, for some target point $y$, that
\[
 d(y,h(x))=D\geq1,\qquad d(y,h(z))\leq D+E
 \quad(z\in B_{R+2}(x)).
\]
Put
\[
 r_i=\tanh((R+i)/2),\quad
 \ell_R=\log(r_2/r_1),\quad a_0=(1-e^{-1})/2.
\]
Then
\begin{equation}\label{schoen-rs-gradient-bound}
 \sup_{B_R(x)}|dh|_{g_{\mathbb H},\mathrm{op}}
 \leq\Gamma(R,E)\sqrt D,\qquad
 \Gamma(R,E)^2=\frac{2E+E^2}{4a_0^2\ell_R}.
\end{equation}
\end{lemma}
\begin{proof}
Centre $x$ at $0$, use the Euclidean source metric, and set
$q=d(y,h)^2$ and
$e_E=\tfrac12|dh|_E^2=4\eta(h)(|h_z|^2+|h_{\bar z}|^2)$,
with the hyperbolic metric on the target and the usual tensor norm
in $|dh|_E^2$. The squared-distance Hessian inequality and Bochner's
formula give $\Delta_Eq\geq4e_E$ and $\Delta_Ee_E\geq0$.
Green's formula yields
\begin{equation}\label{schoen-green-cancellation}
 \frac{2\ell_R}{\pi}\int_{|z|<r_1}e_E\area
 \leq\frac1{2\pi}\int_{-\pi}^{\pi}q(r_2e^{it})\,dt-q(0)
 \leq2DE+E^2.
\end{equation}
For $|z|=t\leq r_0$, the Euclidean disk of radius
$a_0(1-t^2)$ centred at $z$ lies in $\{|z|<r_1\}$. Indeed
$(r_1-t)/(1-t^2)$ decreases on $[0,r_0]$, and its value at $r_0$ is
$\tanh(1/2)/(1+r_0\tanh(1/2))\geq a_0$.
The mean-value inequality and the source conformal factor therefore give
\[
 |dh(z)|_{g_{\mathbb H},\mathrm{op}}^2
 \leq\frac{(1-t^2)^2}{2}e_E(z)
 \leq\frac{2DE+E^2}{4a_0^2\ell_R}
 \leq\Gamma(R,E)^2D.
\]
Keeping the factor $(1-t^2)^2$ avoids an unnecessary loss near the
boundary of the centred disk.
\end{proof}

\begin{lemma}[Preliminary finiteness of displacement]\label{schoen-finite-displacement}
Two quasiconformal homeomorphisms of the disk having the same boundary values
are at bounded hyperbolic distance. We do not assert a
bound depending on only one of their dilatations.
\end{lemma}
\begin{proof}
If $u$ is $L$-quasiconformal and fixes $\Sone$ pointwise, conjugate by
a disk automorphism to centre any source point, and rotate so that
$u(0)=a\in[0,1)$. The map $T_a\circ u$ fixes $0$, where
$T_a(z)=(z-a)/(1-az)$. At $\xi_\pm=T_a^{-1}(\pm i)$ its boundary
values are $\pm i$. Mori's theorem gives
\[
 2\leq16|\xi_+-\xi_-|^{1/L},\qquad
 |\xi_+-\xi_-|=\frac{2(1-a^2)}{1+a^2}
              =\frac2{\cosh d(0,a)}.
\]
Thus $\sup_xd(x,u(x))\leq\operatorname{arcosh}(2\cdot8^L)$.
Apply this to $u=g^{-1}\circ h$ and then use
\eqref{schoen-coarse-upper} for $g$.
\end{proof}

\begin{theorem}[Uniform displacement]\label{schoen-displacement}
Let $g:\D\to\D$ be $K$-quasiconformal and let $h:\D\to\D$ be a
harmonic quasiconformal homeomorphism with the same boundary values.
There is a bound depending only on $K$,
\begin{equation}\label{schoen-displacement-bound}
 \boxed{\quad\sup_{z\in\D}d(h(z),g(z))\leq C_{\mathrm{disp}}(K).\quad}
\end{equation}
An explicit choice is
\begin{equation}\label{schoen-displacement-constants}
 \begin{split}
 E_K&=K(R_K+2)+b_K+1,\qquad
 \Gamma_K=\Gamma(R_K,E_K),\\
 C_{\mathrm{disp}}(K)
 &=\max\{1,\,4(KR_K+b_K)^2R_K^2\Gamma_K^2\}.
 \end{split}
\end{equation}
The constants do not involve $K_h$ or the total excess energy of $h$.
\end{theorem}
\begin{proof}
Let $S=\sup d(h,g)<\infty$, by Lemma~\ref{schoen-finite-displacement}.
If $S\leq1$ there is nothing to prove. Otherwise choose $x$ with
$D=d(h(x),g(x))\geq1$ and $S\leq D+\epsilon$, where $0<\epsilon\leq1$.
For $z\in B_{R_K+2}(x)$, \eqref{schoen-coarse-upper} gives
$d(g(x),h(z))\leq D+E_K$. Lemma~\ref{schoen-rs-gradient} applies.

For $z\in B_{R_K}(x)$ write
\[
 \begin{aligned}
 D'&=d(h(z),g(z)),& E&=d(g(x),g(z)),& E'&=d(h(x),h(z)),\\
 F&=d(h(x),g(z)),& F'&=d(g(x),h(z)).
 \end{aligned}
\]
The CAT(0) Ptolemy inequality \cite{FLS} gives $FF'\leq DD'+EE'$.
These are ordinary distances; the classical hyperbolic Ptolemy formula
with hyperbolic sines is a different statement.
Consequently, using the triangle inequalities,
\[
 \begin{split}
 D(F-D+F'-D')
 &\leq(F-D)(D-F')+EE'\leq2EE',\\
 F-D+F'-D
 &\leq\frac{2(KR_K+b_K)R_K\Gamma_K}{\sqrt D}+\epsilon.
 \end{split}
\]
Here $D'\leq D+\epsilon$, $E\leq KR_K+b_K$, and the gradient bound
integrated along a source geodesic gives $E'\leq R_K\Gamma_K\sqrt D$.
Average on $\partial B_{R_K}(x)$. Equation~\eqref{schoen-strict-drift}
with $y=h(x)$ gives $\int(F-D)\,d\omega_{x,R_K}\geq1$.
Since $d(g(x),h(\cdot))$ is subharmonic,
$\int(F'-D)\,d\omega_{x,R_K}\geq0$. Therefore
\[
 1\leq\frac{2(KR_K+b_K)R_K\Gamma_K}{\sqrt D}+\epsilon.
\]
Choose almost-maximal points with $\epsilon\downarrow0$, so that
$D\to S$. This proves \eqref{schoen-displacement-bound}.
\end{proof}

\begin{remark}[Relation to the general comparison]
This is the interior comparison of \cite[Proposition 32]{RiestenbergSmillie},
with an almost maximum in place of an attained maximum. The additional
$1$ in $E_K$ absorbs that change, and the final $\epsilon$ disappears
in the limit. Our approximants already exist on the whole disk, so no
boundary estimate for an exhausting sequence of bounded domains is
needed. The comparator need only be quasiconformal: the argument uses
\eqref{schoen-coarse-upper}, not its derivatives. Individual
quasiconformality of $h$ supplies the preliminary finite $S$; it is not
being used as a uniform bound in the comparison. Hopf coordinates and
their zeros do not enter this step.
\end{remark}

\subsection{Uniform energy density and quasiconformality}
\begin{proposition}[Obtaining a quasiconformal bound]
\label{schoen-uniform-qc}
Under the hypotheses of Theorem~\ref{schoen-displacement}, put
\begin{equation}\label{schoen-qc-constants}
 \begin{split}
 T_K&=2C_{\mathrm{disp}}(K)+K\log3+b_K,\\
 M_K&=\max\{1,2T_K^2/\log2\},\qquad
 \delta_K=e^{-2M_K}/M_K,\\
 Q_K&=\frac{1+\sqrt{1-\delta_K}}{1-\sqrt{1-\delta_K}}.
 \end{split}
\end{equation}
Then,  
\begin{equation}\label{schoen-density-and-gap}
 \begin{gathered}
 1\leq\mathcal H=\frac{\eta(h)|h_z|^2}{\eta},\qquad
 \mathcal H+\mathcal L\leq M_K,\qquad
 \mathcal L=\frac{\eta(h)|h_{\bar z}|^2}{\eta},\\
 1-|\mu_h|^2\geq\delta_K,\qquad
 K_h\leq Q_K\leq4M_Ke^{2M_K},\qquad
 \sup_\D\frac{|\Phi_2|}{\eta}\leq M_K/2.
 \end{gathered}
\end{equation}
\end{proposition}
\begin{proof}
The harmonic homeomorphism is a diffeomorphism by the metric Lewy
theorem \cite{Martin}; for our minimisers this is already supplied by
Proposition~\ref{quadratic-minimiser-regularity}.
The displacement and coarse bounds put $h(B_{\log3}(x))$ in the
target ball of radius $T_K$ about $h(x)$. Centre $x$ at $0$, so the
source ball has Euclidean radius $1/2$. Apply Green's formula to
$q=d(h(0),h)^2$, whose central value is zero, and then the mean-value
inequality to $e_E=\tfrac12|dh|_E^2$ on radius $1/4$:
\[
 \int_{|z|<1/4}e_E\area\leq\frac{\pi T_K^2}{2\log2},\qquad
 e_E(0)\leq\frac{8T_K^2}{\log2}.
\]
Since $e_E(0)=4(\mathcal H+\mathcal L)(0)$, this proves the density
bound at every centre. Completeness of $\mathcal H g_{\mathbb H}$
and the Schwarz comparison used in
Proposition~\ref{quadratic-density-qc-bound} give $\mathcal H\geq1$.
Apply Proposition~\ref{quadratic-density-qc-bound} with $B=M_K$ to obtain $1-|\mu_h|^2\geq\delta_K$ and the asserted bound for $K_h$. Finally,
$|\Phi_2|/\eta=\sqrt{\mathcal H\mathcal L}\leq(\mathcal H+\mathcal L)/2$.
\end{proof}

\subsection{Passage from smooth minimisers to the prescribed boundary map}
\begin{lemma}[Smooth approximation with a fixed quasisymmetry bound]
\label{smooth-uniform-qs-approximation}
Every orientation-preserving quasisymmetric circle homeomorphism
is a uniform limit of smooth orientation-preserving circle
diffeomorphisms with a common quasisymmetry bound.
In the angular adjacent-interval convention of the introduction,
the original constant $M$ can be retained.
\end{lemma}
\begin{proof}
Let $u$ be the increasing lift, and convolve it on the real line
with a nonnegative smooth approximate identity $\rho_\epsilon$
which is positive in a neighbourhood of zero and has compact
support. Then
$u_\epsilon(\theta+2\pi)=u_\epsilon(\theta)+2\pi$.
The distributional measure $du$ gives positive mass to every open
interval, so $u_\epsilon'=\rho_\epsilon*du>0$.
Uniform continuity of $u(\theta)-\theta$ gives uniform convergence
of the resulting circle maps.
For every $\theta$ and $0<t\leq\pi$, the original inequalities
\[
 M^{-1}[u(\theta)-u(\theta-t)]
 \leq u(\theta+t)-u(\theta)
 \leq M[u(\theta)-u(\theta-t)]
\]
are linear inequalities between increments. Averaging their
translates against $\rho_\epsilon$ preserves both inequalities
with the same $M$. This proves the assertion. The corresponding
quasiconformal extension constants can consequently be bounded
by one $K_0=K_0(M)$.
\end{proof}
\begin{lemma}[Smooth boundary data admit finite-energy barriers]\label{smooth-barrier}
Every smooth orientation-preserving circle diffeomorphism has a smooth quasiconformal extension $F:\overline\D\to\overline\D$ with $|\mu_F(z)|=O(1-|z|)$ as $|z|\to1$. In particular $\I_p(F)<\infty$ for every $p>1$.
\end{lemma}
\begin{proof}
Write the lift of the boundary map as $u(\theta)$, with $u'>0$. On a sufficiently thin collar set
\[
 F(re^{i\theta})=\bigl[1-(1-r)u'(\theta)\bigr]e^{iu(\theta)}.
\]
This is an orientation-preserving smooth collar embedding. At $r=1$ its radial and unit tangential derivatives are $u'e^{iu}$ and $iu'e^{iu}$, so $F_{\bar z}=0$ there and $|\mu_F|=O(1-r)$. The collar extends to a smooth disk diffeomorphism: apply the same formula to the isotopy $u_t=(1-t)\theta+tu(\theta)$, extend its velocity field smoothly to the disk using a cutoff which is one near the circle, and take the time-one flow. On a thin collar this flow agrees with the formula above. The resulting extension is quasiconformal, and its energy on the collar is bounded by a constant times $\int_0^\epsilon s^{p-2}\,ds<\infty$ for $p>1$; its energy on the remaining compact disk is finite as well.
\end{proof}

\begin{theorem}[Schoen's theorem through quadratic minimisers]
\label{schoen-through-minimisers}
Every orientation-preserving quasisymmetric circle homeomorphism
$\beta$ has a unique quasiconformal hyperbolic harmonic extension
$h:\D\to\D$. It is a real analytic diffeomorphism. More precisely,
let $\beta_j\to\beta$ be the smooth approximations of
Lemma~\ref{smooth-uniform-qs-approximation}, and let $g_j$ be
quasiconformal extensions with $K(g_j)\leq K_0$. The quadratic
minimisers $h_j$ with these inverse traces satisfy
\[
 \sup_j\sup_\D d(h_j,g_j)\leq C_{\mathrm{disp}}(K_0),\qquad
 K(h_j),K(h)\leq Q_{K_0}.
\]
They converge to $h$ uniformly on the closed disk and smoothly on
compact subsets of $\D$. Their inverses converge uniformly to $h^{-1}$.
\end{theorem}
\begin{proof}
Each smooth $\beta_j$ has a finite-energy barrier: apply Lemma~\ref{smooth-barrier} to the forward trace $\beta_j^{-1}$.
Propositions~\ref{quadratic-direct-existence} and
\ref{quadratic-minimiser-regularity} give us the quadratic
minimiser $h_j$, a hyperbolic harmonic diffeomorphism.
Proposition~\ref{quadratic-energy-qc-bound} makes it individually quasiconformal. Li--Tam uniqueness identifies this variational extension with their classical smooth-data harmonic extension.
The forward map is $f_j=h_j^{-1}$, and
\[
 \I_2(f_j)=\int_\D
 \frac{|\mu_{f_j}|^2}{1-|\mu_{f_j}|^2}\eta\area
 =\E_2(h_j)=\int_\D\eta(h_j)|(h_j)_{\bar z}|^2\area.
\]
These finite energies are not assumed uniformly bounded.

It is Theorem~\ref{schoen-displacement} and
Proposition~\ref{schoen-uniform-qc} that give the asserted uniform bounds.
The remaining limiting step is now the classical compactness theorem for
uniformly quasiconformal harmonic disk maps; see
\cite{YaoHarmonicCompactness}, including its $C^\infty_{\loc}$
strengthening. We recall its application to fix the normalisation.
After three-point normalisation, quasiconformal compactness for maps
and inverses gives uniform convergence on the closed disk, along a
subsequence, to inverse homeomorphisms with traces $\beta$ and
$\beta^{-1}$. The normalisation excludes a constant limit; because the
limiting boundary map is a homeomorphism, the normalisations converge
and can be undone. Harmonic compactness gives smooth local convergence
to a harmonic $Q_{K_0}$-quasiconformal map. Its Jacobian is positive:
the smaller hyperbolic singular value of each $h_j$ is
$\sqrt{\mathcal H_j}(1-|\mu_{h_j}|)
\geq1-\sqrt{1-\delta_{K_0}}>0$.
The analytic harmonic-map equation gives real analyticity.
Uniqueness in the quasiconformal harmonic class
\cite[Section 1]{LiTamProperII} identifies every subsequential limit,
so the whole sequence converges as stated.
\end{proof}

\begin{remark}[Scope and the size of the constants]\label{schoen-constants-scope}
Of course the limit will not, in general, have finite $\E_2$. Once the smooth
minimisers are harmonic, the displacement comparison uses harmonicity
and the boundary distortion bound alone.

The constants in \eqref{schoen-boundary-constants},
\eqref{schoen-displacement-constants} and \eqref{schoen-qc-constants}
are completely explicit once the comparator dilatation $K$ is given.
They provide quantitative bounds for displacement, energy density,
the Bloch norm of the Hopf differential and quasiconformal dilatation,
with no unspecified constants from a compactness argument. These estimates, despite being large, are already
effective. For example
$\ell_R\geq e^{-(R+1)}$, so
\[
 C_{\mathrm{disp}}(K)
 \leq\max\left\{1,
 \frac{e^{R_K+1}}{a_0^2}{(KR_K+b_K)^2R_K^2(2E_K+E_K^2)} 
\right\}.
\]
Since $R_K=O(K^2)$, this is at most $A\exp(BK^2)$ for universal
constants $A,B$. The subsequent bound $Q_K$ is substantially larger.
This asymptotic description is only a summary of the explicit formulas
above. Improving their size, especially near conformal boundary data,
is a separate problem. One may already optimise the
comparison over radii $R$ with
$R/K-(\log2)/K-c_K>0$, retaining this positive margin in place of $1$.

The constants do not reflect the behaviour near conformal data: $C_{\mathrm{disp}}(K)\geq1$ by construction, so $Q_K$ stays bounded away from $1$ as $K\to1$. Nevertheless the dilatation of the harmonic extension does tend to $1$ as the quasisymmetry constant $M$ of the boundary map tends to $1$, by a compactness argument with no rate. Suppose $M_n\to1$ but $K(h_n)\geq K_0>1$ for the harmonic quasiconformal extensions $h_n$. Normalise by M\"obius maps so that $|\mu_{h_n}(0)|\geq k_0>0$ and the boundary maps fix three points; both constants are M\"obius invariant. The boundary maps are $M_n$-quasisymmetric and normalised, hence equicontinuous, and every limit is $1$-quasisymmetric, i.e.\ M\"obius, i.e.\ the identity; so they converge uniformly to the identity. The $h_n$ are uniformly $K$-quasiconformal with $K=Q_{K(\sup_nM_n)}$, so a subsequence converges locally uniformly to a $K$-quasiconformal map $h$ with boundary values the identity, and by Cheng's gradient estimate the convergence is smooth on compact sets, so $h$ is harmonic. Li--Tam uniqueness gives $h=\mathrm{id}$, and then $\mu_{h_n}(0)\to0$, a contradiction. A quantitative version would be a bound for the hyperbolic supremum norm of the Hopf differential in terms of $M-1$, through Wan's criterion \cite{WanHarmonic}.  However we do not know one.
\end{remark}
\begin{proof}[Proof of Theorem~\ref{thm-B}]
Existence, uniqueness, real analyticity and the convergence statement are Theorem~\ref{schoen-through-minimisers}. For the displacement and dilatation bounds, let $g$ be any $K$-quasiconformal extension of $\beta$. The harmonic extension $h$ is quasiconformal, so Theorem~\ref{schoen-displacement} applies to the pair $(h,g)$ directly and gives $\sup d(h,g)\leq C_{\mathrm{disp}}(K)$; Proposition~\ref{schoen-uniform-qc} then gives $K(h)\leq Q_K$.
\end{proof}

\subsection{A local energy comparison with a displacement error}\label{quadratic-geodesic-comparison}
The displacement bound of Theorem~\ref{schoen-displacement} was converted into an energy-density bound by a Green's formula on one ball. The following intrinsic comparison, which freezes neither the target metric nor a patch, gives a second passage from displacement to energy density and may be of independent use. All quantities are measured with the source and target metric $\eta|dz|^2$. For a smooth map $F$ put
\[
 e_F=\tfrac12|dF|_\eta^2
 =\frac{\eta(F)}{\eta}
       (|F_w|^2+|F_{\bw}|^2),
\]
so that $e_h=\mathcal H+\mathcal L$ in the notation of Proposition~\ref{quadratic-energy-qc-bound}.

\begin{proposition}[Local energy comparison with a displacement error]
\label{quadratic-geodesic-caccioppoli}
Let $h:\D\to\D$ be hyperbolic harmonic and let $G:\D\to\D$ be a
smooth comparison map. Put $d(w)=d_\eta(h(w),G(w))$.
For every compactly supported Lipschitz function $0\leq\chi\leq1$,
\begin{equation}\label{quadratic-geodesic-energy}
 \boxed{\quad
 \int_\D\chi^2 e_h\,dA_\eta
 \leq 2\int_\D\chi^2 e_G\,dA_\eta
       +4\int_\D|\nabla_\eta\chi|^2d^2\,dA_\eta .
 \quad}
\end{equation}
Neither map is assumed to have finite total hyperbolic energy.
In particular,
\begin{equation}\label{quadratic-annular-displacement}
 \int_{B_R^\eta(a)}e_h\,dA_\eta
 \leq 2\int_{B_{2R}^\eta(a)}e_G\,dA_\eta
       +\frac4{R^2}
         \int_{B_{2R}^\eta(a)\setminus B_R^\eta(a)}
                          d^2\,dA_\eta .
\end{equation}
\end{proposition}
\begin{proof}
The hyperbolic plane has a unique geodesic between any two points.
Let $V(w)=\log_{h(w)}G(w)$ be its initial velocity, with
$|V|_\eta=d$. Test harmonicity of $h$ with the compactly supported
field $\chi^2V$; equivalently use the variation
$\exp_h(t\chi^2V)$ at $t=0$.

For target points $p,q$, the differential of $\log_p q$ with
respect to $p$ is minus the Hessian of
$\tfrac12d_\eta(p,q)^2$. In curvature $-4$ its radial and transverse
eigenvalues are $-1$ and $-2d\coth(2d)$.
With respect to $q$, after parallel transport to $p$, the
corresponding eigenvalues are $1$ and $2d/\sinh(2d)$.
Their limiting values at $d=0$ are understood.
Consequently, for each source tangent vector $X$,
\[
 \langle dh(X),\nabla_X V\rangle_\eta
 \leq-|dh(X)|_\eta^2
       +|dh(X)|_\eta\,|dG(X)|_\eta .
\]
Sum over a source orthonormal frame and use the first variation.
Writing norms of differentials as real Hilbert--Schmidt norms,
one obtains
\[
 \int\chi^2|dh|_\eta^2
 \leq\int\chi^2|dh|_\eta|dG|_\eta
       +2\int\chi|\nabla_\eta\chi|\,|dh|_\eta\,d .
\]
All integrals here use $dA_\eta$.
Young's inequality bounds the right side by
\[
 \tfrac12\int\chi^2|dh|_\eta^2
       +\int\chi^2|dG|_\eta^2
       +4\int|\nabla_\eta\chi|^2d^2.
\]
Absorb the first term and divide by two to obtain
\eqref{quadratic-geodesic-energy}.
Take the radial cutoff equal to one on $B_R^\eta(a)$, decreasing
linearly to zero on the surrounding annulus, with
$|\nabla_\eta\chi|\leq R^{-1}$, to prove
\eqref{quadratic-annular-displacement}.
The logarithm field is smooth even at coincident points, and
Lipschitz cutoffs follow by approximation.
\end{proof}
\section{Integrable distortion and the boundary classes}\label{boundary-classes}
This section proves Theorem~\ref{thm-C}. The order of the argument is as follows. A boundary regularity criterion for finite-distortion maps with hyperbolically integrable coefficient (Section~\ref{boundary-criterion}); the classes $T^p$ and $\mathcal T_p$ (Section~\ref{two-classes}); a factorisation of every finite-energy map through a quasiconformal map, which gives the homeomorphic boundary extension (Section~\ref{factorisation-section}); the Sobolev regularity of the inverse map and a capacity estimate giving quantitative quasisymmetry and symmetry (Section~\ref{boundary-moduli}); and finally the identity $\mathcal T_2=T^2$ (Section~\ref{wp-characterisation}). Except for the last step the results hold for every $p\geq1$; the identity $\mathcal T_p=T^p$ for every $p>1$ is proved in Section~\ref{p-classes} by a different method using only the homeomorphic boundary extension and the capacity estimate of Lemma~\ref{finite-distortion-inverse-calculus}.

\subsection{A boundary regularity criterion}\label{boundary-criterion}
We use the disk version of the theorem of Gutlyanskii, Sakan and Sugawa \cite[Theorem 1.7]{Gut}. This theorem concerns a $\mu$-conformal \emph{homeomorphism} of the disk onto itself. Here, the assumption that the map is a homeomorphism is necessary.

\begin{theorem}\label{extend}
Let $f:\ID\to\ID$ be a homeomorphism of finite distortion in the class to which \cite[Theorem 1.7]{Gut} applies. If
\begin{equation}\label{muint}
 \int_\ID\frac{|\mu_f|}{1-|\mu_f|}\,dA_\eta<\infty,
\end{equation}
then $f$ extends homeomorphically to $\overline\ID$, and its boundary map is a $C^1$ diffeomorphism.
\end{theorem}
\begin{proof}
For $\zeta\in\IS$ let
\[
 T(\zeta;r_1,r_2)=\left\{z\in\ID:r_1<\left|\frac{z-\zeta}{z+\zeta}\right|<r_2\right\}.
\]
The two integrands in the cited criterion are, in absolute value, bounded by the integrable function in \eqref{muint}, since
\[
 |z^2-\zeta^2|\geq1-|z|^2,\qquad
 \frac{|\mu|^2}{1-|\mu|^2}\leq\frac{|\mu|}{1-|\mu|},\qquad
 \frac{|\mu|}{1-|\mu|^2}\leq\frac{|\mu|}{1-|\mu|}.
\]
Moreover $T(\zeta;r_1,r_2)\subset B(\zeta,2r_2)$, so its area is at most $4\pi r_2^2$, uniformly in $\zeta$. Absolute continuity of the integral gives uniform convergence in both conditions. The cited theorem supplies continuity and strict positivity of the boundary derivative. \end{proof}

Theorem~\ref{extend} is a boundary version of the Teichm\"uller--Wittich--Belinskii theorem, and it is worth saying what separates the two. That theorem is an interior, pointwise statement about quasiconformal maps: if $f$ is quasiconformal near $0$ with $f(0)=0$ and $\int_{|z|<1}|\mu_f|\,|z|^{-2}\,dA<\infty$, then $f(z)/z$ has a nonzero limit at $0$. Three things are different here. First, the point lies on the boundary of a self-map of the disk, where reflection across the circle is not available for a $\mu$-homeomorphism (the reflected map need not be in $W^{1,1}_{\loc}$ across the circle, \cite[Example 4.6]{Gut}), so the modulus estimates must be made one-sidedly, for semiannuli, and the conclusion is differentiability of the boundary map, with derivative $\lim_{z\to\zeta}|f(z)-f(\zeta)|/|z-\zeta|>0$, rather than conformality of $f$ at $\zeta$. Second, the map is only a $\mu$-homeomorphism in $W^{1,1}_{\loc}$ with $\|\mu\|_\infty=1$ allowed, which is what makes the result available for finite-distortion maps; the hypotheses are the two conditions of Brakalova--Jenkins type, with the denominators $1-|\mu|^2$ carrying the degeneration, and the single Wittich condition is meaningless where $|\mu|\to1$. Third, pointwise conditions give pointwise differentiability only; the continuity of the derivative needs the integrals to converge uniformly in $\zeta$, and that uniformity is what \eqref{muint} provides through the absolute continuity of the integral, the sets $T(\zeta;r_1,r_2)$ having uniformly small area.

\begin{proposition}\label{inverseboundary}
Let $H:\ID\to\ID$ be a homeomorphism of finite distortion. Assume that $F=H^{-1}$ is also of finite distortion and that the inverse and area formulas apply. If
\begin{equation}\label{hypint}
 \int_\ID\eta(H)\,|H_wH_{\bar w}|\,dA(w)<\infty,
\end{equation}
then both boundary maps are $C^1$ diffeomorphisms.
\end{proposition}
\begin{proof}
The inverse formula gives $|\mu_F(H(w))|=|\mu_H(w)|$. Therefore
\[
 \int_\ID\frac{|\mu_F|}{1-|\mu_F|^2}\,dA_\eta
 =\int_\ID\eta(H)\,|H_wH_{\bar w}|\,dA.
\]
Since $(1-r)^{-1}\leq2(1-r^2)^{-1}$ for $0\leq r<1$, Theorem~\ref{extend} applies to $F$. Its boundary inverse is also a $C^1$ diffeomorphism.
\end{proof}

\begin{corollary}\label{harmonicboundary}
Let $H:\ID\to\ID$ be an orientation-preserving hyperbolic harmonic diffeomorphism onto the disk. If its Hopf coefficient
\[
 Q_H=\eta(H)H_w\overline{H_{\bar w}}
\]
belongs to $L^1(\ID,dA)$, then its boundary map is a $C^1$ diffeomorphism.
\end{corollary}
\begin{proof}
Apply Proposition~\ref{inverseboundary}. The inverse and area formulas hold for the diffeomorphism $H$.
\end{proof}

Harmonicity alone does not imply that a map is a homeomorphism: constant maps and $H(w)=w^2$ are immediate counterexamples, with zero Hopf differential. The degree or injectivity assumptions in results such as \cite{Martin} must be retained. Conversely, the existence of a harmonic extension with \emph{integrable} Hopf coefficient is not established here from $C^1$ boundary regularity alone.
\subsection{The classes $T^p$ and $\mathcal T_p$}\label{two-classes}
For $p\geq1$, let $T^p$ consist of the circle homeomorphisms admitting a quasiconformal extension $F:\ID\to\ID$ with $\mu_F\in L^p(dA_\eta)$. The spaces were introduced in \cite{Guo}; for $p=2$ the Douady--Earle extension preserves this integrability \cite{Cui}. The relation with the Weil--Petersson Hilbert-manifold structure is developed in \cite{TT}. Elements of $T^p$ are symmetric, and for $0<p\leq1$ the classical spaces consist of $C^1$ diffeomorphisms \cite{B5,AB}. The converse inclusion from all $C^1$ diffeomorphisms is false even for $T^2$; see the example below and \cite{Shen}.

The coefficient and boundary classes under discussion are
\begin{align}
 \mathcal B_p&=\{\mu_F:F:\ID\to\ID\text{ a homeomorphism of finite distortion},\ \mathcal I_p(F)<\infty\},\label{Bp}\\
 \mathcal T_p&=\{g\in\mathrm{Homeo}^+(\IS):\exists F\text{ as above extending continuously with }F|_\IS=g\}.\label{Tpclass}
\end{align}
The existence of an integrating map is part of this definition of $\mathcal B_p$. This is a different problem to start with an arbitrary measurable coefficient satisfying only the integral bound as the next example shows.

\begin{example}[The coefficient condition alone does not give existence]\label{rawcoefficient}
For $0<r<e^{-1}$ put $t=\log(1/r)$ and set
\[
 \mu(re^{i\theta})=-\frac{t^2-1}{t^2+1}e^{2i\theta};
 \qquad \mu=0\quad\hbox{for }e^{-1}\leq r<1.
\]
The value at zero is immaterial. This coefficient satisfies the integral in \eqref{funct} for every finite $p\geq1$: its support stays away from the boundary and its integrand grows only like $t^{p_*}$ as $r\downarrow0$, whereas $r\,dr=e^{-2t}(-dt)$.

Nevertheless there is no homeomorphic disk solution with this coefficient. On the punctured disk it is the coefficient of
\[
 U(re^{i\theta})=R(r)e^{i\theta},\qquad
 R(r)=\begin{cases}\exp(-2+1/\log(1/r)),&0<r<e^{-1},\\r,&e^{-1}\leq r<1.\end{cases}
\]
Indeed $rR'/R=t^{-2}$ on the inner part. The map $U$ takes the punctured disk onto the annulus $e^{-2}<|w|<1$. If a homeomorphic solution $F:\ID\to\ID$ existed, $F\circ U^{-1}$ would be conformal from that finite-modulus annulus onto $\ID\setminus\{F(0)\}$, a punctured disk of infinite modulus. This is impossible.
\end{example}

For each fixed $p>0$,
\[
 \min\{1,2/p\}\leq\frac{1-r^2}{1-r^p}\leq\max\{1,2/p\}\qquad(0\leq r<1).
\]
Thus replacing the inner power $2$ in the denominator by $p$ gives an equivalent integrability condition, with constants depending on $p$.
\begin{lemma}\label{easyinclusion}
For $p\geq1$, $T^p\subset\mathcal T_p$.
\end{lemma}
\begin{proof}
For a quasiconformal representative $|\mu_F|\leq k<1$, the denominator in \eqref{funct} is bounded away from zero.
\end{proof}
\subsection{Coefficient factorisation}\label{factorisation-section}
\begin{lemma}[Quasiconformal postcomposition at integrable distortion]
\label{integrable-distortion-composition}
Let $G:\D\to\D$ be a finite-distortion homeomorphism with
$K_G\in L^1(\D)$, and let $A$ be a quasiconformal disk map.
Then $H=A\circ G$ belongs to $W^{1,1}(\D)$, has finite distortion,
and satisfies the usual derivative and Beltrami composition formulas
almost everywhere. Also
\[
 H^{-1}=G^{-1}\circ A^{-1}\in W^{1,2}(\D).
\]
\end{lemma}
\begin{proof}
Hencl--Koskela inverse regularity gives
$g=G^{-1}\in W^{1,2}(\D)$ of finite distortion
\cite[Theorems 1.2--1.3]{HenclKoskelaInverse}.
In particular $g$ has Lusin's property $N$
\cite[Theorem 3.3.7]{AIM}.
Quasiconformal precomposition preserves $W^{1,2}$, with
\[
\|g\circ A^{-1}\|_{W^{1,2}(\D)} \leq K(A)\|g\|_{W^{1,2}(\D)} ,
\]
by \cite[Theorem 3.8.2]{AIM}. Its chain rule and the $N$ and
$N^{-1}$ properties of $A$ show that $g\circ A^{-1}$ has finite
distortion: if its Jacobian vanishes, its full derivative vanishes.
Apply the inverse theorem once more. Its inverse $H$ has finite
distortion and belongs to $W^{1,1}_{\loc}$; the inverse derivative
estimate and the finite Dirichlet integral give the global
$W^{1,1}$ bound.

For the asserted chain rule, let $N_A$ be the null set where $A$
is not differentiable. Property $N$ of $g$ implies
$|G^{-1}(N_A)|=|g(N_A)|=0$.
The Gehring--Lehto theorem makes $G$ differentiable almost
everywhere \cite[Corollary 3.3.3]{AIM}. Thus the classical chain
rule applies almost everywhere to $A\circ G$, and these derivatives
are its Sobolev derivatives. At $DG=0$ this gives $DH=0$;
elsewhere it gives the Beltrami composition formula.
\end{proof}

Let $G:\ID\to\ID$ be a homeomorphism of finite distortion with $\mathcal I_p(G)<\infty$. Write $\mu=\mu_G$ and set
\[
 \nu=\mu|\mu|^{p-1}.
\]
For all $p\geq1$,
\begin{equation}\label{improvedintegrability}
 \frac{|\nu|}{1-|\nu|^2}
 =\frac{|\mu|^p}{1-|\mu|^{2p}}
 \leq\frac{|\mu|^p}{(1-|\mu|^2)^{p_*/2}}.
\end{equation}
Indeed, $1-r^{2p}\geq1-r^2$, and, for $p\geq2$, $1-r^2\geq(1-r^2)^{p/2}$.

The coefficient prescribed for a quasiconformal factor $A$ is
\begin{equation}\label{factorcoeff}
 \mu_A(G(z))=
 \frac{\nu(z)-\mu(z)}{1-\nu(z)\overline{\mu(z)}}
 \frac{G_z(z)}{\overline{G_z(z)}}.
\end{equation}
Choose Borel representatives of the derivatives of $G$, and put the right side equal to zero where those representatives are exceptional or $G_z=0$. Transporting this Borel function by the homeomorphism $G^{-1}$ defines a measurable coefficient on the whole disk. Lemma~\ref{integrable-distortion-composition} justifies the subsequent Sobolev composition. Its modulus satisfies
\begin{equation}\label{factorbound}
 |\mu_A(G(z))|=d_p(|\mu(z)|),\qquad
 d_p(r)=\frac{r(1-r^{p-1})}{1-r^{p+1}}
 \leq\frac{p-1}{p+1}.
\end{equation}
For $p>1$ this follows on writing $r=e^{-t}$ and using the increase of $\sinh x/x$:
\[
 d_p(e^{-t})=\frac{\sinh((p-1)t/2)}{\sinh((p+1)t/2)}.
\]
For $p=1$, the factor is conformal. The measurable Riemann mapping theorem therefore produces a $p$-quasiconformal disk map $A$ with the prescribed coefficient.

\begin{proposition}[Factorisation of every finite-energy extension]\label{factorisation}
Let $p\geq1$ and let $G:\D\to\D$ be a finite-distortion homeomorphism with $\I_p(G)<\infty$. There is a $p$-quasiconformal disk map $A$ for which $H=A\circ G$ has a $C^1$ diffeomorphic boundary value $h$. In particular $G$ extends homeomorphically to the closed disk, and
\[
 g=G|_\IS=a^{-1}\circ h,\qquad a=A|_\IS,
\]
is quasisymmetric.
\end{proposition}
\begin{proof}
Finite energy gives $K_G\in L^1(\D)$ as in Lemma~\ref{finite-distortion-inverse-calculus}. Construct $A$ from \eqref{factorcoeff}. Lemma~\ref{integrable-distortion-composition} proves that $H$ has finite distortion and that its coefficient is $\mu_H=\nu$ almost everywhere. By \eqref{improvedintegrability}, $\int |\mu_H|/(1-|\mu_H|^2)\,dA_\eta<\infty$. Apply Theorem~\ref{extend}, using the factor of at most $2$ between the two denominators. A $C^1$ orientation-preserving circle diffeomorphism is quasisymmetric, as are the boundary values of $A$ and $A^{-1}$. Their composition gives the assertion.
\end{proof}

The inverse regularity and its Lusin property justify factorisation for the entire finite-energy class. No Lusin $N$ property of $G$ itself was assumed.

For $p\ne2$, the stronger identification
\begin{equation}\label{classquestion}
 \mathcal T_p\stackrel{?}{=}T^p
\end{equation}
requires an additional argument; the quadratic case is settled in Section~\ref{wp-characterisation}, and the general case in Section~\ref{p-classes} by a different method. In particular, \eqref{factorbound} does not itself show $\mu_A\in L^p(dA_\eta)$. When change of variables is valid, the required integral is
\begin{equation}\label{badweight}
 \int_\ID|\mu_A(w)|^p\eta(w)\,dA(w)
 =\int_\ID d_p(|\mu_G(z)|)^p\eta(G(z))J_G(z)\,dA(z).
\end{equation}
The hypothesis controls a different weight, $\eta(z)$, and no comparison with $\eta(G)J_G$ is available. Nor does the fact that $h$ is a $C^1$ diffeomorphism alone give a quasiconformal extension with an $L^1(dA_\eta)$ Beltrami coefficient; the affirmative answer to \eqref{classquestion} in Section~\ref{p-classes} does not go through the factorisation.  
\subsection{Sobolev regularity of the inverse and quantitative boundary moduli}\label{boundary-moduli}
We now prove quantitative boundary quasisymmetry from the finite-energy condition. The factorisation in Proposition~\ref{factorisation}, proved above using only inverse Sobolev regularity, supplies the boundary extension when it has not been prescribed.

\begin{lemma}[Sobolev regularity of the inverse and the capacity estimate]
\label{finite-distortion-inverse-calculus}
Let $F:\D\to\D$ be a finite-distortion homeomorphism with
$\I_p(F)<\infty$, $p\geq1$, and let $h=F^{-1}$.
Then $h\in W^{1,2}(\D)$ and has finite distortion. For $1\leq p\leq2$,
\begin{equation}\label{finite-distortion-inverse-energy}
 \|h\|_{W^{1,2}(\D)}\leq\int_\D\K_F\,dA\leq\pi+2\I_p(F),
 \qquad \E_p(h)\leq\I_p(F).
\end{equation}
For a scalar function $u$ smooth in the disk and continuous on its
closure, whenever the right side is finite,
\begin{equation}\label{inverse-capacity-pullback}
 \int_\D|D(u\circ h)|^2\,dA
 \leq\int_\D K_F|Du|^2\,dA.
\end{equation}
The same assertion follows by approximation for the square-coordinate
potentials used below.
\end{lemma}
\begin{proof}
For $p\leq2$, $\K_F-1\leq2R_p(|\mu_F|^2)$ and $\eta\geq1$.
For $p>2$, \eqref{Psiidentity} and H\"older's inequality on the
finite Euclidean area disk give $\K_F\in L^1(\D)$ as well.
Since $K_F\leq2\K_F$, the inverse theorems
\cite[Theorems 1.2--1.3]{HenclKoskelaInverse} give
$h\in W^{1,2}_{\loc}$ and finite distortion of $h$.
At its regular points the inverse derivative formula holds; on
$\{J_h=0\}$ we have $Dh=0$ almost everywhere.
The area inequality for a Sobolev homeomorphism, applied to $h$,
therefore gives the global Dirichlet estimate by exhaustion and
\[
 \int_\D\eta(h)|h_w|^{2-p}|h_{\bar w}|^p\,dA
 =\int_{\{J_h>0\}}
       \eta(h)R_p(|\mu_F\circ h|^2)J_h\,dA
 \leq\I_p(F) \quad(p\leq2).
\]
Similarly the local chain rule and
$|Dh|_{\mathrm{op}}^2=K_hJ_h$ give
\[
 |D(u\circ h)|^2\leq (|Du|^2K_F)\circ h\;J_h
\]
almost everywhere, with zero derivative on $\{J_h=0\}$.
Integration and the same area inequality prove
\eqref{inverse-capacity-pullback}. Truncation and exhaustion give
the assertion for finite-energy potentials smooth only in the
interior. This argument needs an area inequality, not a Lusin
$N$ assertion for the map $F$.
\end{proof}

\begin{proposition}[Quantitative boundary modulus control]
\label{finite-energy-boundary-moduli}
Let $F:\D\to\D$ be a finite-distortion homeomorphism with
$\I_p(F)<\infty$, $p\geq1$. Then $F$ extends homeomorphically to
the closed disk. Denote its boundary map by $g$, and set
\[
 C_0=\frac{16\pi}{\Gamma(1/4)^4},\qquad
 X_p=\begin{cases}2C_0\I_p(F),&1\leq p\leq2,\\
 2C_0^{2/p}\I_p(F)^{2/p},&p>2,
 \end{cases}
 \qquad \Lambda(X)=1+X+\sqrt{X(X+2)}.
\]
For every modulus-one marked quadrilateral $Q$ in the disk,
\begin{equation}\label{boundary-modulus-explicit}
 \Lambda(X_p)^{-1}\leq\operatorname{mod}g(Q)
                              \leq\Lambda(X_p).
\end{equation}
Consequently $g$ is quasisymmetric and, in fact, symmetric.
After a fixed three-point normalisation, its quasisymmetry constant
is bounded in terms of $p$ and $\I_p(F)$ alone.
\end{proposition}
\begin{proof}
Proposition~\ref{factorisation} gives the homeomorphic boundary extension; its proof does not use this result.
All modulus-one marked disk quadrilaterals form one M\"obius orbit.
For the model with four equally spaced vertices, a conformal map
to a square has derivative $(1+z^4)^{-1/2}$, up to rotation.
The square has side $L=\Gamma(1/4)^2/(4\sqrt\pi)$.
The unit-square extremal density is therefore
\[
 \rho_0^2(z)=\frac1{L^2|1+z^4|},\qquad
 \int_\D\rho_0^2\,dA=1,
 \qquad \frac{\rho_0^2(z)}{\eta(z)}\leq C_0.
\]
Indeed $|1+z^4|\geq1-|z|^4$, so this ratio is at most
$C_0(1-|z|^2)/(1+|z|^2)$ and tends uniformly to zero at the
boundary. Conformal transport gives
$\rho_Q^2\leq C_0\eta$ and $\int\rho_Q^2=1$ for every such $Q$.

Take a coordinate of the unit-square map as a scalar test
potential, with boundary values zero and one on the two designated
opposite sides. Its composition with $F^{-1}$ is admissible for
the image quadrilateral. The Dirichlet characterisation of modulus
and \eqref{inverse-capacity-pullback} give
\[
 \operatorname{mod}g(Q)\leq Y_Q:=\int_\D K_F\rho_Q^2\,dA.
\]
Put $X_Q=\int_\D(\K_F-1)\rho_Q^2\,dA$.
The identity $(K_F-1)^2=2(\K_F-1)K_F$ and Cauchy--Schwarz imply
\[
 Y_Q\leq1+\sqrt{2X_QY_Q},\qquad
 Y_Q\leq\Lambda(X_Q).
\]
All integrals are finite since $K_F\leq2\K_F$.
For $p\leq2$, $X_Q\leq2C_0\I_p(F)$.
For $p>2$, H\"older's inequality for the probability measure
$\rho_Q^2\,dA$ gives
\[
 X_Q\leq
 \left(\int_\D(\K_F-1)^{p/2}\rho_Q^2\,dA\right)^{2/p}
 \leq2C_0^{2/p}\I_p(F)^{2/p}.
\]
Apply the same upper bound to the conjugate quadrilateral to get
the lower bound in \eqref{boundary-modulus-explicit}.
The modulus characterisation of quasisymmetry now applies
\cite{AIM,HK}. In invariant terms this controls cross-ratio
distortion; angular quasisymmetry constants require the stated
normalisation.

For symmetry, consider modulus-one quadrilaterals whose
normalising automorphisms escape compact subsets of the
automorphism group. Their transported functions
$q_Q=\rho_Q^2/\eta$ tend uniformly to zero on each compact subset
of $\D$, and $0\leq q_Q\leq C_0$.
For $p\leq2$, integrability of $(\K_F-1)\eta$ therefore gives
$X_Q\to0$ by a compact-set and tail decomposition.
For $p>2$, the same decomposition uses H\"older and
\[
 \int_\D q_Q^{p/(p-2)}\,dA_\eta
 \leq C_0^{2/(p-2)}\int_\D q_Q\,dA_\eta
 =C_0^{2/(p-2)}.
\]
Thus shrinking boundary quadrilaterals have image modulus tending
uniformly to one. The asymptotic form of the modulus
characterisation gives symmetry.
\end{proof}

\begin{remark}
No unnormalised angular bound depending only on energy is possible:
disk automorphisms have zero energy, while their angular
quasisymmetry constants need not be bounded. Boundary extension itself follows from the independently justified factorisation, rather than from an unproved half-annulus modulus estimate.
\end{remark}
\subsection{The finite-distortion characterisation of the Weil--Petersson class}\label{wp-characterisation}
\begin{theorem}[Finite-distortion characterisation of the Weil--Petersson class]
\label{quadratic-boundary-class-equality}
A circle homeomorphism $g$ is Weil--Petersson if and only if it has
a finite-distortion homeomorphic disk extension with trace $g$ such that
\[
 \int_\D(\K_F-1)\,dA_\eta<\infty.
\]
Equivalently, $\mathcal T_2=T^2$.
\end{theorem}
\begin{proof}
The inclusion $T^2\subset\mathcal T_2$ is Lemma~\ref{easyinclusion}.
Conversely, Proposition~\ref{finite-energy-boundary-moduli} makes
$g$ quasisymmetric, and
Lemma~\ref{finite-distortion-inverse-calculus} makes $F^{-1}$ an
admissible finite-energy member of $\mathcal M_{g^{-1}}$.
Propositions~\ref{quadratic-direct-existence} and \ref{quadratic-minimiser-regularity}, which need only a finite-energy competitor in $\mathcal M_{g^{-1}}$, produce the
harmonic diffeomorphic minimiser $h$. Its finite excess energy
gives quasiconformality by
Proposition~\ref{quadratic-energy-qc-bound}. For $f=h^{-1}$,
\[
 \int_\D|\mu_f|^2\,dA_\eta
 \leq\I_2(f)=\E_2(h)
 \leq\E_2(F^{-1})\leq\I_2(F)<\infty.
\]
Thus $f$ is a quasiconformal representative of $g$ in $T^2$.
\end{proof}

\begin{remark}[The classes for $1<p<2$]\label{class-inclusions}
For $1\leq p\leq2$ and $0\leq r<1$ we have $r^2/(1-r^2)\leq r^p/(1-r^2)$, so $\I_2(F)\leq\I_p(F)$ for every finite-distortion map $F$ (see \eqref{subquadratic-energy-order}). Hence a finite-distortion $\I_p$-barrier is a finite-distortion $\I_2$-barrier, and Lemma~\ref{easyinclusion} together with Theorem~\ref{quadratic-boundary-class-equality} gives
\[
 T^p\subset\mathcal T_p\subset\mathcal T_2=T^2\qquad(1\leq p\leq2).
\]
Thus every $g\in\mathcal T_p$ is Weil--Petersson. Its harmonic extension $h_g$ is a quasiconformal representative, since $\mathcal H\geq1$ in \eqref{quadratic-energy-density-bounds} gives $\int_\D|\mu_{h_g}|^2\,dA_\eta\leq\E_2(h_g)$, and so is $h_{g^{-1}}^{-1}$. That some quasiconformal extension has $\mu\in L^p(dA_\eta)$, and not merely $L^2(dA_\eta)$, is Theorem~\ref{pc-main}; whether $h_g$ itself does is not settled there. At $p=2$ the corresponding statement for the Douady--Earle extension is Cui's theorem \cite{Cui}.
\end{remark}

\begin{proof}[Proof of Theorem~\ref{thm-C}]
The characterisation of the Weil--Petersson class is Theorem~\ref{quadratic-boundary-class-equality}, since $\I_2(F)=\frac12\int_\D(\K_F-1)\,dA_\eta$. The homeomorphic extension and the factorisation are Proposition~\ref{factorisation}, and the quantitative quasisymmetry and symmetry are Proposition~\ref{finite-energy-boundary-moduli}.
\end{proof}

\subsection{Why a $C^1$ boundary factor need not be Weil--Petersson}
Shen's intrinsic characterisation says that an absolutely continuous circle homeomorphism is Weil--Petersson precisely when its logarithmic angular derivative has $H^{1/2}$ regularity \cite{Shen}. The following example distinguishes the two conditions.

Let
\[
 u(\theta)=\sum_{n=1}^{\infty}\frac{2^{-n/2}}{\sqrt n}\cos(2^n\theta).
\]
The series converges uniformly, so $u$ is continuous. Its homogeneous $H^{1/2}$ seminorm diverges, since, up to a fixed factor,
\[
 \sum_{n\geq1}2^n\left(\frac{2^{-n/2}}{\sqrt n}\right)^2
 =\sum_{n\geq1}\frac1n=\infty.
\]
Set $c=(2\pi)^{-1}\int_0^{2\pi}e^{u(t)}dt$ and
\[
 v(\theta)=c^{-1}\int_0^\theta e^{u(t)}dt,\qquad
 h(e^{i\theta})=e^{iv(\theta)}.
\]
Then $h$ is a $C^1$ diffeomorphism with strictly positive continuous derivative, but $\log v'=u-\log c\notin H^{1/2}$. Hence $h\notin T^2$. In particular, $C^1$ regularity of the factor $h$ cannot by itself justify a quasiconformal extension of $g$ with an $L^1(dA_\eta)$ coefficient; Theorem~\ref{pc-main} shows that the boundary values of a finite-energy map are nevertheless in $T^p$, by an argument that does not pass through the factorisation.
\section{The classes \texorpdfstring{$\mathcal T_p$ and $T^p$}{T\_p and T\^p} coincide for every \texorpdfstring{$p>1$}{p>1}}\label{p-classes}
This section proves Theorem~\ref{thm-E}: for every $1<p<\infty$ the boundary values of a finite-distortion homeomorphism $G:\D\to\D$ with $\I_p(G)<\infty$ lie in $T^p$. The argument is a boundary estimate and uses none of the variational theory of Sections~\ref{quadratic-theory}--\ref{schoen-conjecture}. From Section~\ref{boundary-classes} it uses exactly two things: the homeomorphic boundary extension with quasisymmetric boundary values (Propositions~\ref{factorisation} and \ref{finite-energy-boundary-moduli}), and the capacity estimate \eqref{inverse-capacity-pullback} of Lemma~\ref{finite-distortion-inverse-calculus}, which rests on the Sobolev regularity of the inverse. The plan is as follows.
\begin{enumerate}[label=\textup{(\arabic*)}]
\item The modulus inequality $\operatorname{mod}g(Q)\le\int K_G\rho_Q^2\,dA$ for every quadrilateral $Q$ (Lemma~\ref{pc-modulus}), applied to $Q$ and to its conjugate, bounds the logarithmic modulus distortion $c_Q$ of the boundary map by $\int(K_G-1)(\rho_Q^2+\rho_{Q^*}^2)\,dA$.
\item The pointwise inequality $K-1\lesssim|\mu|+e_p(|\mu|)$, $e_p(m)=m^p/(1-m^2)$, and H\"older's inequality for the probability measure $\rho_Q^2\,dA/\operatorname{mod}Q$ give $|c_Q|^p\lesssim X_Q$ for bounded $X_Q=\int e_p(|\mu_G|)(\rho_Q^2+\rho_{Q^*}^2)\,dA$ (Proposition~\ref{pc-perQ}).
\item For the two-parameter family of quadrilaterals with vertices $x-t,x-\frac t3,x+\frac t3,x+t$ on $\partial\IH$, the overlap function $\int\!\!\int(\rho_{x,t}^2+\rho_{x,t,*}^2)\,dx\,dt/t^2$ is a constant multiple of the hyperbolic density (Lemma~\ref{pc-overlap}); four finite vertices are essential. Integration gives $\int\!\!\int|c_{x,t}|^p\,dx\,dt/t^2\lesssim\I_p(G)$ (Theorem~\ref{pc-starthm}).
\item Two exact identities for the quasisymmetric ratio $q(x,t)=\log\bigl[(g(x+t)-g(x))/(g(x)-g(x-t))\bigr]$, one for the four-point cross-ratio and one for scale doubling, with two telescoping arguments, show that $c\in L^p(dx\,dt/t^2)$ forces $q\in L^p(dx\,dt/t^2)$ (Lemmas~\ref{pc-telescope} and \ref{pc-doubling}).
\item A dyadic interpolation estimate shows that the Beurling--Ahlfors extension of a quasisymmetric $g$ has $|\mu|$ dominated by the normalised affine-interpolation defect, whose $L^r(dx\,dt/t^2)$ norm is dominated by that of $q$ (Lemma~\ref{pc-extension}). Hence $q\in L^p(dx\,dt/t^2)$ gives $g\in T^p$.
\end{enumerate}
Steps (1)--(4) are quantitative, and steps (1)--(3) hold for $p=1$ as well. A by-product is a derivative-free characterisation of $T^p$ for quasisymmetric maps (Theorem~\ref{pc-characterisation}), which at $p=2$ is essentially a theorem of Bishop; see Remark~\ref{pc-bishop}.

\subsection{Quadrilaterals, moduli and the half-plane setting}\label{pc-setting}
We work on the upper half-plane $\IH$, where the hyperbolic density is $\eta(z)=(2\operatorname{Im}z)^{-2}$. All quantities below ($|\mu_G|$, $K_G$, $dA_\eta$, moduli of quadrilaterals, the extremal densities $\rho_Q^2\,dA$) are invariant under pre- and postcomposition with M\"obius self-maps, so we pass freely between $\D$ and $\IH$; in particular $\I_p(G)$ is unchanged if $G$ is replaced by $A\circ G\circ B$ with $A,B$ M\"obius. For $0\le m<1$ write
\begin{equation}\label{pc-ep}
 e_p(m)=\frac{m^p}{1-m^2},\qquad E_p(G)=\int e_p(|\mu_G|)\,dA_\eta\le\I_p(G),
\end{equation}
with equality for $1\le p\le2$; for $p>2$ the density of $\I_p$ in \eqref{intro-functionals} dominates $e_p$ because $(1-m^2)^{-p/2}\ge(1-m^2)^{-1}$.

A quadrilateral $Q=(\IH;a,b,c,d)$ consists of $\IH$ and four points $a<b<c<d$ of $\partial\IH=\IR$ (one of them may be $\infty$, the inequalities being read cyclically). Its $a$-sides are the arcs $[a,b]$ and $[c,d]$, its $b$-sides are $[b,c]$ and $[d,a]$ (the latter through $\infty$). The modulus is defined by the Dirichlet principle,
\begin{equation}\label{pc-mod-def}
 \operatorname{mod}Q=\inf\Bigl\{\int_\IH|\nabla v|^2\,dA:\ v\in\mathcal A(Q)\Bigr\},
\end{equation}
where $\mathcal A(Q)$ consists of the $v\in W^{1,2}_{\loc}(\IH)\cap C(\overline\IH\cup\{\infty\})$ with $\int_\IH|\nabla v|^2\,dA<\infty$, $v=0$ on $[a,b]$ and $v=1$ on $[c,d]$. (The class is local because its members tend to a constant at $\infty$ and are not in $L^2(\IH)$; under a Cayley map it corresponds to $\{v\in W^{1,2}(\D)\cap C(\overline\D)\}$ with the transported boundary values, since the Dirichlet integral is conformally invariant and a function continuous on the closed disk with finite Dirichlet integral lies in $W^{1,2}(\D)$.) The infimum is attained by $u_Q=\operatorname{Re}f_Q$, where $f_Q$ maps $\IH$ conformally onto the rectangle $(0,1)\times(0,\operatorname{mod}Q)$ with $[a,b]\mapsto\{0\}\times[0,\operatorname{mod}Q]$ and $[c,d]\mapsto\{1\}\times[0,\operatorname{mod}Q]$. By Carath\'eodory's theorem $f_Q$ extends continuously to $\overline\IH\cup\{\infty\}$, so $u_Q\in\mathcal A(Q)$ is harmonic in $\IH$ with values $0$ and $1$ on the $a$-sides. We write
\[
 \rho_Q=|\nabla u_Q|=|f_Q'|,\qquad\int_\IH\rho_Q^2\,dA=\operatorname{mod}Q .
\]
The conjugate quadrilateral $Q^*=(\IH;b,c,d,a)$ exchanges the roles of the sides; its extremal function is $\operatorname{Im}f_Q/\operatorname{mod}Q$, so $\rho_{Q^*}=\rho_Q/\operatorname{mod}Q$ and
\begin{equation}\label{pc-conjugate}
 \operatorname{mod}Q^*=\frac1{\operatorname{mod}Q}.
\end{equation}
For an increasing homeomorphism $g$ of $\IR\cup\{\infty\}$ we write $g(Q)=(\IH;g(a),g(b),g(c),g(d))$; then $g(Q^*)=g(Q)^*$.

\paragraph{Modulus as a function of the cross-ratio.} For $a<b<c<d$ finite put
\begin{equation}\label{pc-chi-def}
 \chi(a,b,c,d)=\frac{(b-a)(d-c)}{(c-a)(d-b)}\in(0,1).
\end{equation}
The M\"obius map taking $(a,b,c,d)$ to $(-1/k,-1,1,1/k)$ exists for exactly one $k\in(0,1)$, namely $k=(1-\sqrt\chi)/(1+\sqrt\chi)$, and the Schwarz--Christoffel map $f(\zeta)=\int_0^\zeta\bigl((1-w^2)(1-k^2w^2)\bigr)^{-1/2}dw$ sends $\IH$ onto the rectangle with vertices $\pm K$, $\pm K+iK'$, where $K=K(k)$ and $K'=K(\sqrt{1-k^2})$ are the complete elliptic integrals; the $a$-sides $[-1/k,-1]$ and $[1,1/k]$ go to the vertical sides. Hence
\begin{equation}\label{pc-M-def}
 \operatorname{mod}Q=M(\chi),\qquad M(\chi)=\frac{K'(k)}{2K(k)},\quad k=\frac{1-\sqrt\chi}{1+\sqrt\chi}.
\end{equation}
By Legendre's relation, $\frac{d}{dk}\frac{K'(k)}{K(k)}=-\frac{\pi}{2k(1-k^2)K(k)^2}$, and $dk/d\chi<0$, so $M$ is a real-analytic increasing bijection of $(0,1)$ onto $(0,\infty)$ with $M'>0$ everywhere.

\paragraph{The family $Q_{x,t}$.} For $x\in\IR$ and $t>0$ let $Q_{x,t}$ be the quadrilateral with vertices
\[
 x-t,\quad x-\tfrac t3,\quad x+\tfrac t3,\quad x+t ,
\]
and for an increasing homeomorphism $g$ of $\IR$ let
\begin{equation}\label{pc-c-def}
 c_{x,t}=c_{x,t}(g)=\log\frac{\operatorname{mod}g(Q_{x,t})}{\operatorname{mod}Q_{x,t}},\qquad
 A_p(g)=\int_0^\infty\!\!\int_\IR|c_{x,t}|^p\,\frac{dx\,dt}{t^2}.
\end{equation}
Here $\chi=\frac14$, $k=\frac13$, and
\begin{equation}\label{pc-m0}
 m_0:=\operatorname{mod}Q_{x,t}=M(\tfrac14)=\frac{K'(1/3)}{2K(1/3)}=0.78170\ldots,\qquad M_0:=\max(m_0,m_0^{-1})=1.2793\ldots .
\end{equation}
The measure $dx\,dt/t^2$ on $\IR\times(0,\infty)$ is invariant under the affine maps $(x,t)\mapsto(\alpha x+\beta,\alpha t)$, $\alpha>0$, which is the action of the affine group $z\mapsto\alpha z+\beta$ on the family $\{Q_{x,t}\}$; the family is not invariant under the full M\"obius group.

\paragraph{Quasisymmetry on the line.} An increasing homeomorphism $g:\IR\to\IR$ is $\Lambda$-quasisymmetric if
\begin{equation}\label{pc-qs-def}
 \Lambda^{-1}\le\frac{g(x+t)-g(x)}{g(x)-g(x-t)}\le\Lambda\qquad(x\in\IR,\ t>0),
\end{equation}
and we write
\begin{equation}\label{pc-q-def}
 q(x,t)=\log\frac{g(x+t)-g(x)}{g(x)-g(x-t)},\qquad d\nu(x,t)=\frac{dx\,dt}{t^2},
\end{equation}
so that $|q|\le\log\Lambda$. By Beurling--Ahlfors \cite{BeurlingAhlfors}, \cite[Chapter~IV]{AhlforsLectures}, the maps that are $\Lambda$-quasisymmetric for some $\Lambda$ are exactly the boundary maps of quasiconformal self-maps of $\IH$ fixing $\infty$, and a circle homeomorphism is quasisymmetric if and only if its conjugate by a Cayley map (a fixed point being sent to $\infty$ on both sides) is $\Lambda$-quasisymmetric for some $\Lambda$. We say that $g:\IR\to\IR$ is in $T^p$ if it has a quasiconformal extension $F:\IH\to\IH$ with $\int_\IH|\mu_F|^p\,d\nu<\infty$; since $d\nu=4\,dA_\eta$ on $\IH$, a circle homeomorphism is in $T^p$ in the sense of Section~\ref{two-classes} if and only if its conjugate on $\IR$ is.

\subsection{The statements}
\begin{theorem}[A derivative-free characterisation of $T^p$]\label{pc-characterisation}
Let $1<p<\infty$ and let $g:\IR\to\IR$ be $\Lambda$-quasisymmetric. The following are equivalent:
\begin{enumerate}[label=\textup{(\roman*)}]
\item $g\in T^p$;
\item $A_p(g)=\displaystyle\int_0^\infty\!\!\int_\IR|c_{x,t}|^p\,\frac{dx\,dt}{t^2}<\infty$;
\item $\displaystyle\int_0^\infty\!\!\int_\IR|q(x,t)|^p\,\frac{dx\,dt}{t^2}<\infty$.
\end{enumerate}
Quantitatively: $\|q\|_{L^p(d\nu)}\le C_{p,\Lambda}A_p(g)^{1/p}$; the Beurling--Ahlfors extension $F$ of $g$ satisfies $\|\mu_F\|_{L^p(d\nu)}\le\frac52(\Lambda+1)\|q\|_{L^p(d\nu)}$; and $A_p(g)\le\Theta_p(E_p(F))$ for every finite-distortion extension $F:\IH\to\IH$ of $g$ with $\I_p(F)<\infty$, with $\Theta_p$ the explicit function of \eqref{pc-star}.
\end{theorem}

\begin{theorem}[$\mathcal T_p=T^p$]\label{pc-main}
Let $1<p<\infty$ and let $G:\D\to\D$ be a finite-distortion homeomorphism with $\I_p(G)<\infty$. Then its boundary map $g=G|_\IS$ belongs to $T^p$. Consequently $\mathcal T_p=T^p$, and $g$ is absolutely continuous with $\log g'\in B^{1/p}_{p,p}(\IS)$.
\end{theorem}

The condition in Theorem~\ref{pc-characterisation}(ii) is not formally M\"obius invariant, but by the equivalence with (i) its finiteness is. The proofs occupy the rest of the section; Theorem~\ref{pc-main} is proved in Section~\ref{pc-proofs}.

\subsection{The modulus inequality}\label{pc-modulus-section}
\begin{lemma}[Finite-distortion homeomorphisms distort moduli by at most $\int K\rho^2$]\label{pc-modulus}
Let $p\ge1$ and let $G:\IH\to\IH$ be a finite-distortion homeomorphism with $\I_p(G)<\infty$, extended by Proposition~\ref{factorisation} to a homeomorphism of $\overline\IH\cup\{\infty\}$ with boundary map $g$. For every quadrilateral $Q$ in $\IH$,
\begin{equation}\label{pc-modulus-ineq}
 \operatorname{mod}g(Q)\le\int_\IH K_G\,\rho_Q^2\,dA
\end{equation}
whenever the right-hand side is finite.
\end{lemma}
\begin{proof}
Transport to the disk by a Cayley map; all quantities are invariant and \eqref{pc-mod-def} becomes the disk definition. The transported $u_Q$ is harmonic in the disk and continuous on its closure (its gradient is unbounded at the four vertices, which is allowed), so \eqref{inverse-capacity-pullback} applies: $v:=u_Q\circ G^{-1}\in W^{1,2}$ and $\int|\nabla v|^2\le\int K_G\rho_Q^2$. Since $G^{-1}$ is a homeomorphism of the closed disk mapping the $a$-sides of $g(Q)$ onto the $a$-sides of $Q$, $v$ is continuous on the closed disk with $v=0$ on $g([a,b])$ and $v=1$ on $g([c,d])$. Thus $v$ is admissible in \eqref{pc-mod-def} for $g(Q)$.
\end{proof}

\begin{remark}[The $Q$-homeomorphism form]\label{pc-Q-homeo}
A homeomorphism $f$ is a $Q$-homeomorphism in the sense of Martio, Ryazanov, Srebro and Yakubov \cite{MRSY} if $M(f\Gamma)\le\int Q\rho^2\,dA$ for every curve family $\Gamma$ and every $\rho$ admissible for $\Gamma$. Inequality \eqref{pc-modulus-ineq} is exactly this statement, with $Q=K_G$, for the family of curves joining the $a$-sides of $Q$, whose modulus is $\operatorname{mod}Q$, and the particular admissible density $\rho_Q$ (if $\gamma$ joins the $a$-sides then $\int_\gamma\rho_Q\,ds\ge|u_Q(\gamma(1))-u_Q(\gamma(0))|=1$); by \eqref{pc-conjugate} the conjugate family is covered as well. Planar homeomorphisms of finite distortion with locally integrable distortion are $K_f$-homeomorphisms (Koskela--Onninen \cite{KoskelaOnninen}; see also \cite[Chapter~4]{MRSY}), and $K_G\in L^1(\D,dA)$ here by Lemma~\ref{finite-distortion-inverse-calculus}. The proof above through \eqref{inverse-capacity-pullback} is self-contained and avoids any Lusin property of $G$ itself.
\end{remark}

\begin{corollary}\label{pc-two-sided}
In the situation of Lemma~\ref{pc-modulus} let $m_0=\operatorname{mod}Q$, $M_0=\max(m_0,m_0^{-1})$, and
\[
 Z_Q=\int_\IH(K_G-1)\rho_Q^2\,dA,\qquad Z_{Q^*}=\int_\IH(K_G-1)\rho_{Q^*}^2\,dA .
\]
If both are finite then
\begin{equation}\label{pc-two-sided-eq}
 \frac{m_0}{1+m_0Z_{Q^*}}\le\operatorname{mod}g(Q)\le m_0\Bigl(1+\frac{Z_Q}{m_0}\Bigr),\qquad
 \bigl|c_Q\bigr|\le M_0\,(Z_Q+Z_{Q^*}),\quad c_Q:=\log\frac{\operatorname{mod}g(Q)}{\operatorname{mod}Q}.
\end{equation}
\end{corollary}
\begin{proof}
Since $\int\rho_Q^2=m_0$ and $\int\rho_{Q^*}^2=m_0^{-1}$, Lemma~\ref{pc-modulus} gives $\operatorname{mod}g(Q)\le m_0+Z_Q$ and $\operatorname{mod}g(Q)^*\le m_0^{-1}+Z_{Q^*}$; the lower bound follows from \eqref{pc-conjugate} applied to $g(Q)^*=g(Q^*)$. Taking logarithms and using $\log(1+x)\le x$ gives $-m_0Z_{Q^*}\le c_Q\le Z_Q/m_0$.
\end{proof}

\subsection{The per-quadrilateral estimate}\label{pc-perQ-section}
\begin{lemma}\label{pc-pointwise}
For $0\le m<1$ and $1\le p<\infty$,
\begin{equation}\label{pc-pointwise-eq}
 K-1=\frac{2m}{1-m}\le\begin{cases}4m,& m\le\tfrac12,\\[2pt] 2^{p+2}\,e_p(m),& m>\tfrac12,\end{cases}
 \qquad\text{and}\qquad m^p\le e_p(m)\ \text{ for all } m .
\end{equation}
\end{lemma}
\begin{proof}
For $m\le\frac12$, $1-m\ge\frac12$. For $m>\frac12$, $K-1=2m(1+m)/(1-m^2)\le4/(1-m^2)$ and $e_p(m)\ge2^{-p}/(1-m^2)$. The last inequality is $(1-m^2)^{-1}\ge1$.
\end{proof}

The point is the asymmetry: for $m$ close to $1$ the density $e_p(m)$ grows like $(1-m)^{-1}$, as fast as $K-1$, while for small $m$ we have only $K-1\approx2m$ against $e_p(m)\approx m^p$. The large-distortion regime costs nothing; the small-distortion regime costs a H\"older inequality, and that is how the exponent $p$ reaches the boundary.

\begin{proposition}\label{pc-perQ}
Let $1\le p<\infty$, let $G$ be as in Lemma~\ref{pc-modulus}, and let $Q$ be a quadrilateral with $\operatorname{mod}Q=m_0$. Put
\begin{equation}\label{pc-XQ}
 X_Q=\int_\IH e_p(|\mu_G|)\,\bigl(\rho_Q^2+\rho_{Q^*}^2\bigr)\,dA .
\end{equation}
If $X_Q<\infty$ then $Z_Q,Z_{Q^*}<\infty$ and
\begin{equation}\label{pc-perQ-eq}
 |c_Q|\le A_pX_Q^{1/p}+B_pX_Q,\qquad A_p=8M_0^{\,2-1/p},\quad B_p=2^{p+2}M_0 ;
\end{equation}
consequently
\begin{equation}\label{pc-perQ-p}
 |c_Q|^p\le2^{p-1}\bigl(A_p^p+B_p^pX_Q^{p-1}\bigr)X_Q .
\end{equation}
\end{proposition}
\begin{proof}
Write $m=|\mu_G|$ and $\rho=\rho_Q$. By Lemma~\ref{pc-pointwise},
\[
 Z_Q\le4\int_{\{m\le1/2\}}m\rho^2\,dA+2^{p+2}\int_{\{m>1/2\}}e_p(m)\rho^2\,dA .
\]
H\"older's inequality for the finite measure $\rho^2\,dA$ of total mass $m_0$, together with $m^p\le e_p(m)$, gives
\[
 \int_{\{m\le1/2\}}m\rho^2\,dA\le\Bigl(\int m^p\rho^2\,dA\Bigr)^{1/p}m_0^{1-1/p}\le m_0^{1-1/p}\Bigl(\int e_p(m)\rho^2\,dA\Bigr)^{1/p}.
\]
The same argument for $Q^*$ has $m_0^{-1}$ in place of $m_0$. Adding, and using $a^{1/p}+b^{1/p}\le2(a+b)^{1/p}$,
\[
 Z_Q+Z_{Q^*}\le8M_0^{\,1-1/p}X_Q^{1/p}+2^{p+2}X_Q .
\]
Now \eqref{pc-two-sided-eq} gives \eqref{pc-perQ-eq}, and \eqref{pc-perQ-p} is $(\alpha+\beta)^p\le2^{p-1}(\alpha^p+\beta^p)$.
\end{proof}

\paragraph{A uniform bound.} The extremal density is dominated by the hyperbolic density. For the quadrilateral $Q_0$ with vertices $-1,-\frac13,\frac13,1$, the map $f_0(\zeta)=\int_0^\zeta 3\,dw\,\bigl((1-w^2)(1-9w^2)\bigr)^{-1/2}$ gives $\rho_{Q_0}^2=|f_0'|^2/(4K^2)$ and $\rho_{Q_0^*}^2=|f_0'|^2/K'^2$ with $K=K(\frac13)$, $K'=K(\sqrt8/3)$. Since $|f_0'(\zeta)|^2\asymp|\zeta-v|^{-1}$ at each vertex $v$, $\asymp|\zeta|^{-4}$ as $\zeta\to\infty$, and is bounded elsewhere, while $\eta(\zeta)^{-1}=4(\operatorname{Im}\zeta)^2\le4|\zeta-v|^2$ and $\le4|\zeta|^2$,
\begin{equation}\label{pc-rho-eta}
 C_\rho:=\sup_{\zeta\in\IH}\frac{\rho_{Q_0}^2(\zeta)+\rho_{Q_0^*}^2(\zeta)}{\eta(\zeta)}<\infty\qquad(C_\rho=0.564\ldots\text{ numerically}).
\end{equation}
By the invariance of $\rho_Q^2/\eta$ under M\"obius maps the same bound holds for every $Q_{x,t}$, and therefore
\begin{equation}\label{pc-XQ-bound}
 X_{Q_{x,t}}\le C_\rho\int_\IH e_p(|\mu_G|)\,dA_\eta=C_\rho\,E_p(G)\le C_\rho\,\I_p(G)\qquad\text{for all }x\in\IR,\ t>0 .
\end{equation}
With \eqref{pc-perQ-eq} this recovers the quantitative quasisymmetry of Proposition~\ref{finite-energy-boundary-moduli} for this family; the new information is that \eqref{pc-perQ-p} is linear in $X_Q$ for bounded $X_Q$.

\subsection{The overlap function and the integrated estimate}\label{pc-overlap-section}
For $x\in\IR$, $t>0$ let $\rho_{x,t}=\rho_{Q_{x,t}}$ and $\rho_{x,t,*}=\rho_{Q_{x,t}^*}$.

\begin{lemma}\label{pc-overlap}
The function
\[
 W(z)=\int_0^\infty\!\!\int_\IR\bigl(\rho_{x,t}^2+\rho_{x,t,*}^2\bigr)(z)\,\frac{dx\,dt}{t^2},\qquad z\in\IH,
\]
is finite and equals $C_W\,\eta(z)$, where
\begin{equation}\label{pc-CW}
 C_W=4\int_\IH\operatorname{Im}\zeta\,\bigl(\rho_{Q_0}^2+\rho_{Q_0^*}^2\bigr)(\zeta)\,dA(\zeta)<\infty\qquad(C_W=3.1468\ldots\text{ numerically}).
\end{equation}
\end{lemma}
\begin{proof}
Let $A(\zeta)=x+t\zeta$, so $A(Q_0)=Q_{x,t}$ and, by conformal invariance, $\rho_{x,t}(x+t\zeta)=\rho_{Q_0}(\zeta)/t$. Hence $\rho_{x,t}^2(i)=t^{-2}\rho_{Q_0}^2\bigl((i-x)/t\bigr)$. The map $(x,t)\mapsto\zeta=(i-x)/t=\xi+i\sigma$ is a diffeomorphism of $\IR\times(0,\infty)$ onto $\IH$ with inverse $x=-\xi/\sigma$, $t=1/\sigma$ and Jacobian $dx\,dt=\sigma^{-3}d\xi\,d\sigma$. Therefore
\[
 \int_0^\infty\!\!\int_\IR\rho_{x,t}^2(i)\,\frac{dx\,dt}{t^2}
 =\int_\IH\sigma^4\rho_{Q_0}^2(\zeta)\,\sigma^{-3}\,d\xi\,d\sigma
 =\int_\IH\operatorname{Im}\zeta\;\rho_{Q_0}^2(\zeta)\,dA(\zeta),
\]
and likewise for the conjugate. Near the vertices $\rho_{Q_0}^2\asymp|\zeta-v|^{-1}$ is integrable; as $\zeta\to\infty$, $\operatorname{Im}\zeta\,\rho_{Q_0}^2(\zeta)=O(|\zeta|^{-3})$ by the Schwarz--Christoffel formula (four finite vertices, each contributing $|\zeta|^{-1}$ to $|f_0'|^2$), and $\int_{|\zeta|>1}|\zeta|^{-3}dA<\infty$. So $W(i)=C_W/4<\infty$.

For general $z=x_0+iy_0$ let $A(\zeta)=x_0+y_0\zeta$, so $A(i)=z$. Then $\rho_{x,t}(A\zeta)=\rho_{x',t'}(\zeta)/y_0$ with $(x',t')=((x-x_0)/y_0,\,t/y_0)$, and the measure is invariant, so $W(z)=y_0^{-2}W(i)=4W(i)\,\eta(z)$.
\end{proof}

\begin{remark}[Why four finite vertices]\label{pc-why-four}
The most natural family of ``quasisymmetry quadrilaterals'' is $(\IH;x-t,x,x+t,\infty)$, whose modulus distortion is a fixed function of the quasisymmetric ratio $e^{q(x,t)}$. For that family the Schwarz--Christoffel derivative has only three finite factors, $\rho^2\asymp|\zeta|^{-3}$ at infinity, and $\int\operatorname{Im}\zeta\,\rho^2\,dA$ diverges logarithmically: $W\equiv\infty$. This is not an artefact of the computation; the quadrilateral $(x-t,x,x+t,\infty)$ has infinite hyperbolic extent, and the family does not see $\mu_G$ locally. Four finite vertices make $W$ finite, and the price is that $c_{x,t}$ controls $q$ only through the identities of Section~\ref{pc-identities}, in which $q$ enters to first order only through differences.
\end{remark}

\begin{theorem}\label{pc-starthm}
Let $1\le p<\infty$ and let $G:\IH\to\IH$ be a finite-distortion homeomorphism with $\I_p(G)<\infty$, normalised so that its boundary map $g$ fixes $\infty$. Then
\begin{equation}\label{pc-star}
 A_p(g)=\int_0^\infty\!\!\int_\IR|c_{x,t}|^p\,\frac{dx\,dt}{t^2}
 \le\Theta_p\bigl(E_p(G)\bigr)\le\Theta_p\bigl(\I_p(G)\bigr),
\end{equation}
where
\[
 \Theta_p(E):=2^{p-1}C_W\Bigl(A_p^p+B_p^p\,(C_\rho E)^{p-1}\Bigr)E ,
\]
with $A_p,B_p$ from \eqref{pc-perQ-eq}, $C_\rho$ from \eqref{pc-rho-eta} and $C_W$ from \eqref{pc-CW}. In particular $A_p(g)<\infty$, and $\Theta_p(E)=O(E)$ as $E\to0$.
\end{theorem}
\begin{proof}
By \eqref{pc-perQ-p} and \eqref{pc-XQ-bound}, $|c_{x,t}|^p\le2^{p-1}\bigl(A_p^p+B_p^p(C_\rho E_p(G))^{p-1}\bigr)X_{Q_{x,t}}$. By Tonelli and Lemma~\ref{pc-overlap},
\[
 \int_0^\infty\!\!\int_\IR X_{Q_{x,t}}\,\frac{dx\,dt}{t^2}
 =\int_\IH e_p(|\mu_G|)(z)\,W(z)\,dA(z)=C_W\int_\IH e_p(|\mu_G|)\,dA_\eta=C_W\,E_p(G),
\]
and $\Theta_p$ is increasing.
\end{proof}

Nothing so far depends on the shape of $Q_{x,t}$ beyond having four finite vertices with a fixed cross-ratio: any family $(x-t,x-\alpha t,x+\beta t,x+t)$, $0<\alpha,\beta<1$, satisfies \eqref{pc-star} with constants depending on $(\alpha,\beta)$. The normalisation $g(\infty)=\infty$ costs nothing, since $\I_p$ is M\"obius invariant.

\subsection{Two exact identities for the quasisymmetric ratio}\label{pc-identities}
From here on $g:\IR\to\IR$ is $\Lambda$-quasisymmetric, $q$ and $\nu$ are as in \eqref{pc-q-def}, and
\begin{equation}\label{pc-L-def}
 L(v)=\log\cosh\frac v2,\qquad 0\le L(v)\le\frac{v^2}8,\qquad L(v)\ge c_Bv^2\ \text{ for }|v|\le B,\quad c_B=\frac{L(B)}{B^2},
\end{equation}
the last because $L(v)/v^2$ is even and decreasing in $|v|$. Both identities come from
\begin{equation}\label{pc-log-sum}
 \log(e^v+e^w)=\frac{v+w}2+\log2+L(v-w).
\end{equation}

\paragraph{The four-point identity.} For $s>0$ and $y\in\IR$ let
\[
 \ell_j=g(y+js)-g(y+(j-1)s)\quad(j=1,2,3),\qquad
 r_s(y)=\log\frac{4\ell_1\ell_3}{(\ell_1+\ell_2)(\ell_2+\ell_3)} .
\]
With $x=y+\frac32s$, $t=\frac32s$, the points $y,y+s,y+2s,y+3s$ are the vertices of $Q_{x,t}$, and $\chi_{x,t}:=\ell_1\ell_3/((\ell_1+\ell_2)(\ell_2+\ell_3))$ is the cross-ratio \eqref{pc-chi-def} of their images; thus $r_s(y)=\log(4\chi_{x,t})$, which vanishes exactly when $\chi_{x,t}=\frac14$, i.e.\ when $c_{x,t}=0$. Put
\[
 a=q(y+s,s)=\log\frac{\ell_2}{\ell_1},\qquad b=q(y+2s,s)=\log\frac{\ell_3}{\ell_2}.
\]
Then $\log(\ell_1+\ell_2)=\log\ell_1+\log(1+e^a)$ and $\log(\ell_2+\ell_3)=\log\ell_2+\log(1+e^b)$, and \eqref{pc-log-sum} gives, exactly,
\begin{equation}\label{pc-four-point}
 r_s(y)=\frac{b-a}2-L(a)-L(b).
\end{equation}
This uses only interval lengths; no derivative of $g$ has been introduced.

\paragraph{The doubling identity.} Fix $x\in\IR$, $s>0$, and put $\ell_+=g(x+s)-g(x)$, $\ell_{++}=g(x+2s)-g(x+s)$, $\ell_-=g(x)-g(x-s)$, $\ell_{--}=g(x-s)-g(x-2s)$, so that $q(x,s)=\log(\ell_+/\ell_-)$, $q(x+s,s)=\log(\ell_{++}/\ell_+)$ and $q(x-s,s)=\log(\ell_-/\ell_{--})$. Then $q(x,2s)=\log(\ell_++\ell_{++})-\log(\ell_-+\ell_{--})=\log\ell_++\log(1+e^{q(x+s,s)})-\log\ell_--\log(1+e^{-q(x-s,s)})$, and \eqref{pc-log-sum} (with $L$ even) gives, exactly,
\begin{equation}\label{pc-doubling-id}
 q(x,2s)=q(x,s)+\frac12\bigl[q(x+s,s)+q(x-s,s)\bigr]+L\bigl(q(x+s,s)\bigr)-L\bigl(q(x-s,s)\bigr).
\end{equation}
Consequently the doubling defect
\begin{equation}\label{pc-e-def}
 e(x,s):=q(x,2s)-2q(x,s)=\frac12\bigl[q(x+s,s)-2q(x,s)+q(x-s,s)\bigr]+L\bigl(q(x+s,s)\bigr)-L\bigl(q(x-s,s)\bigr)
\end{equation}
is a second difference of $q(\cdot,s)$ at step $s$ plus quadratic terms.

The two identities play different roles and neither suffices alone. In \eqref{pc-four-point} the quadratic terms $L(a)+L(b)$ enter with one sign, which makes the next lemma coercive and controls $q$ in $L^{2p}$; in \eqref{pc-doubling-id} they enter as a difference, which could cancel, but the linear term is a second difference whose telescoping in the scale converges geometrically and controls $q$ in $L^p$.

\subsection{Telescoping}\label{pc-telescoping-section}
\begin{lemma}\label{pc-telescope}
Let $1\le p<\infty$, $s>0$, and let $v:\IR\to[-B,B]$ be measurable. If
\[
 r(y)=\frac{v(y+s)-v(y)}2-L(v(y))-L(v(y+s))\in L^p(\IR),
\]
then $v\in L^{2p}(\IR)$ and
\begin{equation}\label{pc-telescope-eq}
 \int_\IR|v|^{2p}\,dy\le C_{p,B}\int_\IR|r|^p\,dy ,
\end{equation}
with $C_{p,B}$ independent of $s$.
\end{lemma}
\begin{proof}
The case $B=0$ is immediate. Define
\[
 H(v)=\frac{\operatorname{sgn}(v)|v|^{2p-1}}{2p-1},\qquad
 W(a,b)=\int_0^1|(1-\theta)a+\theta b|^{2p-2}\,d\theta ,
\]
so that $H'(v)=|v|^{2p-2}$ and $W$ is the divided difference of $H$:
\begin{equation}\label{pc-divided}
 (b-a)W(a,b)=H(b)-H(a).
\end{equation}
By homogeneity and compactness on $\{|a|+|b|=1\}$, where $W$ is continuous and positive, $W(a,b)\asymp_p(|a|+|b|)^{2p-2}$. With \eqref{pc-L-def} this gives, for $|a|,|b|\le B$,
\begin{equation}\label{pc-coercive}
 (L(a)+L(b))W(a,b)\ge c_{p,B}(|a|+|b|)^{2p},\qquad W(a,b)^{p'}\le C_p(|a|+|b|)^{2p}.
\end{equation}
Multiply the defining identity for $r$ by $W(a,b)$, with $a=v(y)$, $b=v(y+s)$:
\begin{equation}\label{pc-tested}
 (L(a)+L(b))W(a,b)=\tfrac12\bigl(H(b)-H(a)\bigr)-rW(a,b).
\end{equation}
Integrability must be established before the telescoping term is discarded. Integrate \eqref{pc-tested} over $[-R,R]$, $R>s$. The integral of $H(v(y+s))-H(v(y))$ is the difference of the integrals of $H(v)$ over $[R,R+s]$ and $[-R,-R+s]$, hence has absolute value at most $2s\sup_{|v|\le B}|H(v)|$. With $X_R=\int_{-R}^R(|v(y)|+|v(y+s)|)^{2p}\,dy$, \eqref{pc-coercive}, \eqref{pc-tested} and H\"older's inequality give
\[
 c_{p,B}X_R\le C_{p,B}\,s+C_p\|r\|_{L^p(\IR)}X_R^{1-1/p}.
\]
Young's inequality bounds $X_R$ independently of $R$, and monotone convergence gives $v\in L^{2p}(\IR)$. Now the two boundary-strip integrals of $H(v)$ tend to zero as $R\to\infty$, by H\"older's inequality on each strip. All other terms of \eqref{pc-tested} are integrable by \eqref{pc-coercive} and H\"older, so $\int_\IR(L(a)+L(b))W(a,b)\,dy=-\int_\IR rW(a,b)\,dy$, and \eqref{pc-coercive} once more gives
\[
 c_{p,B}\int_\IR(|a|+|b|)^{2p}\,dy\le C_p\|r\|_p\Bigl(\int_\IR(|a|+|b|)^{2p}\,dy\Bigr)^{1-1/p},
\]
which is \eqref{pc-telescope-eq}. The temporary boundary term involving $s$ has disappeared from the final estimate.
\end{proof}

For a $\Lambda$-quasisymmetric $g$ the image cross-ratio lies in a compact subinterval of $(0,1)$:
\begin{equation}\label{pc-chi-range}
 \chi_{x,t}=\frac1{(1+\ell_2/\ell_1)(1+\ell_2/\ell_3)}\in\Bigl[\frac1{(1+\Lambda)^2},\frac{\Lambda^2}{(1+\Lambda)^2}\Bigr].
\end{equation}
On this interval the two real-analytic functions $\log(4\chi)$ and $\log M(\chi)-\log M(\frac14)$ vanish only at $\chi=\frac14$, both with nonzero derivative there since $M'>0$; hence
\begin{equation}\label{pc-r-vs-c}
 |r_s(y)|\le C_\Lambda\,|c_{x,t}|\qquad(x=y+\tfrac32s,\ t=\tfrac32s),
\end{equation}
and since $dx\,dt/t^2=\frac23\,dy\,ds/s^2$,
\begin{equation}\label{pc-r-integral}
 \int_0^\infty\!\!\int_\IR|r_s(y)|^p\,dy\,\frac{ds}{s^2}\le\frac32C_\Lambda^p\,A_p(g).
\end{equation}

\begin{corollary}\label{pc-q2p}
Let $1\le p<\infty$ and let $g$ be $\Lambda$-quasisymmetric with $A_p(g)<\infty$. Then
\begin{align}
 \int_\IH|q|^{2p}\,d\nu&\le C_{p,\Lambda}\,A_p(g),\label{pc-q-2p}\\
 \int_\IH|q(x+s,s)-q(x,s)|^p\,d\nu(x,s)&\le C_{p,\Lambda}\,A_p(g).\label{pc-q-diff}
\end{align}
\end{corollary}
\begin{proof}
By \eqref{pc-r-integral} and Fubini, $r_s\in L^p(\IR)$ for a.e.\ $s>0$. For such $s$ apply Lemma~\ref{pc-telescope} to $v(y)=q(y+s,s)$, $B=\log\Lambda$, whose $r$ is $r_s$ by \eqref{pc-four-point}: $\int_\IR|q(y+s,s)|^{2p}\,dy\le C_{p,\Lambda}\int_\IR|r_s|^p\,dy$. Since $\int_\IR|q(y+s,s)|^{2p}dy=\int_\IR|q(x,s)|^{2p}dx$, integrating against $ds/s^2$ and using \eqref{pc-r-integral} gives \eqref{pc-q-2p}. For \eqref{pc-q-diff}, \eqref{pc-four-point} gives $q(y+2s,s)-q(y+s,s)=b-a=2r_s(y)+2L(a)+2L(b)$ with $L(a)\le a^2/8$, so
\[
 \int_\IR|q(y+2s,s)-q(y+s,s)|^p\,dy\le3^{p-1}2^p\Bigl(\int_\IR|r_s|^p\,dy+2\cdot8^{-p}\int_\IR|q(x,s)|^{2p}\,dx\Bigr);
\]
the left side equals $\int_\IR|q(x+s,s)-q(x,s)|^p\,dx$, and integration against $ds/s^2$ with \eqref{pc-r-integral} and \eqref{pc-q-2p} gives \eqref{pc-q-diff}.
\end{proof}

\begin{lemma}[Telescoping in the scale]\label{pc-doubling}
Let $1<p<\infty$, let $q$ be a bounded measurable function on $\IR\times(0,\infty)$, and let $e(x,s)=q(x,2s)-2q(x,s)$. If $e\in L^p(d\nu)$ then $q\in L^p(d\nu)$ and
\begin{equation}\label{pc-scale-telescope}
 \|q\|_{L^p(d\nu)}\le\frac{\|e\|_{L^p(d\nu)}}{2-2^{1/p}} .
\end{equation}
\end{lemma}
\begin{proof}
From $q(x,s)=\frac12q(x,2s)-\frac12e(x,s)$, by induction, $q(x,s)=2^{-N}q(x,2^Ns)-\sum_{k=0}^{N-1}2^{-(k+1)}e(x,2^ks)$ for $N\ge1$. Since $q$ is bounded, $2^{-N}q(x,2^Ns)\to0$, and since $e$ is bounded the series converges absolutely; thus
\[
 q(x,s)=-\sum_{k=0}^\infty2^{-(k+1)}e(x,2^ks)\qquad\text{pointwise}.
\]
The substitution $\sigma=2^ks$ gives $ds/s^2=2^k\,d\sigma/\sigma^2$, so $\|e(\cdot,2^k\,\cdot)\|_{L^p(d\nu)}=2^{k/p}\|e\|_{L^p(d\nu)}$. Minkowski's inequality for the partial sums and Fatou's lemma give
\[
 \|q\|_{L^p(d\nu)}\le\sum_{k=0}^\infty2^{-(k+1)}2^{k/p}\|e\|_{L^p(d\nu)}=\frac{\|e\|_{L^p(d\nu)}}{2}\sum_{k=0}^\infty2^{-k(1-1/p)}=\frac{\|e\|_{L^p(d\nu)}}{2-2^{1/p}},
\]
the series converging because $p>1$.
\end{proof}

\begin{corollary}\label{pc-qp}
Let $1<p<\infty$ and let $g$ be $\Lambda$-quasisymmetric with $A_p(g)<\infty$. Then
\begin{equation}\label{pc-q-p}
 \|q\|_{L^p(d\nu)}\le C_{p,\Lambda}\,A_p(g)^{1/p}.
\end{equation}
\end{corollary}
\begin{proof}
In \eqref{pc-e-def} write $D(x,s)=q(x+s,s)-q(x,s)$, so that $q(x-s,s)-q(x,s)=-D(x-s,s)$. Translation in $x$ for fixed $s$ preserves $\nu$, so $\|D(\cdot-s,\cdot)\|_{L^p(d\nu)}=\|D\|_{L^p(d\nu)}$, and $\|L(q(\cdot\pm s,\cdot))\|_{L^p(d\nu)}\le\frac18\|q\|_{L^{2p}(d\nu)}^2$ by \eqref{pc-L-def}. Hence $\|e\|_{L^p(d\nu)}\le\|D\|_{L^p(d\nu)}+\frac14\|q\|_{L^{2p}(d\nu)}^2\le C_{p,\Lambda}A_p(g)^{1/p}$ by \eqref{pc-q-diff} and \eqref{pc-q-2p}, and Lemma~\ref{pc-doubling} gives \eqref{pc-q-p}.
\end{proof}

\subsection{An extension estimate without boundary differentiability}\label{pc-extension-section}
\begin{lemma}\label{pc-extension}
Let $g$ be $\Lambda$-quasisymmetric and let $F$ be its Beurling--Ahlfors extension in the normalisation \eqref{pc-BA} below. For every $1\le r<\infty$,
\begin{equation}\label{pc-extension-eq}
 \|\mu_F\|_{L^r(d\nu)}\le\frac52(\Lambda+1)\,\|q\|_{L^r(d\nu)} .
\end{equation}
\end{lemma}
\begin{proof}
For $I=[x-y,x+y]$ let $P_I$ be the affine function agreeing with $g$ at the endpoints of $I$, and define the normalised interpolation defect
\[
 \beta(x,y)=\frac{\sup_{z\in I}|g(z)-P_I(z)|}{g(x+y)-g(x-y)} .
\]
Let $\kappa=\Lambda/(\Lambda+1)<1$. Each dyadic child $J'$ of an interval $J$ satisfies $|g(J')|\le\kappa|g(J)|$, by \eqref{pc-qs-def}. If $J$ has midpoint $m$ and half-length $h$, its midpoint interpolation defect is
\[
 \Bigl|g(m)-\frac{g(m-h)+g(m+h)}2\Bigr|=\frac{|g(J)|}2\Bigl|\tanh\frac{q(m,h)}2\Bigr|\le\frac{|g(J)|}4|q(m,h)| .
\]
The successive dyadic piecewise affine interpolants $P_n$ of $g$ on $I$ ($P_0=P_I$) converge uniformly to $g$, and $P_{n+1}-P_n$ is, on each level-$n$ interval $J$, the tent function of height equal to the midpoint defect of $J$. Since $|g(J)|\le\kappa^n|g(I)|$ for level-$n$ intervals,
\begin{equation}\label{pc-beta-sum}
 \beta(x,y)\le\frac14\sum_{n=0}^\infty\kappa^n\max_{0\le j<2^n}\bigl|q(x+\alpha_{n,j}y,\,2^{-n}y)\bigr|,\qquad\alpha_{n,j}=-1+(2j+1)2^{-n}.
\end{equation}
For each $n,j$ the change of variables $(X,T)=(x+\alpha_{n,j}y,\,2^{-n}y)$ has $dx\,dy/y^2=2^{-n}\,dX\,dT/T^2$, so
\[
 \int_\IH\max_{j<2^n}|q(x+\alpha_{n,j}y,2^{-n}y)|^r\,d\nu\le\sum_{j<2^n}2^{-n}\int_\IH|q|^r\,d\nu=\int_\IH|q|^r\,d\nu ,
\]
and Minkowski's inequality in \eqref{pc-beta-sum} gives
\begin{equation}\label{pc-beta-Lr}
 \|\beta\|_{L^r(d\nu)}\le\frac1{4(1-\kappa)}\|q\|_{L^r(d\nu)}=\frac{\Lambda+1}4\|q\|_{L^r(d\nu)} .
\end{equation}

Now let
\begin{equation}\label{pc-BA}
 F(x+iy)=U+iV,\qquad U(x,y)=\frac1{2y}\int_{x-y}^{x+y}g,\qquad V(x,y)=\frac1y\Bigl(\int_x^{x+y}g-\int_{x-y}^xg\Bigr).
\end{equation}
This is the Beurling--Ahlfors extension \cite{BeurlingAhlfors}, \cite[Chapter~IV]{AhlforsLectures}, with the imaginary part normalised by a factor $2$ (which changes $F$ by a real affine map and preserves quasiconformality); it is a quasiconformal homeomorphism of $\IH$ for quasisymmetric $g$, and $U,V$ are $C^1$ for continuous $g$. It commutes with real affine maps of the domain and of the range, and so do $\beta$ and $|\mu_F|$. To estimate $\mu_F$ at a point, normalise so that the point is $(0,1)$ and $g(-1)=0$, $g(1)=1$; write $g(t)=\frac{t+1}2+e(t)$ on $[-1,1]$, so $e(\pm1)=0$ and $\|e\|_\infty=\beta(0,1)=:\beta$. Differentiating \eqref{pc-BA} at $(0,1)$,
\[
 U_x=\tfrac12,\qquad U_y=\tfrac12-U=-\tfrac12\textstyle\int_{-1}^1e,\qquad V_x=1-2g(0)=-2e(0),
\]
\[
 V_y=1-V=\tfrac12-\Bigl(\int_0^1e-\int_{-1}^0e\Bigr).
\]
Since $g$ is increasing, $0\le V\le1$, so $V_y\ge0$ and $|F_z|=\frac12|(U_x+V_y)+i(V_x-U_y)|\ge\frac12(U_x+V_y)\ge\frac14$. Also $|U_y|\le\beta$, $|V_x|\le2\beta$ and $|V_y-\frac12|\le2\beta$, so $|F_{\bar z}|=\frac12|(U_x-V_y)+i(V_x+U_y)|\le\frac52\beta$. Hence $|\mu_F|\le10\beta$ pointwise, and \eqref{pc-beta-Lr} gives \eqref{pc-extension-eq}. No derivative or absolute continuity of $g$ was used.
\end{proof}

\begin{remark}[Bishop's theorem]\label{pc-bishop}
At $p=2$, Theorem~\ref{pc-characterisation} is essentially Theorem~1.6 of Bishop \cite{Bish1}, which characterises the Weil--Petersson class by $\sum_I\mathrm{qs}(\varphi,I)^2<\infty$ over dyadic decompositions of the circle, where $\mathrm{qs}(\varphi,I)=|\varphi(m(I))-m(\varphi(I))|/|\varphi(I)|$ is the normalised midpoint defect, $m(I)$ the midpoint of $I$. In the notation above $\mathrm{qs}(\varphi,I)=\frac12|\tanh(q/2)|\asymp|q|$ at the midpoint and scale of $I$, so his dyadic sum is a discretisation of $\int\!\!\int|q|^2\,dx\,dt/t^2$, and \eqref{pc-extension-eq} at $r=2$ is the continuous form of his sufficiency statement. Our proof of Lemma~\ref{pc-extension} adapts, in part, his argument for that direction: the geometric decay of the image lengths of nested dyadic intervals is used in the same way, with the Beurling--Ahlfors extension and the affine-interpolation defect in place of his piecewise affine extension on a triangulation. The necessity direction, (i)$\Rightarrow$(ii)$\Rightarrow$(iii), is where the methods differ: here it is the capacity estimate of Theorem~\ref{pc-starthm} together with the identities \eqref{pc-four-point} and \eqref{pc-doubling-id}, and it gives the same conclusion for every $p>1$. For the necessity of a cross-ratio condition at $p=2$ see also Wu \cite{Wu}.
\end{remark}

\subsection{Proofs of Theorems~\ref{pc-characterisation}, \ref{pc-main} and \ref{thm-E}}\label{pc-proofs}
\begin{proof}[Proof of Theorem~\ref{pc-characterisation}]
(i)$\Rightarrow$(ii). Let $F:\IH\to\IH$ be a quasiconformal extension of $g$ with $\|\mu_F\|_\infty\le k<1$ and $\int_\IH|\mu_F|^p\,d\nu<\infty$. Then $F$ has finite distortion and $\I_p(F)\le(1-k^2)^{-p_*/2}\int|\mu_F|^p\,dA_\eta<\infty$ (Lemma~\ref{easyinclusion}), so Lemma~\ref{pc-modulus} applies and Theorem~\ref{pc-starthm} gives $A_p(g)\le\Theta_p(E_p(F))<\infty$. (ii)$\Rightarrow$(iii) is Corollary~\ref{pc-qp}, and (iii)$\Rightarrow$(i) is Lemma~\ref{pc-extension} with $r=p$.
\end{proof}

\begin{proof}[Proof of Theorem~\ref{pc-main}]
By Propositions~\ref{factorisation} and \ref{finite-energy-boundary-moduli} the boundary map $\hat g=G|_\IS$ is a quasisymmetric homeomorphism. Composing $G$ with a rotation of the target we may assume $\hat g(1)=1$. Let $T(z)=(z-i)/(z+i)$ and $G_\IH=T^{-1}\circ G\circ T:\IH\to\IH$; it is a finite-distortion homeomorphism with $\I_p(G_\IH)=\I_p(G)$ whose boundary map $g=T^{-1}\circ\hat g\circ T$ fixes $\infty$ and is $\Lambda$-quasisymmetric for some $\Lambda$. Theorem~\ref{pc-starthm} gives $A_p(g)\le\Theta_p(\I_p(G))<\infty$, and Theorem~\ref{pc-characterisation} gives $g\in T^p$, i.e.\ $\hat g\in T^p$. Thus $\mathcal T_p\subset T^p$; the reverse inclusion is Lemma~\ref{easyinclusion}. Finally, Wei--Matsuzaki \cite[Theorem~5.5]{WeiMatsuzakiIntegrable} identify $T^p$, for quasisymmetric maps and every $p>1$, with the class of absolutely continuous homeomorphisms $e^{i\theta}\mapsto e^{i\varphi(\theta)}$ with $\log\varphi'\in B^{1/p}_{p,p}(\IS)$.
\end{proof}

\begin{proof}[Proof of Theorem~\ref{thm-E}]
The identification is Theorem~\ref{pc-main}, the characterisation is Theorem~\ref{pc-characterisation}, and the explicit extension is that of Lemma~\ref{pc-extension}: for $G$ normalised as in the proof of Theorem~\ref{pc-main}, the Beurling--Ahlfors extension $F$ of the conjugate boundary map satisfies
\[
 \int_\IH|\mu_F|^p\,d\nu\le\Bigl(\frac52(\Lambda+1)\Bigr)^pC_{p,\Lambda}\,\Theta_p(\I_p(G)),
\]
where, after a three-point normalisation, $\Lambda$ is bounded in terms of $p$ and $\I_p(G)$ by Proposition~\ref{finite-energy-boundary-moduli}.
\end{proof}

\begin{remark}\label{pc-remarks}
(a) Theorem~\ref{pc-main} gives a second proof of $\mathcal T_2=T^2$ (Theorem~\ref{quadratic-boundary-class-equality}). The proof in Section~\ref{wp-characterisation} produces the harmonic representative and is tied to the variational problem; the proof here produces the Beurling--Ahlfors representative and uses no equation. For $p\ne2$ the question was raised in Section~\ref{boundary-classes}; its answer is used in Part~II \cite{PartII}, where every statement made for $T^p$ data consequently holds for $\mathcal T_p$ data, and in particular a finite-distortion barrier of finite energy can always be replaced by a quasiconformal one.

(b) \emph{$p=1$.} Lemma~\ref{pc-modulus}, Proposition~\ref{pc-perQ}, Theorem~\ref{pc-starthm} and Lemma~\ref{pc-telescope} hold for $p=1$: $A_1(g)\le16M_0C_W\,\I_1(G)$ and $q\in L^2(d\nu)$, so Lemma~\ref{pc-extension} with $r=2$ gives $\mathcal T_1\subset T^2$. Lemma~\ref{pc-doubling} fails at $p=1$ ($\sum2^{-(k+1)}2^{k}$ diverges), and whether $\mathcal T_1=T^1$ is not settled here; the endpoint space at $p=1$ is in any case different \cite{Matsuzaki,WeiMatsuzakiIntegrable}.

(c) The Besov description of the classes enters only through the last sentence of the proof of Theorem~\ref{pc-main}. A direct real-variable derivation of $\log\varphi'\in B^{1/p}_{p,p}$ from Theorem~\ref{pc-starthm} is also possible, by linearising $c_{x,t}$ as a third difference of $\int\log\varphi'$ and using the characterisation of Besov spaces by differences; it requires $\log\varphi'\in B^{1/(2p)}_{2p,2p}$ a priori, which Theorem~\ref{pc-main} at exponent $2p$ supplies. We do not include it.
\end{remark}
\section{The canonical energy of Weil--Petersson data}\label{canonical-energy}
We return to the quadratic problem of Section~\ref{quadratic-theory} and prove Theorem~\ref{thm-D}, which identifies the minimum energy as a M\"obius-invariant function of Weil--Petersson data.
\begin{proof}[Proof of Theorem~\ref{thm-D}]
Let $\beta$ be Weil--Petersson and let $h_\beta$ be the minimiser of Theorem~\ref{thm-A}. If $F$ is any finite-distortion homeomorphic extension of $\beta^{-1}$ with $\I_2(F)<\infty$,
Lemma~\ref{finite-distortion-inverse-calculus} gives
$F^{-1}\in\mathcal M_\beta$ and
$\E_2(F^{-1})\leq\I_2(F)$. Hence
$\mathcal S(\beta)\leq\I_2(F)$, with equality for the smooth
quasiconformal map $h_\beta^{-1}$ by the change of variables. Equality for another $F$ forces
$F^{-1}=h_\beta$ by uniqueness. Energy invariance under disk
automorphisms, \eqref{inverse-invariance}, proves the claimed boundary invariance. Zero energy
forces $\bar\partial h_\beta=0$, and hence a disk automorphism.

Write $\mathcal H_\beta=\eta(h_\beta)|(h_\beta)_w|^2/\eta$ and $\mathcal L_\beta=\eta(h_\beta)|(h_\beta)_{\bar w}|^2/\eta$, so that $B_\beta=\sup\mathcal H_\beta$ and $1\leq\mathcal H_\beta\leq B_\beta$ by \eqref{quadratic-energy-density-bounds}. The identity
$|\Phi_\beta|^2/\eta^2=\mathcal H_\beta\mathcal L_\beta$,
together with $\int_\D\mathcal L_\beta\,dA_\eta=\mathcal S(\beta)$, gives the two-sided estimate by integration. Finally $B_\beta\leq B_0$ with $B_0$ the constant of Proposition~\ref{quadratic-energy-qc-bound} evaluated at $E_0=\mathcal S(\beta)$.
\end{proof}

The pointwise identity $\eta(h_\beta)|(h_\beta)_{\bar w}|^2=|\mu_{h_\beta}|\,|\Phi_\beta|$ also gives
\begin{equation}\label{canonical-hopf-lone}
 \mathcal S(\beta)=\int_\D|\mu_{h_\beta}|\,|\Phi_\beta|\,dA
                   \leq\|\Phi_\beta\|_{L^1(\D)} .
\end{equation}
The weighted square norm is thus quantitatively equivalent to the
minimum energy, with constants controlled by that energy.
For a smoothly varying family through the identity with harmonic
Beltrami tangent $\dot\mu=\eta^{-1}\overline q$, $q$ holomorphic, the formal quadratic
term of $\mathcal S$ is
$\int_\D|q|^2\eta^{-1}\,dA$, the Weil--Petersson norm in this
normalisation. This observation does not assert a global K\"ahler
potential or an identity with Liouville or Loewner energy.

\subsection{Two spaces of holomorphic differentials}\label{two-spaces}
The natural Hilbert norm on a holomorphic quadratic coefficient for quadratic Weil--Petersson geometry is
\[
 \|q\|_{A^2_\eta}^2=\int_\D|q|^2\eta^{-1}\area
 =\int_\D|q|^2(1-|z|^2)^2\area .
\]
This should be kept distinct from $\int_\D|q|\area$, which enters \eqref{canonical-hopf-lone} and Section~\ref{boundary-criterion}. For example, with analytic branches on the disk,
\[
 q(z)=\frac{1}{(1-z)^2\,[\log(e/(1-z))]^{3/4}}
\]
belongs to $A^2_\eta$ but not to $L^1(dA)$: near $z=1$ the two radial tests have logarithmic powers $3/2$ and $3/4$, respectively. The distinction matters, since an $L^1$ Hopf hypothesis forces the stronger boundary conclusion of Corollary~\ref{harmonicboundary}.

\subsection{Minimal graphs and the minimal Lagrangian extension}
\begin{proposition}[Minimal-graph area and distortion excess]\label{wolf-area-excess}
Let $h=f^{-1}$ be a hyperbolic harmonic disk diffeomorphism with
Hopf differential $\Phi=\eta(h)h_w\overline{h_{\bar w}}$.
On a univalent natural-coordinate patch $\zeta'^2=\Phi$, the
minimal graph with height $s=2\operatorname{Im}(\zeta\circ f)$ has
area density $K_f\,dA_\eta$ over the forward disk. Its globally
defined excess density satisfies the extended nonnegative identity
\begin{equation}\label{wolf-area-identity}
 \mathcal A_{\mathrm{ex}}(f):=\int_\D(K_f-1)\,dA_\eta
       =2\E_2(h)+2\|\Phi\|_{L^1(\D,dA)}.
\end{equation}
In particular this excess is finite if and only if
$\Phi\in L^1(\D)$, and then the boundary values are $C^1$
diffeomorphisms.
\end{proposition}
\begin{proof}
At $p=2$, $t=|\nabla\operatorname{Im}(\zeta\circ f)|/\sqrt\eta
=\sqrt r/(1-r)$, where $r=|\mu_f|$.
Hence the graph area multiplier is
$\sqrt{1+4t^2}=(1+r)/(1-r)=K_f$.
Changing variables by the diffeomorphism $h$ gives pointwise
\[
 (K_f\circ h-1)\eta(h)J_h
 =2\eta(h)|h_{\bar w}|^2+2\eta(h)|h_w||h_{\bar w}|.
\]
Integration proves \eqref{wolf-area-identity}, including infinite
values. Since $\E_2(h)\leq\|\Phi\|_1$ by the pointwise inequality
behind \eqref{canonical-hopf-lone}, finiteness is equivalent to
$\Phi\in L^1$. Corollary~\ref{harmonicboundary} gives the boundary
regularity.
\end{proof}

The local graph interpretation is Wolf's correspondence
\cite{WolfMinimal}. A global single-valued height is not asserted:
odd-order zeros of $\Phi$ and periods of $\sqrt\Phi$ can obstruct
it. The area density and \eqref{wolf-area-identity} remain globally
well defined. The naive area excess therefore imposes the stronger
$L^1$ Hopf condition, whereas $\mathcal S$ uses the weighted
$A^2_\eta$ norm; the example in Section~\ref{two-spaces} shows that the two are different. A further renormalisation would be needed to obtain
a graph-area quantity finite on the whole Weil--Petersson class.
\begin{corollary}[Comparison with the minimal Lagrangian extension]
\label{minimal-lagrangian-comparator}
Let $\beta$ be Weil--Petersson and let $\Psi_\beta$ be its
quasiconformal minimal Lagrangian extension. Then
\begin{equation}\label{minimal-lagrangian-energy-comparison}
 \mathcal S(\beta)\leq\I_2(\Psi_\beta^{-1})<\infty.
\end{equation}
\end{corollary}
\begin{proof}
The inverse of the minimal Lagrangian extension is the minimal
Lagrangian extension of $\beta^{-1}$. Since the Weil--Petersson
class is closed under inversion, the integrability characterisation
in \cite{AdS} makes its Beltrami coefficient hyperbolically square
integrable. Quasiconformality then makes $\I_2$ finite.
Apply the minimisation statement in Theorem~\ref{thm-D}.
\end{proof}

The inequality compares two canonical extensions with the same
boundary data. Comparison with their renormalised areas, or with
Loewner energy, requires a further quantitative theorem; no such
identification is used here.
\section{Connections and questions}\label{connections}
The following connections position the problem and suggest further questions. None of them is used in the proofs above.

\paragraph{Besov traces.}
The logarithmic derivative of the boundary map is a more precise invariant than $C^1$ regularity: Shen's theorem \cite{Shen} characterises Weil--Petersson homeomorphisms as the absolutely continuous circle homeomorphisms with $\log\varphi'\in H^{1/2}$, and Matsuzaki \cite{Matsuzaki} characterises the integrable Teichm\"uller spaces by critical Besov data, with a modified endpoint space at $p=1$, relating these structures to Cauchy and Hilbert transforms. Theorems~\ref{thm-C} and \ref{thm-E} add a finite-distortion characterisation for every $p>1$: an arbitrary finite-distortion barrier with $\I_p(F)<\infty$ forces the critical Besov seminorm of $\log\varphi' $ to be finite. The route is through the quasisymmetric ratio rather than through the derivative, and the derivative-free characterisation of Theorem~\ref{pc-characterisation} is a $p$-analogue of Bishop's dyadic criterion \cite{Bish1}.

\paragraph{Sigma models and energy subtraction.}
The harmonic-map energy is the classical action of a nonlinear sigma model, and its stress tensor in dimension two is represented by a holomorphic quadratic differential; see \cite{Jost}. The Hopf coefficient here fits this framework exactly. Pointwise, if $e=\eta(H)(|H_w|^2+|H_{\bar w}|^2)$ and $a=\eta(H)J_H$, then $e-a=2\eta(H)|H_{\bar w}|^2$: the inverse functional measures the excess of Dirichlet-energy density over area density. It is nonnegative and vanishes on conformal maps. This is the right way to think of the finite quantity on an infinite-area disk, while avoiding the undefined expression $\infty-\infty$.

\paragraph{Loewner energy, random curves and determinants.}
Wang identifies finite Loewner-energy loops with Weil--Petersson quasicircles and relates their energy to universal Liouville action and regularised determinants \cite{Wang}. Peltola--Wang connect chordal and multichordal Loewner energies to small-noise SLE large deviations and semiclassical conformal-field-theory equations \cite{PW}. A natural question is to compare the canonical energy $\mathcal S$ of Theorem~\ref{thm-D} with Loewner energy as a function of conformal welding. A first task is to compare second variations at the identity, keeping track of both welding sides and the inversion convention. Equality of finiteness classes is not equality of numerical energies, and even a two-sided comparison of energies would not identify their large-deviation rate functions.

\paragraph{Anti-de Sitter geometry and renormalised volume.}
Diaf--Moriani--Sma\"i--Smith--Trebeschi \cite{AdS} characterise Weil--Petersson homeomorphisms through maximal spacelike surfaces in anti-de Sitter three-space with finite renormalised area, and through the minimal Lagrangian extension with square-integrable Beltrami coefficient; Corollary~\ref{minimal-lagrangian-comparator} compares the latter with the harmonic extension. Bridgeman--Bromberg--Vargas Pallete--Wang relate universal Liouville action, under their regularity hypotheses, to a renormalised volume bounded by Epstein--Poincar\'e surfaces \cite{BBVW}, and Bishop's geometric characterisations connect the Weil--Petersson condition to curvature across scales and to minimal surfaces \cite{Bish2}. Comparison of $\mathcal S$ with renormalised areas or volumes requires a further quantitative theorem.

\paragraph{Operator theory.}
Takhtajan--Teo characterise Weil--Petersson points by Hilbert--Schmidt Grunsky operators and relate the corresponding Fredholm determinant to Liouville action \cite{TT}; Fan--Viklund--Wang study Loewner energy directly from a welding operator and its determinants \cite{FVW}. Operator or Besov criteria offer independent ways to test a proposed boundary-class identification such as $\mathcal T_p=T^p$, and may provide numerical diagnostics for examples.

\appendix
\section{Carleson-type hypotheses on the distortion}\label{carleson-appendix}
This appendix is off the main line of the paper but close to it. Section~\ref{p-classes} controls the boundary map of a finite-distortion homeomorphism by the energy $\I_p$; reading its proof again, one sees that the energy enters only through the quantity $Z_Q=\int(K_G-1)\rho_Q^2\,dA$ of Corollary~\ref{pc-two-sided}, which is a M\"obius-invariant average of the distortion excess $K_G-1$ at the hyperbolic position of the quadrilateral $Q$. We record what follows if one assumes bounds on these averages directly. The result is a scale of hypotheses, indexed by $p\in(1,\infty]$, whose endpoint $p=\infty$ is a Carleson-measure condition on $(K-1)\,dA$ and gives quasisymmetric boundary values, and whose finite exponents give boundary values in $T^p$. The conditions are strictly weaker than finite energy. Nothing here is used elsewhere in the paper.

\subsection{The invariant average of the distortion excess}
Let $F:\D\to\D$ be a homeomorphism of finite distortion with $K=K_F$, and for $w\in\D$ let
\begin{equation}\label{carl-Phi}
 \Phi_F(w)=\int_\D\bigl(K(z)-1\bigr)\frac{(1-|w|^2)^2}{|1-\bar wz|^4}\,dA(z)
 =\int_\D\frac{K(z)-1}{\cosh^4d(z,w)}\,dA_\eta(z),
\end{equation}
where $d$ is the hyperbolic distance of curvature $-4$; the two expressions agree because $1-|A^{-1}(z)|^2=\cosh^{-2}d(z,w)$ for a M\"obius self-map $A$ of $\D$ with $A(0)=w$, and $(1-|w|^2)^2|1-\bar wz|^{-4}=(1-|A^{-1}(z)|^2)^2\eta(z)$. Equivalently, by the change of variables $z=A(\zeta)$,
\begin{equation}\label{carl-Phi-moebius}
 \Phi_F(w)=\int_\D\bigl(K(A(\zeta))-1\bigr)\,dA(\zeta)=\|K\circ A-1\|_{L^1(\D)}\qquad(A(0)=w),
\end{equation}
so $\Phi_F(w)$ is the Euclidean $L^1$ norm of the distortion excess seen from $w$. In particular $\Phi_F(0)=\int_\D(K-1)\,dA$, and $\Phi_F\circ A=\Phi_{F\circ A}$ for every M\"obius $A$.

\begin{lemma}\label{carl-properties}
\begin{enumerate}[label=\textup{(\alph*)}]
\item $\Phi_F(w')\le e^{4d(w,w')}\Phi_F(w)$ for all $w,w'\in\D$.
\item $\sup_\D\Phi_F<\infty$ if and only if $(K-1)\,dA$ is a Carleson measure for the Bergman space, that is, $\int_{S(I)}(K-1)\,dA\le C|I|^2$ for every Carleson box $S(I)=\{re^{i\theta}:1-|I|\le r<1,\ e^{i\theta}\in I\}$ over an arc $I\subset\IS$; the two constants are comparable with absolute constants.
\item If $\Phi_F\in L^p(\D,dA_\eta)$ for some $1\le p<\infty$, then $\sup_\D\Phi_F\le C_p\|\Phi_F\|_{L^p(dA_\eta)}$ and $\Phi_F(w)\to0$ as $|w|\to1$.
\item If $\I_p(F)<\infty$ for some $1\le p<\infty$, then $\Phi_F\in L^p(\D,dA_\eta)$, with $\int_\D\Phi_F^p\,dA_\eta\le C_p\bigl(1+E_p(F)^{p-1}\bigr)E_p(F)$, $E_p(F)\le\I_p(F)$ as in \eqref{pc-ep}.
\end{enumerate}
\end{lemma}
\begin{proof}
(a) The triangle inequality gives $\cosh d(z,w)\le e^{d(w,w')}\cosh d(z,w')$, and the kernel in \eqref{carl-Phi} is $\cosh^{-4}$.

(b) This is the standard characterisation of Carleson measures for Bergman spaces by the normalised reproducing kernel; see \cite[Theorem~7.4]{Zhu} or \cite{Hastings}. (The easy direction: for $w=(1-|I|)e^{i\theta_I}$ the kernel is $\gtrsim|I|^{-2}$ on $S(I)$, so $\int_{S(I)}(K-1)\,dA\lesssim|I|^2\Phi_F(w)$. Conversely a dyadic decomposition of $\D$ into boxes, on each of which the kernel is comparable to its value at the box centre, gives the kernel bound from the box bound.)

(c) By (a), $\Phi_F(w)^p\,|B(w,1)|_\eta\le e^{4p}\int_{B(w,1)}\Phi_F^p\,dA_\eta$, where $B(w,1)$ is the hyperbolic unit ball and $|B(w,1)|_\eta$ is independent of $w$. This bounds $\sup\Phi_F$, and it gives $\Phi_F(w)\to0$ as $|w|\to1$ because $B(w,1)$ leaves every compact set and $\Phi_F^p\in L^1(dA_\eta)$.

(d) The kernel $\cosh^{-4}d(z,w)$ is at most $1$ and $\int_\D\cosh^{-4}d(z,w)\,dA_\eta(z)=\pi$ for every $w$. Put $\Psi(w)=\int_\D e_p(|\mu_F|)(z)\cosh^{-4}d(z,w)\,dA_\eta(z)$, so that $\Psi\le E_p(F)$ and, by Fubini and the symmetry of the kernel, $\int_\D\Psi\,dA_\eta=\pi E_p(F)$. Lemma~\ref{pc-pointwise} and H\"older's inequality for the probability measure $\pi^{-1}\cosh^{-4}d(z,w)\,dA_\eta(z)$ on $\{|\mu_F|\le\frac12\}$ give, exactly as in Proposition~\ref{pc-perQ},
\[
 \Phi_F(w)\le4\pi^{1-1/p}\Psi(w)^{1/p}+2^{p+2}\Psi(w),
\]
hence $\Phi_F^p\le2^{p-1}\bigl(4^p\pi^{p-1}+2^{p(p+2)}E_p(F)^{p-1}\bigr)\Psi$, and integration gives the claim.
\end{proof}

\begin{remark}
Since $K-1=2m/(1-m)\le4m/(1-m^2)=4e_1(m)$, $m=|\mu_F|$, one also has $\Phi_F\le4\int e_1(|\mu_F|)\,dA_\eta=4E_1(F)$ pointwise. Thus every finite-energy map satisfies $\sup\Phi_F<\infty$, and the conditions below contain those of Proposition~\ref{finite-energy-boundary-moduli} and Theorem~\ref{pc-main}. They also contain the quasiconformal case, $K\le K_0$, for which $\Phi_F\le\pi(K_0-1)$. In the quasiconformal case $\Phi_F\in L^p(dA_\eta)$ forces $\Phi_F\to0$ at the boundary by (c), so a quasiconformal map with constant $K>1$ near the circle does not satisfy any of the finite-exponent conditions, as it should not.
\end{remark}

\subsection{Boundary condensers}
The hypothesis $\sup\Phi_F<\infty$ makes $F$ behave on boundary condensers like a quasiconformal map, although $K$ may be unbounded, and even non-integrable to any power $q>1$, in the interior.

\begin{lemma}\label{carl-half-annulus}
Let $F$ be as above with $\sup_\D\Phi_F=:\Lambda_F<\infty$. For $\xi\in\IS$ and $0<2r\le R\le\frac12$ let $Q=Q(\xi;r,R)=\{z\in\D:r<|z-\xi|<R\}$, regarded as a quadrilateral whose $a$-sides are the two circular cross-cuts, and let $Q^*$ be its conjugate, whose $a$-sides are the two arcs of $\partial Q$ on $\IS$. Then, with absolute constants,
\begin{equation}\label{carl-half-annulus-eq}
 \frac{c}{\log(R/r)}\le\operatorname{mod}Q\le\frac{\pi}{\log(R/r)},\qquad
 \operatorname{mod}F(Q)\le\frac{C(1+\Lambda_F)}{\log(R/r)},
\end{equation}
and
\begin{equation}\label{carl-half-annulus-eq2}
 \operatorname{mod}F(Q^*)\le C(1+\Lambda_F)\log(R/r),
\end{equation}
where $F(Q)$ denotes the image quadrilateral with the corresponding sides. In particular $F$ distorts the moduli of these quadrilaterals and of their conjugates by at most the factor $K'=C'(1+\Lambda_F)$.
\end{lemma}
\begin{proof}
The proof of Lemma~\ref{pc-modulus} uses only that $u_Q\circ G^{-1}$ is admissible for the image quadrilateral; it therefore gives $\operatorname{mod}F(Q)\le\int_QK|\nabla u|^2\,dA$ for every Lipschitz $u$ on $\overline\D$ that is admissible for $Q$, with $h=F^{-1}\in W^{1,2}(\D)$ supplied by Lemma~\ref{finite-distortion-inverse-calculus}, which needs only $K\in L^1(\D)$, i.e.\ $\Phi_F(0)<\infty$. In polar coordinates $z=\xi+se^{i\theta}$ about $\xi$, the region $Q$ is $\{r<s<R,\ \alpha_1(s)<\theta<\alpha_2(s)\}$ with $\alpha_2(s)-\alpha_1(s)=2\arccos(s/2)\in[\pi-s,\pi]$ and $|\alpha_i'(s)|\le1$ for $s\le\frac12$.

For $Q$ take $u=\log(s/r)/\log(R/r)$, so $|\nabla u|=(s\log(R/r))^{-1}$. Then $\operatorname{mod}Q\le\int_Q|\nabla u|^2\,dA\le\pi/\log(R/r)$, and the lower bound follows by comparing with the sector $\{r<s<R,\ |\theta-\theta_0|<\pi/4\}\subset Q$, whose curves joining the cross-cuts form a subfamily. Next,
\[
 \operatorname{mod}F(Q)\le\int_QK|\nabla u|^2\,dA\le\frac{\pi}{\log(R/r)}+\frac1{\log^2(R/r)}\int_Q\frac{K-1}{|z-\xi|^2}\,dA .
\]
Decompose $Q$ into the pieces $Q_j=Q\cap\{2^{-j-1}R\le s<2^{-j}R\}$, $0\le j\le J:=\lceil\log_2(R/r)\rceil$. Each $Q_j$ lies in a Carleson box $S(I_j)$ with $|I_j|\le4\cdot2^{-j}R$, and $s^{-2}\le2^{2j+2}R^{-2}$ on $Q_j$; by Lemma~\ref{carl-properties}(b),
\[
 \int_{Q_j}\frac{K-1}{|z-\xi|^2}\,dA\le2^{2j+2}R^{-2}\cdot C\Lambda_F(4\cdot2^{-j}R)^2=64C\Lambda_F ,
\]
so $\int_Q(K-1)|z-\xi|^{-2}dA\le C'\Lambda_F\log(R/r)$, which gives the bound for $\operatorname{mod}F(Q)$.

For $Q^*$ take $v=(\theta-\alpha_1(s))/(\alpha_2(s)-\alpha_1(s))$ on $Q$, extended by $0$ and $1$ to the two components of $\D\setminus Q$ adjacent to the two arcs; $v$ is Lipschitz and admissible for $Q^*$, and $|\nabla v|^2\le Cs^{-2}+C$ on $Q$ because $|\partial_\theta v|=1/(\alpha_2-\alpha_1)\le1/2$ and $|\partial_sv|\le C$. Hence
\[
 \operatorname{mod}F(Q^*)\le\int_QK|\nabla v|^2\,dA\le C\int_Q\frac{K}{|z-\xi|^2}\,dA+C\int_\D K\,dA ,
\]
and the right side is at most $C(\pi+C'\Lambda_F)\log(R/r)+C(\pi+\Lambda_F)$,
with $\log(R/r)\ge\log2$. Finally $\operatorname{mod}Q^*=1/\operatorname{mod}Q\ge\log(R/r)/\pi$ by \eqref{pc-conjugate}.
\end{proof}

\begin{proposition}\label{carl-extension}
Let $F:\D\to\D$ be a homeomorphism of finite distortion with $\sup_\D\Phi_F<\infty$. Then $F$ extends to a homeomorphism of $\overline\D$, and its boundary map is quasisymmetric, with a constant depending only on $\sup\Phi_F$ after a three-point normalisation.
\end{proposition}
\begin{proof}
Lemma~\ref{carl-half-annulus} says that $F$ distorts the moduli of all half-annuli $Q(\xi;r,R)$ and of their conjugates by at most the factor $K'=C'(1+\sup\Phi_F)$ of \eqref{carl-half-annulus-eq}--\eqref{carl-half-annulus-eq2}. The classical proof that a $K'$-quasiconformal self-map of the disk extends to a homeomorphism of the closed disk with quasisymmetric boundary values uses only these moduli: the modulus of continuity at a boundary point comes from the first inequality (the image of $Q(\xi;r,R)$ has small modulus as $r\to0$, so $F(\{|z-\xi|<r\}\cap\D)$ has small diameter once a three-point normalisation fixes the size of the complementary part), the injectivity of the boundary map from the inequality for the conjugates (an arc collapsed to a point would make the image of a suitable conjugate family have infinite modulus), and the quasisymmetry from the comparison of the two; see \cite[I.8 and II.7]{LehtoVirtanen} and \cite[Chapter~3]{AIM}. We do not repeat the argument. The normalised quasisymmetry constant depends only on $K'$.
\end{proof}

\subsection{The theorem}
For the family $Q_{x,t}$ of Section~\ref{p-classes} we need to dominate the extremal density by the kernels of \eqref{carl-Phi}. Work in $\IH$, where \eqref{carl-Phi} reads $\Phi_F(w)=\int_\IH(K-1)\,k(\cdot,w)\,dA$ with
\[
 k(z,w)=\frac{\eta(z)}{\cosh^4d(z,w)}=\frac{4(\operatorname{Im}w)^2}{\bigl(4\operatorname{Im}z\operatorname{Im}w+|z-w|^2\bigr)^2},
\]
using $\cosh^2d(z,w)=1+|z-w|^2/(4\operatorname{Im}z\operatorname{Im}w)$. Let $Q_{0,1}$ have vertices $v_1,\dots,v_4=-1,-\frac13,\frac13,1$ and $\rho_{Q_{0,1}}^2=|f_0'|^2/(4K^2)$ as in \eqref{pc-rho-eta}, with $|f_0'(\zeta)|^2=9/|(1-\zeta^2)(1-9\zeta^2)|$. This is $\asymp|\zeta-v_j|^{-1}$ near $v_j$, $\asymp|\zeta|^{-4}$ for $|\zeta|\ge2$, and bounded elsewhere. For $|\zeta-v_j|\asymp2^{-i}$ the kernel with $w=v_j+i2^{-i}$ is $\ge c\,2^{2i}$ (both $|\zeta-w|^2$ and $4\operatorname{Im}\zeta\operatorname{Im}w$ are $O(4^{-i})$); for $|\zeta|\ge2$ the kernel with $w=i$ is $\asymp|\zeta|^{-4}$; and on the remaining bounded region $k(\cdot,i)$ is bounded below. Hence, with an absolute constant,
\begin{equation}\label{carl-rho-domination}
 \rho_{Q_{0,1}}^2(\zeta)\le C\Bigl(k(\zeta,i)+\sum_{j=1}^4\sum_{k\ge0}2^{-k}k(\zeta,v_j+i2^{-k})\Bigr)\qquad(\zeta\in\IH).
\end{equation}

\begin{theorem}\label{carl-theorem}
Let $F:\D\to\D$ be a homeomorphism of finite distortion.
\begin{enumerate}[label=\textup{(\alph*)}]
\item If $(K_F-1)\,dA$ is a Carleson measure, equivalently $\sup_\D\Phi_F<\infty$, then $F$ extends to a homeomorphism of $\overline\D$ whose boundary map $g$ is quasisymmetric, with constant depending only on $\sup\Phi_F$ after a three-point normalisation.
\item If moreover $\Phi_F\in L^p(\D,dA_\eta)$ for some $1<p<\infty$, then $g$ is symmetric and $g\in T^p$; in particular $g$ is absolutely continuous with $\log g'\in B^{1/p}_{p,p}(\IS)$.
\end{enumerate}
By Lemma~\ref{carl-properties}(c), the hypothesis of \textup{(b)} implies that of \textup{(a)}, and by Lemma~\ref{carl-properties}(d) it is implied by $\I_p(F)<\infty$.
\end{theorem}
\begin{proof}
(a) is Proposition~\ref{carl-extension}.

(b) Pass to the half-plane as in the proof of Theorem~\ref{pc-main}, so that $g:\IR\to\IR$ is $\Lambda$-quasisymmetric by (a). For $Q_{x,t}=A(Q_{0,1})$ with $A(\zeta)=x+t\zeta$, conformal invariance of $\rho_Q^2\,dA$ and of $k(z,w)\,dA(z)$, together with \eqref{carl-rho-domination}, give
\[
 Z_{Q_{x,t}}=\int_\IH(K\circ A-1)\rho_{Q_{0,1}}^2\,dA\le C\Bigl(\Phi_F(x+ti)+\sum_{j,k}2^{-k}\Phi_F\bigl(x+t(v_j+i2^{-k})\bigr)\Bigr).
\]
For fixed $\omega=a+ib\in\IH$ the substitution $(X,T)=(x+ta,\,tb)$ has $dx\,dt/t^2=b\,dX\,dT/T^2$, so $\|\Phi_F(x+t\omega)\|_{L^p(d\nu)}=b^{1/p}\|\Phi_F\|_{L^p(d\nu)}$, where $d\nu=4\,dA_\eta$ on $\IH$. Minkowski's inequality gives
\[
 \|Z_{Q_{x,t}}\|_{L^p(d\nu)}\le C\Bigl(1+4\sum_{k\ge0}2^{-k}2^{-k/p}\Bigr)\|\Phi_F\|_{L^p(d\nu)}\le9C\|\Phi_F\|_{L^p(d\nu)}<\infty .
\]
Since $\rho_{Q^*}=\rho_Q/\operatorname{mod}Q$, Corollary~\ref{pc-two-sided} gives $|c_{x,t}|\le M_0(1+m_0^{-2})Z_{Q_{x,t}}$, so $A_p(g)<\infty$, and Theorem~\ref{pc-characterisation} gives $g\in T^p$. Symmetry: $\Phi_F\to0$ at the boundary by Lemma~\ref{carl-properties}(c), so $Z_Q\to0$ along quadrilaterals whose normalising M\"obius maps leave compact sets (the sum over $j,k$ is dominated by $2\sup\Phi_F$ outside any finite range of $k$, and the finitely many remaining terms tend to zero); the asymptotic form of the modulus characterisation then gives symmetry as in Proposition~\ref{finite-energy-boundary-moduli}. The Besov statement is Wei--Matsuzaki \cite[Theorem~5.5]{WeiMatsuzakiIntegrable}.
\end{proof}

\begin{example}[The hypothesis of \textup{(b)} is weaker than finite energy]\label{carl-example}
Work in $\IH$. Let $B_k$ be the Euclidean disk of centre $2^{-k}i$ and radius $k^{-1/2}2^{-k-2}$, $k\ge1$, let $S=\bigcup_kB_k$, and let $F$ be a quasiconformal self-map of $\IH$ with $|\mu_F|=\frac13$ on $S$ and $\mu_F=0$ elsewhere, so $K_F=2$ on $S$ and $1$ elsewhere. The hyperbolic area of $B_k$ is $\asymp k^{-1}$, so $E_p(F)=\int e_p(|\mu_F|)\,dA_\eta=\infty$ for every $p$, and $\int|\mu_F|^p\,dA_\eta=\infty$. On the other hand the hyperbolic distance from $B_k$ to a point $w$ at height $2^{-m}$ is at least $c\,|k-m|$ plus the hyperbolic horizontal displacement, so $\Phi_F(w)\lesssim\sum_kk^{-1}e^{-4c|k-m|}\bigl(1+|\operatorname{Re}w|2^m\bigr)^{-4}\lesssim m^{-1}(1+|\operatorname{Re}w|2^m)^{-4}$, and $\int_\IH\Phi_F^p\,d\nu\lesssim\sum_mm^{-p}<\infty$ for every $p>1$. Thus the boundary map of $F$ is in $T^p$ for every $p>1$ by Theorem~\ref{carl-theorem}, although $F$ itself has infinite $p$-energy for every $p$. (Here $F$ is quasiconformal, so the conclusion $g\in T^p$ could also be reached by replacing $\mu_F$ by zero on $S$; the point is only that the hypothesis of \textup{(b)} sees the averages of $K-1$ and not its pointwise size.)
\end{example}

\subsection{Existence from the coefficient alone}
The hypotheses of Theorem~\ref{carl-theorem} are conditions on $K_F$, but they presuppose the homeomorphism $F$. Since $\sup\Phi_F<\infty$ gives only $K\in L^1(\D)$, it does not by itself produce a map; the measurable Riemann mapping theorem needs exponential integrability. By David's theorem in the form of \cite[Chapter~20]{AIM} (see also \cite{David}), if $\mu$ is measurable on $\D$ with $|\mu|<1$ almost everywhere and $e^{\lambda K_\mu}\in L^1(\D)$ for some $\lambda>0$, $K_\mu=(1+|\mu|)/(1-|\mu|)$, then the Beltrami equation $f_{\bar z}=\mu f_z$, with $\mu$ extended by zero, has a homeomorphic solution $f:\IC\to\IC$ in $W^{1,1}_{\loc}$, unique up to postcomposition with a conformal map; $f(\D)$ is a Jordan domain, and composing with a Riemann map gives a finite-distortion homeomorphism $F:\D\to\D$ with $\mu_F=\mu$.

\begin{corollary}\label{carl-corollary}
Let $\mu$ be measurable on $\D$ with $|\mu|<1$ a.e., $e^{\lambda K_\mu}\in L^1(\D)$ for some $\lambda>0$, and $(K_\mu-1)\,dA$ a Carleson measure. Then every finite-distortion homeomorphic solution $F:\D\to\D$ of $F_{\bar z}=\mu F_z$ extends to a homeomorphism of $\overline\D$ with quasisymmetric boundary values. If moreover $\Phi_\mu:=\Phi_F\in L^p(dA_\eta)$ for some $1<p<\infty$, the boundary values lie in $T^p$.
\end{corollary}

The two hypotheses are of different natures and act in different places: exponential integrability is an interior condition (it is the uniform Orlicz--Sobolev estimate behind the $H^1$--$\mathrm{BMO}$ duality for the Jacobian) and is what makes $F$ exist and be a homeomorphism; the Carleson condition is a boundary condition, invisible on compact subsets, and is what controls the boundary map. Neither can replace the other: the known examples of Beltrami coefficients with $K\in L^q_{\loc}$ for all $q<\infty$ and no homeomorphic solution \cite[Chapter~20]{AIM} are local at an interior point and can be taken with $K\equiv1$ near $\IS$, so they satisfy every hypothesis on $\Phi_\mu$; and quasiconformal maps show that no interior condition controls the boundary map beyond quasisymmetry.

\begin{remark}[Monotone solutions]
One may ask whether, in the spirit of Section~\ref{no-hairs}, a monotone finite-distortion solution with $(K-1)\,dA$ Carleson, or with $K-1\in L^1(dA_\eta)$, is automatically a homeomorphism. The relevant interior fact is the theorem of Iwaniec and \v Sver\'ak \cite{IwaniecSverak}: a map in $W^{1,2}_{\loc}$ of a planar domain with $K\in L^1_{\loc}$ is open and discrete. For a monotone map discrete fibres are singletons, so a monotone $W^{1,2}_{\loc}$ solution with $K\in L^1_{\loc}$ is a local homeomorphism, hence a homeomorphism of $\D$ onto its image by Proposition~\ref{hair-homeomorphism}(b) once boundary hairs are excluded, and then Theorem~\ref{carl-theorem} applies. The obstruction is the forward Sobolev regularity: finite energy, hyperbolic integrability and the Carleson condition all concern the inverse map (Section~\ref{conventions}) and give no $W^{1,2}$ bound on $F$ itself. Without $W^{1,2}_{\loc}$ we do not know whether a monotone finite-distortion map with $K\in L^1_{\loc}$ can collapse a continuum; the radial collapses $z\mapsto\rho(|z|)z/|z|$ with $\rho=0$ on $[0,\frac12]$ cannot, since $K\ge r\rho'/\rho$ forces $\int K\,dA\ge2\pi\int\frac{d\rho}{\rho}=\infty$, but this is not an argument.
\end{remark}

\section*{Acknowledgements and author information}
Work of both authors was partially supported by the New Zealand Marsden Fund. C. Yao was partially supported by the Young Scientist Program of the Ministry of Science and Technology of China (No. 2021YFA1002200), the National Natural Science Foundation of China (Nos. 12401096 and 12101362), the Natural Science Foundation of Shandong Province (Nos. ZR2024QA035 and ZR2022YQ01), and the Fundamental Research Funds for the Central Universities.

\medskip
\noindent Gaven Martin: Institute for Advanced Study, Massey University, Auckland, New Zealand.\\
\texttt{g.j.martin@massey.ac.nz}

\medskip
\noindent Cong Yao: Research Center for Mathematics and Interdisciplinary Sciences, Shandong University, Qingdao 266237, and Frontiers Science Center for Nonlinear Expectations, Ministry of Education, P. R. China.\\
\texttt{c.yao@sdu.edu.cn}


\clearpage
\addcontentsline{toc}{section}{References}
\begin{thebibliography}{99}
\bibitem{AB} V. Alberge and M. A. Brakalova, \emph{On smoothness of the elements of some integrable Teichm\"uller spaces}, Math. Rep. \textbf{23} (2021), 95--105. \url{https://arxiv.org/abs/2003.11878}.

\bibitem{AdS} F. Diaf, A. Moriani, R. Sma\"i, G. A. Smith and E. Trebeschi, \emph{Weil--Petersson homeomorphisms, minimal lagrangian diffeomorphisms, and maximal surfaces in anti-de Sitter space}, preprint, 20 April 2026. \url{https://arxiv.org/abs/2604.17804}.

\bibitem{AhlforsLectures} L.~V. Ahlfors, \emph{Lectures on Quasiconformal Mappings}, second edition, University Lecture Series 38, American Mathematical Society, Providence, RI, 2006.

\bibitem{Akutagawa} K. Akutagawa, \emph{Harmonic diffeomorphisms of the hyperbolic plane}, Trans. Amer. Math. Soc. \textbf{342} (1994), 325--342. \url{https://doi.org/10.1090/S0002-9947-1994-1147398-9}.

\bibitem{AIM} K. Astala, T. Iwaniec and G.J. Martin, {\em Elliptic partial differential equations and quasiconformal mappings in the plane}, Princeton Mathematical Series, 48. Princeton University Press, Princeton, NJ, 2009.

\bibitem{B5} M. A. Brakalova, {\em Symmetric properties of p-integrable Teichm\"uller spaces}, Anal. Math.
Phys. {\bf 8} (2018), 541--549.

\bibitem{BBVW} M. Bridgeman, K. Bromberg, F. Vargas Pallete and Y. Wang, \emph{Universal Liouville action as a renormalized volume and its gradient flow}. \url{https://arxiv.org/abs/2311.18767}.

\bibitem{BenoistHulin} Y. Benoist and D. Hulin,
\emph{Harmonic quasi-isometric maps between rank one symmetric spaces},
Ann. of Math. (2) \textbf{185} (2017), 895--917.
\url{https://doi.org/10.4007/annals.2017.185.3.4}.

\bibitem{BeurlingAhlfors} A. Beurling and L.~V. Ahlfors, \emph{The boundary correspondence under quasiconformal mappings}, Acta Math. \textbf{96} (1956), 125--142.

\bibitem{Bish1} C.J. Bishop, {\em Function theoretic characterizations of Weil--Petersson curves}, Rev. Mat. Iberoam., {\bf 38}, (2022), no. 7, 2355--2384

\bibitem{Bish2} C. J. Bishop, \emph{Weil--Petersson curves, $\beta$-numbers, and minimal surfaces}, Ann. of Math. (2) \textbf{202} (2025), 111--188. \url{https://annals.math.princeton.edu/2025/202-1/p02}.

\bibitem{Cui} G. Cui, {\em Integrably asymptotic affine homeomorphisms of the circle and Teichm\"uller
spaces}, Sci. China Ser. A. {\bf 43}, (2000), 267--279.

\bibitem{David} G. David, \emph{Solutions de l'\'equation de Beltrami avec $\|\mu\|_\infty=1$}, Ann. Acad. Sci. Fenn. Ser. A I Math. \textbf{13} (1988), 25--70.

\bibitem{DMM} F. Diaf, V. Markovi\'c and A. Mesbah, \emph{Harmonic extension of Weil--Petersson circle homeomorphisms}, preprint, 9 June 2026. \url{https://arxiv.org/abs/2606.11141}.

\bibitem{FLS} T. Foertsch, A. Lytchak and V. Schroeder,
\emph{Non-positive curvature and the Ptolemy inequality},
Int. Math. Res. Not. (2007), article rnm100.
Author version: \url{https://arxiv.org/abs/0705.1042}.

\bibitem{FVW} S. Fan, F. Viklund and Y. Wang, \emph{On the Loewner energy of a welding homeomorphism}, preprint, 2026. \url{https://arxiv.org/abs/2604.16737}.

\bibitem{Guo} H. Guo, {\em Integrable Teichm\"uller spaces}, Sci. China Ser. A. {\bf 43}, (2000), 47--58.

\bibitem{Gut} V. Gutlyanskii, K.-I. Sakan and T. Sugawa, \emph{On $\mu$-conformal homeomorphisms and boundary correspondence}, Complex Var. Elliptic Equ. \textbf{58} (2013), 947--962. \url{https://arxiv.org/abs/1102.4987}.

\bibitem{Hastings} W.~W. Hastings, \emph{A Carleson measure theorem for Bergman spaces}, Proc. Amer. Math. Soc. \textbf{52} (1975), 237--241.

\bibitem{HenclKoskelaInverse} S. Hencl and P. Koskela,
\emph{Regularity of the inverse of a planar Sobolev homeomorphism},
Arch. Ration. Mech. Anal. \textbf{180} (2006), 75--95.
\url{https://doi.org/10.1007/s00205-005-0394-1}.

\bibitem{HK} S. Hencl and P. Koskela, {Lectures on Mappings of Finite Distortion} Lecture notes on Mathematics 2096, Springer-Verlag 2014.

\bibitem{IOMonotoneApprox} T. Iwaniec and J. Onninen, \emph{Monotone Sobolev mappings of planar domains and surfaces}, Arch. Ration. Mech. Anal. \textbf{219} (2016), 159--181. \url{https://arxiv.org/abs/1505.06439}.

\bibitem{IwaniecSverak} T. Iwaniec and V. \v Sver\'ak, \emph{On mappings with integrable dilatation}, Proc. Amer. Math. Soc. \textbf{118} (1993), 181--188.

\bibitem{Jost} J. Jost, E. Ke\ss ler, J. Tolksdorf, R. Wu and M. Zhu, \emph{From harmonic maps to the nonlinear supersymmetric sigma model of quantum field theory}, Vietnam J. Math. \textbf{47} (2019), 39--67. \url{https://doi.org/10.1007/s10013-018-0298-7}.

\bibitem{JostUnivalent} J. Jost,
\emph{Univalency of harmonic mappings between surfaces},
J. Reine Angew. Math. \textbf{324} (1981), 141--153.

\bibitem{KalajReplacement} D. Kalaj, \emph{Energy-minimal diffeomorphisms between doubly connected Riemann surfaces}, preprint, arXiv:1108.0773. \url{https://arxiv.org/html/1108.0773}. See Section 3 for the local metric harmonic-replacement construction.

\bibitem{KorevaarSchoen} N. J. Korevaar and R. M. Schoen,
\emph{Sobolev spaces and harmonic maps for metric space targets},
Comm. Anal. Geom. \textbf{1} (1993), 561--659.

\bibitem{KoskelaOnninen} P. Koskela and J. Onninen, \emph{Mappings of finite distortion: capacity and modulus inequalities}, J. Reine Angew. Math. \textbf{599} (2006), 1--26.

\bibitem{LehtoVirtanen} O. Lehto and K.~I. Virtanen, \emph{Quasiconformal Mappings in the Plane}, second edition, Grundlehren der mathematischen Wissenschaften 126, Springer, Berlin, 1973.

\bibitem{LemmMarkovic} M. Lemm and V. Markovi\'c,
\emph{Heat flows on hyperbolic spaces},
J. Differential Geom. \textbf{108} (2018), 495--529.
\url{https://doi.org/10.4310/jdg/1519959624}.
Author version: \url{https://arxiv.org/abs/1506.04345}.

\bibitem{LiTamProper} P. Li and L.-F. Tam, \emph{Uniqueness and regularity of proper harmonic maps}, Ann. of Math. (2) \textbf{137} (1993), 167--201. \url{https://doi.org/10.2307/2946622}.

\bibitem{LiTamProperII} P. Li and L.-F. Tam, \emph{Uniqueness and regularity of proper harmonic maps. II}, Indiana Univ. Math. J. \textbf{42} (1993), 591--635. \url{https://doi.org/10.1512/iumj.1993.42.42027}.

\bibitem{Markovic} V. Markovi\'c, \emph{Harmonic maps and the Schoen conjecture}, J. Amer. Math. Soc. \textbf{30} (2017), 799--817. \url{https://doi.org/10.1090/jams/881}.

\bibitem{Martin} G.J. Martin, {\em Harmonic degree 1 maps are diffeomorphisms: Lewy's theorem for curved metrics}, Transactions of the American Mathematical Society, {\bf 368}, (2016), 647--658.

\bibitem{Matsuzaki} K. Matsuzaki, \emph{Integrable Teichm\"uller spaces for analysis on Weil--Petersson curves}, preprint, 2025. \url{https://arxiv.org/abs/2508.20341}.

\bibitem{Mori1956} A. Mori, \emph{On an absolute constant in the theory of quasi-conformal mappings}, J. Math. Soc. Japan \textbf{8} (1956), 156--166. \url{https://doi.org/10.2969/jmsj/00820156}.

\bibitem{MRSY} O. Martio, V. Ryazanov, U. Srebro and E. Yakubov, \emph{Moduli in Modern Mapping Theory}, Springer Monographs in Mathematics, Springer, New York, 2009.

\bibitem{MYDiffeomorphic} G. Martin and C. Yao, \emph{Diffeomorphic solutions of Ahlfors--Hopf equations}, arXiv:2510.19375v2, 8 January 2026. \url{https://arxiv.org/html/2510.19375v2}.

\bibitem{MYLp} G. Martin and C. Yao, \emph{The $L^p$ Teichm\"uller theory: Existence and regularity of critical points}, Arch. Ration. Mech. Anal. \textbf{248} (2024), article 13. \url{https://doi.org/10.1007/s00205-023-01955-9}. The theorem numbering used here follows the preprint \url{https://arxiv.org/abs/2007.15149v1}.

\bibitem{MYRadonRiesz} G. Martin and C. Yao, \emph{Extremal mappings of finite distortion and the Radon--Riesz property}, Rev. Mat. Iberoam. \textbf{38} (2022), no. 7, 2057--2068. \url{https://doi.org/10.4171/RMI/1379}.

\bibitem{NV} S. Nag and A. Verjovsky, {\em $Diff (\IS)$ and the Teichm\"uller spaces},
Communications in Mathematical Physics, {\bf 130}, (1990), 123--138.

\bibitem{PartII} G. Martin and C. Yao, \emph{Integrable distortion II: the $p$-Weil--Petersson variational problem}, preprint (2026).

\bibitem{PW} E. Peltola and Y. Wang, \emph{Large deviations of multichordal SLE$_{0+}$, real rational functions, and zeta-regularized determinants of Laplacians}, J. Eur. Math. Soc. \textbf{26} (2024), 469--535. \url{https://doi.org/10.4171/JEMS/1274}.

\bibitem{RiestenbergSmillie} J. M. Riestenberg and P. Smillie, \emph{Harmonic maps to Hadamard spaces and a universal higher Teichm\"uller space}, arXiv:2511.11469v3 (2026). \url{https://arxiv.org/abs/2511.11469v3}. The interior comparison used here is Proposition~32 and Lemmas~33--34 of that version.

\bibitem{SchoenHarmonic} R. M. Schoen, \emph{The role of harmonic mappings in rigidity and deformation problems}, in \emph{Complex geometry (Osaka, 1990)}, Lecture Notes in Pure and Appl. Math. \textbf{143}, Dekker, New York, 1993, 179--200.

\bibitem{Shen} Y. Shen, \emph{Weil--Petersson Teichm\"uller space}, Amer. J. Math. \textbf{140} (2018), 1041--1074. \url{https://arxiv.org/abs/1304.3197}.

\bibitem{SidlerWenger} H. Sidler and S. Wenger,
\emph{Harmonic quasi-isometric maps into Gromov hyperbolic
${\rm CAT}(0)$-spaces}, preprint (2018).
\url{https://arxiv.org/abs/1804.06286}.

\bibitem{TT} L. A. Takhtajan and L.-P. Teo, \emph{Weil--Petersson metric on the universal Teichm\"uller space}, Mem. Amer. Math. Soc. \textbf{183} (2006), no. 861. Part II preprint: \url{https://arxiv.org/abs/math/0406408}.

\bibitem{Wang} Y. Wang, \emph{Equivalent descriptions of the Loewner energy}, Invent. Math. \textbf{218} (2019), 573--621. \url{https://arxiv.org/abs/1802.01999}.

\bibitem{WanHarmonic} T. Y.-H. Wan, \emph{Constant mean curvature surface, harmonic maps, and universal Teichm\"uller space}, J. Differential Geom. \textbf{35} (1992), 643--657. \url{https://doi.org/10.4310/jdg/1214448260}.

\bibitem{WeiMatsuzakiIntegrable} H. Wei and K. Matsuzaki,
\emph{The $p$-integrable Teichm\"uller space for $p\geq1$},
Proc. Japan Acad. Ser. A Math. Sci. \textbf{99} (2023), 37--42.
\url{https://doi.org/10.3792/pjaa.99.008}.
Theorem numbering follows the revised version
\url{https://arxiv.org/abs/2210.04720v4}, 27 November 2023.

\bibitem{Witten} E. Witten, \emph{Coadjoint orbits of the Virasoro group}, Commun. Math. Phys. \textbf{114} (1988), 1--53. \url{https://doi.org/10.1007/BF01218287}.

\bibitem{WolfMinimal} M. Wolf, \emph{Minimal graphs in $\mathbb H^2\times\IR$ and their projections}, author manuscript, 29 June 2007. \url{https://mwolf40.math.gatech.edu/minimal_graph.pdf}.

\bibitem{Wu} C. Wu, \emph{The cross-ratio distortion of integrably asymptotic affine homeomorphism of unit circle}, Sci. China Math. \textbf{55} (2012), 625--632.

\bibitem{YaoHarmonicCompactness} G. W. Yao, \emph{On $C^\infty$-compactness of quasiconformal harmonic maps on the Poincar\'e disk}, Comput. Methods Funct. Theory \textbf{16} (2016), no. 3, 365--374. \url{https://doi.org/10.1007/s40315-015-0148-5}.

\bibitem{YauSchwarz} S.-T. Yau, \emph{A general Schwarz lemma for K\"ahler manifolds}, Amer. J. Math. \textbf{100} (1978), 197--203. \url{https://doi.org/10.2307/2373880}.
\bibitem{Zhu} K. Zhu, \emph{Operator Theory in Function Spaces}, second edition, Mathematical Surveys and Monographs 138, American Mathematical Society, Providence, RI, 2007.

\end{thebibliography}
\end{document}